\documentclass{amsart}
\usepackage[all]{xy}
\usepackage{tikz}
 \tikzstyle{int}=[circle, draw,fill=black,outer sep=0,minimum size=3pt, inner sep=0]
  \tikzstyle{ext}=[circle, draw=black,outer sep=0,inner sep=1pt]
\usepackage{latexsym}
\usepackage{amssymb}
\usepackage{amsmath}
\usepackage{amsthm}
\usepackage{amscd}
\usepackage{fancyhdr}
\usepackage{fullpage}
\usepackage{mathrsfs}
\input cyracc.def
\xyoption{arc}

\def\id{{\mbox{1 \hskip -8pt 1}}}
\newcommand{\sgn}{{\mathit s  \mathit g\mathit  n}}
 \newcommand{\lon}{\longrightarrow}
 \newcommand{\bu}{\bullet}
 
 \newcommand{\rar}{\rightarrow}

 \newcommand{\End}{{\mathsf E\mathsf n \mathsf d}}
\newcommand{\p}{{\partial}}
\newcommand{\Id}{{\mathrm I\mathrm d}}

 \newcommand{\Z}{{\mathbb Z}}
 \newcommand{\bS}{{\mathbb S}}

 \newcommand{\R}{{\mathbb R}}
 \newcommand{\N}{{\mathbb N}}
 \newcommand{\K}{{\mathbb K}}

\newcommand{\GC}{\mathsf{GC}}

\newcommand{\LB}{{\mathcal L}{\mathit i}{\mathit e} {\mathcal B}}
\newcommand{\LoB}{{\mathcal L}{\mathit i}{\mathit e}\hspace{-0.3mm}^\diamond\hspace{-0.4mm}
{\mathcal B}}
\newcommand{\BV}{{\mathcal B}{\mathcal V}}

\newcommand{\Koz}{{\mbox {\scriptsize !`}}}

\newcommand{\alg}[1]{\mathfrak{#1}}
 \newcommand{\ot}{\otimes}

\newcommand{\sC}{{\mathsf C}}

\newcommand{\sG}{{\mathsf G}}
\newcommand{\sP}{{\mathsf P}}

\newcommand{\Lie}{\mathsf{ Lie}}

\newcommand{\Def}{{\mathsf D\mathsf e\mathsf f }}

 \newcommand{\Beq}{\begin{equation}}
 \newcommand{\Eeq}{\end{equation}}
 \newcommand{\Beqr}{\begin{eqnarray}}
 \newcommand{\Eeqr}{\end{eqnarray}}
 \newcommand{\Beqrn}{\begin{eqnarray*}}
 \newcommand{\Eeqrn}{\end{eqnarray*}}
 \newcommand{\Ba}{\begin{array}}
 \newcommand{\Ea}{\end{array}}
 \newcommand{\Bi}{\begin{itemize}}
 \newcommand{\Ei}{\end{itemize}}
 \newcommand{\Bc}{\begin{center}}
 \newcommand{\Ec}{\end{center}}
 
 \newcommand{\fg}{{\mathfrak g}}

\newcommand{\ft}{{\mathfrak t}}
\newcommand{\fr}{{\mathfrak r}}

 \newcommand{\f}{{\mathcal O}}
 \newcommand{\cA}{{\mathcal A}}
 \newcommand{\cB}{{\mathcal B}}
 \newcommand{\cC}{{\mathcal C}}
 
 \newcommand{\cE}{{\mathcal E}}
 \newcommand{\cF}{{\mathcal F}}
 \newcommand{\cG}{{\mathcal G}}

 \newcommand{\caL}{{\mathcal L}}

 \newcommand{\cP}{{\mathcal P}}
 \newcommand{\cQ}{{\mathcal Q}}
 \newcommand{\cR}{{\mathcal R}}

 \newcommand{\cV}{{\mathcal V}}

 \newcommand{\ga}{\gamma}
 
 \newcommand{\Ga}{\Gamma}

 \newcommand{\la}{\lambda}
 \newcommand{\om}{\omega}

 \newcommand{\Hom}{{\mathrm H\mathrm o\mathrm m}}
 
 \newcommand{\sip}{\smallskip}
 \newcommand{\bip}{\bigskip}
 \newcommand{\mip}{\vspace{2.5mm}}

\newcommand{\hoe}{\mathrm{hoe}}
\newcommand{\Aut}{\mathrm{Aut}}
 \newcommand{\GCor}{\GC^{or}}
 \newcommand{\hGCor}{\widehat{\GC}^{or}}

 \newcommand{\fGCor}{\mathsf{fGC}^{or}}

 \DeclareMathOperator{\Exp}{\mathrm{Exp}}

 \newcommand{\LieBi}{{\caL ie\cB}}
  \newcommand{\hLieBi}{\widehat{\LieBi}}
  \newcommand{\hLoB}{\widehat{\LoB}}
 \newcommand{\coLieBi}{{\caL ie\cB^*}}
 \newcommand{\invcoLieBi}{(\LoB)^*}
 \newcommand{\LieBiP}{\LieBi P}

  \newcommand{\Frob}{{\cF rob}}
  \newcommand{\hFrob}{{\widehat{\Frob}}}
 \newcommand{\invFrob}{\Frob^\diamond}
 \newcommand{\coFrob}{{\cF rob^*}}
 \newcommand{\invcoFrob}{(\cF rob^\diamond)^*}
  \newcommand{\invcoFrobtwo}{(\cF rob^\diamond_2)^*}

 \newcommand{\invLieBi}{\LoB}

 \newcommand{\grt}{\mathfrak{grt}}
 \newcommand{\Der}{\mathrm{Der}}

 \newcommand{\gr}{\mathrm{gr}}
 \newcommand{\vecspan}{\mathrm{span}}

\theoremstyle{plain}
\swapnumbers
\newtheorem{theorem}{Theorem}[subsection]
\newtheorem{corollary}[theorem]{Corollary}

\newtheorem{lemma}[theorem]{Lemma}
\newtheorem{proposition}[theorem]{Proposition}

\newtheorem{prop-def}[theorem]{Proposition-definition}

\newtheorem{main-theorem}{Main~Theorem}[section]
\newtheorem{section-theorem}{Theorem}[section]
\newtheorem{section-corollary}{Corollary}[section]

\theoremstyle{definition}

\newtheorem{remark}[theorem]{Remark}

\renewcommand{\thesection}{{\bf\arabic{section}}}
\renewcommand{\thesubsection}{{\bf\arabic{section}.\arabic{subsection}}}
\renewcommand{\thesubsubsection}{\bf\arabic{section}.\arabic{subsection}.\arabic{subsubsection}}

\begin{document}

\sloppy

 \newenvironment{proo}{\begin{trivlist} \item{\sc {Proof.}}}
  {\hfill $\square$ \end{trivlist}}

\long\def\symbolfootnote[#1]#2{\begingroup%
\def\thefootnote{\fnsymbol{footnote}}\footnote[#1]{#2}\endgroup}

  \title{The Frobenius properad is Koszul}

\author{Ricardo~Campos}
\address{Ricardo~Campos: Institute of Mathematics, University of Zurich, Zurich, Switzerland}
\email{ricardo.campos@math.uzh.ch}

\author{Sergei~Merkulov}
\address{Sergei~Merkulov:  Department of Mathematics, Stockholm University, Sweden and Mathematics Research Unit, Luxembourg University,  Grand Duchy of Luxembourg (present address) }
\email{sergei.merkulov@uni.lu}

\author{Thomas~Willwacher}
\address{Thomas~Willwacher: Institute of Mathematics, University of Zurich, Zurich, Switzerland}
\email{thomas.willwacher@math.uzh.ch}

 \begin{abstract}
  We show Koszulness of the prop governing involutive Lie bialgebras and also of the props governing non-unital and unital-counital Frobenius algebras, solving a long-standing problem. This gives us
  minimal models for  their deformation complexes, and for deformation complexes of their algebras which are discussed in detail.

  Using an operad of graph complexes we prove, with the help of an earlier result of one of the authors \cite{Wi2}, that there is a highly non-trivial action
  of the  Grothendieck-Teichm\"uller group $GRT_1$ on (completed versions of) the minimal models of the properads governing Lie bialgebras and involutive Lie bialgebras by automorphisms.  As a corollary one obtains a large class of universal deformations of any (involutive) Lie bialgebra and any Frobenius algebra, parameterized by elements of the Grothendieck-Teichm\"uller Lie algebra.

We also prove that, for any given homotopy involutive Lie bialgebra structure in a vector space,
 there is an associated homotopy Batalin-Vilkovisky algebra structure on the associated Chevalley-Eilenberg complex.

\end{abstract}
 \maketitle

{\large
\section{\bf Introduction}
}

The notion of Lie bialgebra was introduced by Drinfeld in \cite{D1}  in the context of the
theory of Yang-Baxter equations. Later this notion played a fundamental role  in his theory
of Hopf algebra deformations of universal enveloping algebras, see the book \cite{ES} and
references cited therein.

 \sip

Many interesting examples of Lie bialgebras automatically satisfy an additional algebraic
condition, the so called {\em involutivity}, or  ``diamond" $\diamondsuit$ constraint.
 A remarkable example of such a Lie bialgebra structure was discovered by Turaev \cite{Tu}
 on the vector space generated by all non-trivial free homotopy classes of curves on an
orientable surface. Chas proved  \cite{Ch} that such a structure is in fact always
involutive. This example was generalized to arbitrary manifolds within  the framework of
{\em string topology}: the equivariant homology
of the free loop space of a compact manifold was shown by Chas and Sullivan \cite{ChSu} to carry the
structure of a graded involutive Lie bialgebra. An involutive Lie bialgebra structure was
also found by Cieliebak and Latschev \cite{CL} in the contact homology of an arbitrary exact symplectic manifold, while Schedler \cite{Sch} introduced a natural involutive Lie bialgebra structure
on the necklace Lie algebra associated to a quiver. It is worth pointing out that
the construction of quantum $A_\infty$-algebras given in \cite{Ba1} (see also \cite{Ha})
stems from the fact that the vector space of cyclic words in elements of a graded vector space $W$ equipped with a
(skew)symmetric pairing admits a canonical involutive Lie bialgebra structure.
Therefore, involutive Lie bialgebras appear in many different areas of modern research.


\mip

In the study of the deformation theory of dg involutive Lie bialgebras
one needs to know a minimal resolution
of the associated properad. Such a minimal resolution is particularly nice and explicit if
the properad happens to be {\em Koszul} \cite{Va}. Koszulness of the prop(erad) of Lie bialgebras $\caL ie\cB$ was established by Markl and Voronov \cite{MaVo} following an idea of Kontsevich \cite{Ko}. The proof made use of a new category of  {\em small props}, which are often called
$\frac{1}{2}$-{\em props} nowadays, and a new technical tool, the {\em path
filtration}\, of a dg free properad. Attempts to settle the question of Koszulness or non-Koszulness of the properad of involutive Lie bialgebras, $\caL
ie\cB^\diamondsuit$, have been made since 2004.
The Koszulness proof of $\caL ie\cB$ does not carry over to $\LoB$ since the additional involutivity relation is not $\frac{1}{2}$-properadic in nature.
Motivated by some computer calculations the authors of
\cite{DCTT} conjectured in 2009 that the properad of involutive Lie bialgebras, $\caL
ie\cB^\diamondsuit$, {\em is}\, Koszul. In Section 2 of this paper we settle this long-standing problem.
There are at least two not very straightforward steps in our solution. First, we extend Kontsevich's exact functor from small props to props by twisting it
with the relative simplicial cohomologies of graphs involved. This step allows us to incorporate operations in arities $(1,1)$, $(1,0)$ and $(0,1)$ into the story, which were strictly prohibited in the Kontsevich construction as they fail exactness of his functor. Second, we reduce
the cohomology computation of some important auxiliary dg properad to a computation checking Koszulness of some ordinary quadratic algebra, which might be of independent interest.

\bip

By Koszul duality theory of properads \cite{Va}, our result implies immediately that the properad of non-unital Frobenius algebras is Koszul. By the curved Koszul duality theory \cite{HM}, the latter result implies, after some extra work, the Koszulness  of the prop of unital-counital Frobenius algebras. These Frobenius properads also admit many applications in various areas of mathematics and mathematical physics, e.~g.\ in representation theory, algebraic geometry, combinatorics, and recently, in 2-dimensional topological quantum field theory.

\bip

Another main result of this paper is a construction of a highly non-trivial action of the
Grothendieck-Teichm\"uller group $GRT_1$ {\cite{D2}} on minimal models of the operads of involutive Lie bialgebras/Frobenius algebras, and hence on the sets of homotopy involutive Lie bialgebra/Frobenius structures
on an arbitrary dg vector space $\fg$. The Grothendieck-Teichm\"uller group $GRT_1$ has recently been shown to include a pro-unipotent subgrop freely generated by an infinite number of generators \cite{Br}, hence our construction provides a rich class of universal symmetries of the aforementioned objects.

 \mip

 In \S 5 of this paper we establish a surprising link between the theory of involutive Lie bialgebras and the operad of framed little disks whose cohomology was proven by Getzler \cite{Ge} to be the operad $\cB\cV$ of Batalin-Vilkovisky algebras, an important structure
 in mathematical physics and algebra.   We prove that, for any homotopy involutive Lie bialgebra structure on a complex $\fg$,
 there is an associated  $\cB\cV^{com}_\infty$ algebra  (and hence
  a $\cB\cV_\infty$ algebra) structure on the associated Chevalley-Eilenberg complex
  $\odot^\bu (\fg[-1])$. The operad $\cB\cV_\infty^{com}$  (introduced in \cite{Kr}) occurs naturally not only in our story, but also in the construction
  of the Koszul resolution of the operad $\cB\cV$  in \cite{GTV}, and in its recent applications  to (Poisson) geometry. We furthermore prove that the cohomology
 of the dg operad $\cB\cV^{com}_\infty$  is isomorphic to the operad $\cB\cV$.

\mip

\subsection*{Acknowledgements}
We are grateful to B. Vallette for helpful discussions.
R.C. and T.W. acknowledge partial support by the Swiss National Science Foundaton, grant 200021\_150012.
Research of T.W. was supported in part by the NCCR SwissMAP of the Swiss National Science Foundation.
S.M.\ is grateful to the Max Planck Institute for Mathematics in Bonn for hospitality and excellent working conditions.

\mip

{\bf Some notation}. In this paper $\mathbb K$ denotes a field of characteristic $0$. The set $\{1,2, \ldots, n\}$ is abbreviated to $[n]$. Its group of automorphisms is
denoted by $\bS_n$. The sign representation of $\bS_n$ is denoted by $\sgn_n$. The
cardinality of a finite set
$A$ is denoted by $\# A$. If $V=\oplus_{i\in \Z} V^i$ is a graded vector space, then
$V[k]$ stands for the graded vector space with $V[k]^i:=V^{i+k}$. For $v\in V^i$ we set $|v|:=i$.
The phrase \emph{differential graded} is abbreviated by dg.
The $n$-fold symmetric product of a (dg) vector space $V$ is denoted  by $\odot^n V$, the full symmetric product space by  $\odot^\bullet V$ or just $\odot V$ and the completed symmetric product by $\hat \odot^\bullet V$.
For a finite group G acting on a vector space $V$, we
denote via $V^G$ the space of invariants with respect to the action of G, and by $V_G$
the space of coinvariants $V_G = V/\{gv- v| v\in V, g\in G\}$. We always work over a field $\K$ of characteristic zero so that, for finite $G$, we have a canonical isomorphism $V_G\cong V^G$.

We  use freely the language of operads and properads and their Koszul duality theory. For a background on operads we refer to the textbook \cite{LV}, while the Koszul duality theory of properads has been developed in \cite{Va}.
For a properad $\cP$ we denote by $\cP\{k\}$ the unique properad which has the following property:
for any graded vector space $V$ there is a one-to-one correspondence between representations of
$\cP\{k\}$ in $V$ and representations of
$\cP$ in $V[-k]$; in particular, $\cE nd_V\{k\}=\cE nd_{V[k]}$.
For $\cC$ a coaugmented co(pr)operad, we will denote by $\Omega(\cC)$ its cobar construction.
Concretely, $\Omega(\cC)=\cF ree\langle\bar \cC[1]\rangle$ as a graded (pr)operad where $\bar \cC$ the cokernel of the coaugmetation and $\cF ree\langle\dots\rangle$ denotes the free (pr)operad generated by an $\bS$-(bi)module.
We will often complexes of derivations of (pr)operads and deformation complexes of (pr)operad maps.
The notation $\Der(\cP)$ shall be used for the complex of derivations of $\cP$, considered as non-unital properad.
For a map of properads $f: \Omega(\cC){\to} \cP$, we will denote by
\[
\Def( \Omega(\cC)\stackrel{f}{\to} \cP )\cong \prod_{n,m} \Hom_{\bS_n\times \bS_m}(\cC(n,m), \cP(n,m))
\]
the associated convolution complex.

\bip

{\large
\section{\bf Koszulness of the prop of involutive Lie bialgebras}
}

\subsection{Involutive Lie bialgebras} A  {\em Lie bialgebra}\, is a graded vector space
$\fg$,
equipped with degree zero linear maps,
$$
\vartriangle: \fg\rightarrow \fg\wedge \fg \ \ \ \mbox{and}\ \ \  [\ , \ ]: \wedge^2 \fg
\rightarrow \fg,
$$
such that
\Bi
\item the data $(\fg,\vartriangle)$ is a Lie coalgebra;
\item the data $(\fg, [\ ,\ ])$ is a Lie algebra;
\item the compatibility condition,
$$
\vartriangle [a, b] = \sum a_1\otimes [a_2, b] +  [a,
b_1]\otimes b_2 + (-1)^{|a||b|}( [b, a_1]\otimes a_2
+ b_1\otimes [b_2, a]),
$$
holds for any $a,b\in \fg$. Here $\vartriangle a=:\sum a_1\otimes a_2$, $\vartriangle
b=:\sum
b_1\otimes b_2$.
\Ei
 A Lie bialgebra $(\fg, [\ ,\ ], \vartriangle)$ is called {\em
involutive}\, if the composition map
$$
\Ba{ccccc}
V & \stackrel{\vartriangle}{\lon} & \Lambda^2V & \stackrel{[\ ,\ ]}{\lon} & V\\
a & \lon &    \sum a'\otimes a'' &\lon & \sum [a',a'']
\Ea
$$
vanishes. A dg (involutive) Lie bialgebra is a
complex $(\fg,d)$ equipped with the structure of an (involutive) Lie bialgebra such that
the maps $[\ ,\ ]$ and $\Delta$ are morphisms of complexes.

\subsubsection{\bf An example}\label{2: subsection on cyclic words}
Let $W$ be a finite dimensional graded vector space over a field $\K$ of
characteristic zero equipped with a degree $0$
skewsymmetric pairing,
$$
\Ba{rccl}
\om: & W\ot W & \lon & \K \\
  &  w_1\ot w_2 & \lon & \om(w_1,w_2)=-(-1)^{|w_1||w_2|}\om(w_2,w_1).
\Ea
$$
 Then the associated
vector space of ``cyclic words in $W$",
$$
Cyc^\bu(W):=\bigoplus_{n\geq 0} (W^{\ot n})_{\Z_n},
$$
admits an involutive Lie bialgebra structure  given by (see, e.g., \cite{Go} and references cited there)
$$
[(w_1\ot...\ot w_n)_{\Z_n}, (v_1\ot ...\ot v_m)_{\Z_n}]:=\hspace{-2mm} \sum_{i\in[n]\atop
j\in [m]}  \pm
\om(w_i,w_j) (w_1\ot ...\ot  w_{i-1}\ot v_{j+1}\ot ... \ot v_m\ot v_1\ot ... \ot v_{j-1}\ot w_{i+1}\ot\ldots\ot w_n)_{\Z_{n+m-2}}
$$
and
\Beqrn
\vartriangle (w_1\ot\ldots\ot w_n)_{\Z_n}:&=&\sum_{i\neq j}
\pm\om(w_i,w_j)(w_{i+1}\ot ...\ot w_{j-1})_{\Z_{j-i-1}}\bigotimes
(w_{j+1}\ot ...\ot w_{i-1})_{\Z_{n-j+i-1}}    \\
\Eeqrn

This example has many applications in various  areas of modern research (see, e.g., \cite{Ba1,Ch, CL, Ha}).


\mip

\subsection{Properad of involutive Lie bialgebras}
By definition, the properad, $\LoB$, of involutive Lie bialgebras is a quadratic properad
given as the quotient,
$$
\LoB:=\cF ree\langle E\rangle/<\cR>,
$$
of the free properad generated by an  $\bS$-bimodule $E=\{E(m,n)\}_{m,n\geq 1}$ with
 all $E(m,n)=0$ except
$$
E(2,1):=\id_1\ot \sgn_2=\mbox{span}\left\langle
\begin{xy}
 <0mm,-0.55mm>*{};<0mm,-2.5mm>*{}**@{-},
 <0.5mm,0.5mm>*{};<2.2mm,2.2mm>*{}**@{-},
 <-0.48mm,0.48mm>*{};<-2.2mm,2.2mm>*{}**@{-},
 <0mm,0mm>*{\circ};<0mm,0mm>*{}**@{},
 <0mm,-0.55mm>*{};<0mm,-3.8mm>*{_1}**@{},
 <0.5mm,0.5mm>*{};<2.7mm,2.8mm>*{^2}**@{},
 <-0.48mm,0.48mm>*{};<-2.7mm,2.8mm>*{^1}**@{},
 \end{xy}
=-
\begin{xy}
 <0mm,-0.55mm>*{};<0mm,-2.5mm>*{}**@{-},
 <0.5mm,0.5mm>*{};<2.2mm,2.2mm>*{}**@{-},
 <-0.48mm,0.48mm>*{};<-2.2mm,2.2mm>*{}**@{-},
 <0mm,0mm>*{\circ};<0mm,0mm>*{}**@{},
 <0mm,-0.55mm>*{};<0mm,-3.8mm>*{_1}**@{},
 <0.5mm,0.5mm>*{};<2.7mm,2.8mm>*{^1}**@{},
 <-0.48mm,0.48mm>*{};<-2.7mm,2.8mm>*{^2}**@{},
 \end{xy}
   \right\rangle
$$
$$
E(1,2):= \sgn_2\ot \id_1=\mbox{span}\left\langle
\begin{xy}
 <0mm,0.66mm>*{};<0mm,3mm>*{}**@{-},
 <0.39mm,-0.39mm>*{};<2.2mm,-2.2mm>*{}**@{-},
 <-0.35mm,-0.35mm>*{};<-2.2mm,-2.2mm>*{}**@{-},
 <0mm,0mm>*{\circ};<0mm,0mm>*{}**@{},
   <0mm,0.66mm>*{};<0mm,3.4mm>*{^1}**@{},
   <0.39mm,-0.39mm>*{};<2.9mm,-4mm>*{^2}**@{},
   <-0.35mm,-0.35mm>*{};<-2.8mm,-4mm>*{^1}**@{},
\end{xy}=-
\begin{xy}
 <0mm,0.66mm>*{};<0mm,3mm>*{}**@{-},
 <0.39mm,-0.39mm>*{};<2.2mm,-2.2mm>*{}**@{-},
 <-0.35mm,-0.35mm>*{};<-2.2mm,-2.2mm>*{}**@{-},
 <0mm,0mm>*{\circ};<0mm,0mm>*{}**@{},
   <0mm,0.66mm>*{};<0mm,3.4mm>*{^1}**@{},
   <0.39mm,-0.39mm>*{};<2.9mm,-4mm>*{^1}**@{},
   <-0.35mm,-0.35mm>*{};<-2.8mm,-4mm>*{^2}**@{},
\end{xy}
\right\rangle
$$
modulo the ideal generated by the following relations
\Beq\label{R for LieB}
\cR:\left\{
\Ba{c}
\begin{xy}
 <0mm,0mm>*{\circ};<0mm,0mm>*{}**@{},
 <0mm,-0.49mm>*{};<0mm,-3.0mm>*{}**@{-},
 <0.49mm,0.49mm>*{};<1.9mm,1.9mm>*{}**@{-},
 <-0.5mm,0.5mm>*{};<-1.9mm,1.9mm>*{}**@{-},
 <-2.3mm,2.3mm>*{\circ};<-2.3mm,2.3mm>*{}**@{},
 <-1.8mm,2.8mm>*{};<0mm,4.9mm>*{}**@{-},
 <-2.8mm,2.9mm>*{};<-4.6mm,4.9mm>*{}**@{-},
   <0.49mm,0.49mm>*{};<2.7mm,2.3mm>*{^3}**@{},
   <-1.8mm,2.8mm>*{};<0.4mm,5.3mm>*{^2}**@{},
   <-2.8mm,2.9mm>*{};<-5.1mm,5.3mm>*{^1}**@{},
 \end{xy}
\ + \
\begin{xy}
 <0mm,0mm>*{\circ};<0mm,0mm>*{}**@{},
 <0mm,-0.49mm>*{};<0mm,-3.0mm>*{}**@{-},
 <0.49mm,0.49mm>*{};<1.9mm,1.9mm>*{}**@{-},
 <-0.5mm,0.5mm>*{};<-1.9mm,1.9mm>*{}**@{-},
 <-2.3mm,2.3mm>*{\circ};<-2.3mm,2.3mm>*{}**@{},
 <-1.8mm,2.8mm>*{};<0mm,4.9mm>*{}**@{-},
 <-2.8mm,2.9mm>*{};<-4.6mm,4.9mm>*{}**@{-},
   <0.49mm,0.49mm>*{};<2.7mm,2.3mm>*{^2}**@{},
   <-1.8mm,2.8mm>*{};<0.4mm,5.3mm>*{^1}**@{},
   <-2.8mm,2.9mm>*{};<-5.1mm,5.3mm>*{^3}**@{},
 \end{xy}
\ + \
\begin{xy}
 <0mm,0mm>*{\circ};<0mm,0mm>*{}**@{},
 <0mm,-0.49mm>*{};<0mm,-3.0mm>*{}**@{-},
 <0.49mm,0.49mm>*{};<1.9mm,1.9mm>*{}**@{-},
 <-0.5mm,0.5mm>*{};<-1.9mm,1.9mm>*{}**@{-},
 <-2.3mm,2.3mm>*{\circ};<-2.3mm,2.3mm>*{}**@{},
 <-1.8mm,2.8mm>*{};<0mm,4.9mm>*{}**@{-},
 <-2.8mm,2.9mm>*{};<-4.6mm,4.9mm>*{}**@{-},
   <0.49mm,0.49mm>*{};<2.7mm,2.3mm>*{^1}**@{},
   <-1.8mm,2.8mm>*{};<0.4mm,5.3mm>*{^3}**@{},
   <-2.8mm,2.9mm>*{};<-5.1mm,5.3mm>*{^2}**@{},
 \end{xy}\ =\ 0,
 \vspace{3mm}\\
 \begin{xy}
 <0mm,0mm>*{\circ};<0mm,0mm>*{}**@{},
 <0mm,0.69mm>*{};<0mm,3.0mm>*{}**@{-},
 <0.39mm,-0.39mm>*{};<2.4mm,-2.4mm>*{}**@{-},
 <-0.35mm,-0.35mm>*{};<-1.9mm,-1.9mm>*{}**@{-},
 <-2.4mm,-2.4mm>*{\circ};<-2.4mm,-2.4mm>*{}**@{},
 <-2.0mm,-2.8mm>*{};<0mm,-4.9mm>*{}**@{-},
 <-2.8mm,-2.9mm>*{};<-4.7mm,-4.9mm>*{}**@{-},
    <0.39mm,-0.39mm>*{};<3.3mm,-4.0mm>*{^3}**@{},
    <-2.0mm,-2.8mm>*{};<0.5mm,-6.7mm>*{^2}**@{},
    <-2.8mm,-2.9mm>*{};<-5.2mm,-6.7mm>*{^1}**@{},
 \end{xy}
\ + \
 \begin{xy}
 <0mm,0mm>*{\circ};<0mm,0mm>*{}**@{},
 <0mm,0.69mm>*{};<0mm,3.0mm>*{}**@{-},
 <0.39mm,-0.39mm>*{};<2.4mm,-2.4mm>*{}**@{-},
 <-0.35mm,-0.35mm>*{};<-1.9mm,-1.9mm>*{}**@{-},
 <-2.4mm,-2.4mm>*{\circ};<-2.4mm,-2.4mm>*{}**@{},
 <-2.0mm,-2.8mm>*{};<0mm,-4.9mm>*{}**@{-},
 <-2.8mm,-2.9mm>*{};<-4.7mm,-4.9mm>*{}**@{-},
    <0.39mm,-0.39mm>*{};<3.3mm,-4.0mm>*{^2}**@{},
    <-2.0mm,-2.8mm>*{};<0.5mm,-6.7mm>*{^1}**@{},
    <-2.8mm,-2.9mm>*{};<-5.2mm,-6.7mm>*{^3}**@{},
 \end{xy}
\ + \
 \begin{xy}
 <0mm,0mm>*{\circ};<0mm,0mm>*{}**@{},
 <0mm,0.69mm>*{};<0mm,3.0mm>*{}**@{-},
 <0.39mm,-0.39mm>*{};<2.4mm,-2.4mm>*{}**@{-},
 <-0.35mm,-0.35mm>*{};<-1.9mm,-1.9mm>*{}**@{-},
 <-2.4mm,-2.4mm>*{\circ};<-2.4mm,-2.4mm>*{}**@{},
 <-2.0mm,-2.8mm>*{};<0mm,-4.9mm>*{}**@{-},
 <-2.8mm,-2.9mm>*{};<-4.7mm,-4.9mm>*{}**@{-},
    <0.39mm,-0.39mm>*{};<3.3mm,-4.0mm>*{^1}**@{},
    <-2.0mm,-2.8mm>*{};<0.5mm,-6.7mm>*{^3}**@{},
    <-2.8mm,-2.9mm>*{};<-5.2mm,-6.7mm>*{^2}**@{},
 \end{xy}\ =\ 0,
 \\
 \begin{xy}
 <0mm,2.47mm>*{};<0mm,0.12mm>*{}**@{-},
 <0.5mm,3.5mm>*{};<2.2mm,5.2mm>*{}**@{-},
 <-0.48mm,3.48mm>*{};<-2.2mm,5.2mm>*{}**@{-},
 <0mm,3mm>*{\circ};<0mm,3mm>*{}**@{},
  <0mm,-0.8mm>*{\circ};<0mm,-0.8mm>*{}**@{},
<-0.39mm,-1.2mm>*{};<-2.2mm,-3.5mm>*{}**@{-},
 <0.39mm,-1.2mm>*{};<2.2mm,-3.5mm>*{}**@{-},
     <0.5mm,3.5mm>*{};<2.8mm,5.7mm>*{^2}**@{},
     <-0.48mm,3.48mm>*{};<-2.8mm,5.7mm>*{^1}**@{},
   <0mm,-0.8mm>*{};<-2.7mm,-5.2mm>*{^1}**@{},
   <0mm,-0.8mm>*{};<2.7mm,-5.2mm>*{^2}**@{},
\end{xy}
\  - \
\begin{xy}
 <0mm,-1.3mm>*{};<0mm,-3.5mm>*{}**@{-},
 <0.38mm,-0.2mm>*{};<2.0mm,2.0mm>*{}**@{-},
 <-0.38mm,-0.2mm>*{};<-2.2mm,2.2mm>*{}**@{-},
<0mm,-0.8mm>*{\circ};<0mm,0.8mm>*{}**@{},
 <2.4mm,2.4mm>*{\circ};<2.4mm,2.4mm>*{}**@{},
 <2.77mm,2.0mm>*{};<4.4mm,-0.8mm>*{}**@{-},
 <2.4mm,3mm>*{};<2.4mm,5.2mm>*{}**@{-},
     <0mm,-1.3mm>*{};<0mm,-5.3mm>*{^1}**@{},
     <2.5mm,2.3mm>*{};<5.1mm,-2.6mm>*{^2}**@{},
    <2.4mm,2.5mm>*{};<2.4mm,5.7mm>*{^2}**@{},
    <-0.38mm,-0.2mm>*{};<-2.8mm,2.5mm>*{^1}**@{},
    \end{xy}
\  +\
\begin{xy}
 <0mm,-1.3mm>*{};<0mm,-3.5mm>*{}**@{-},
 <0.38mm,-0.2mm>*{};<2.0mm,2.0mm>*{}**@{-},
 <-0.38mm,-0.2mm>*{};<-2.2mm,2.2mm>*{}**@{-},
<0mm,-0.8mm>*{\circ};<0mm,0.8mm>*{}**@{},
 <2.4mm,2.4mm>*{\circ};<2.4mm,2.4mm>*{}**@{},
 <2.77mm,2.0mm>*{};<4.4mm,-0.8mm>*{}**@{-},
 <2.4mm,3mm>*{};<2.4mm,5.2mm>*{}**@{-},
     <0mm,-1.3mm>*{};<0mm,-5.3mm>*{^2}**@{},
     <2.5mm,2.3mm>*{};<5.1mm,-2.6mm>*{^1}**@{},
    <2.4mm,2.5mm>*{};<2.4mm,5.7mm>*{^2}**@{},
    <-0.38mm,-0.2mm>*{};<-2.8mm,2.5mm>*{^1}**@{},
    \end{xy}
\  - \
\begin{xy}
 <0mm,-1.3mm>*{};<0mm,-3.5mm>*{}**@{-},
 <0.38mm,-0.2mm>*{};<2.0mm,2.0mm>*{}**@{-},
 <-0.38mm,-0.2mm>*{};<-2.2mm,2.2mm>*{}**@{-},
<0mm,-0.8mm>*{\circ};<0mm,0.8mm>*{}**@{},
 <2.4mm,2.4mm>*{\circ};<2.4mm,2.4mm>*{}**@{},
 <2.77mm,2.0mm>*{};<4.4mm,-0.8mm>*{}**@{-},
 <2.4mm,3mm>*{};<2.4mm,5.2mm>*{}**@{-},
     <0mm,-1.3mm>*{};<0mm,-5.3mm>*{^2}**@{},
     <2.5mm,2.3mm>*{};<5.1mm,-2.6mm>*{^1}**@{},
    <2.4mm,2.5mm>*{};<2.4mm,5.7mm>*{^1}**@{},
    <-0.38mm,-0.2mm>*{};<-2.8mm,2.5mm>*{^2}**@{},
    \end{xy}
\ + \
\begin{xy}
 <0mm,-1.3mm>*{};<0mm,-3.5mm>*{}**@{-},
 <0.38mm,-0.2mm>*{};<2.0mm,2.0mm>*{}**@{-},
 <-0.38mm,-0.2mm>*{};<-2.2mm,2.2mm>*{}**@{-},
<0mm,-0.8mm>*{\circ};<0mm,0.8mm>*{}**@{},
 <2.4mm,2.4mm>*{\circ};<2.4mm,2.4mm>*{}**@{},
 <2.77mm,2.0mm>*{};<4.4mm,-0.8mm>*{}**@{-},
 <2.4mm,3mm>*{};<2.4mm,5.2mm>*{}**@{-},
     <0mm,-1.3mm>*{};<0mm,-5.3mm>*{^1}**@{},
     <2.5mm,2.3mm>*{};<5.1mm,-2.6mm>*{^2}**@{},
    <2.4mm,2.5mm>*{};<2.4mm,5.7mm>*{^1}**@{},
    <-0.38mm,-0.2mm>*{};<-2.8mm,2.5mm>*{^2}**@{},
    \end{xy}\ =\ 0,\\
\Ba{c}
\xy
 (0,0)*{\circ}="a",
(0,6)*{\circ}="b",
(3,3)*{}="c",
(-3,3)*{}="d",
 (0,9)*{}="b'",
(0,-3)*{}="a'",
\ar@{-} "a";"c" <0pt>
\ar @{-} "a";"d" <0pt>
\ar @{-} "a";"a'" <0pt>
\ar @{-} "b";"c" <0pt>
\ar @{-} "b";"d" <0pt>
\ar @{-} "b";"b'" <0pt>
\endxy
\Ea\ =\ 0.
\Ea
\right.
\Eeq

The properad governing Lie bialgebras $\LieBi$ is defined in the same manner, except that the last relation of \eqref{R for LieB} is omitted.

\sip

Recall \cite{Va} that any quadratic properad $\cP$ has an associated
Koszul dual coproperad $\cP^\Koz$ such that its cobar construction,
$$
\Omega(\cP^\Koz)=\cF ree\langle\bar\cP^\Koz[1]\rangle,
$$
comes equipped with a differential $d$ and with a canonical surjective map of dg
properads,
$$
(\Omega(\cP^\Koz), d) \lon (\cP, 0).
$$
This map always induces an isomorphism in cohomology in degree 0. If, additionally, the map is a quasi-isomorphism, then the properad $\cP$ is called {\em Koszul}. In
this
case the cobar construction $\Omega(\cP^\Koz)$ gives us a minimal resolution of $\cP$
and is denoted by $\cP_\infty$. It is well known, for example, that the properad governing Lie bialgebras $\LieBi$ is Koszul \cite{MaVo}.

\sip

We shall study below the Koszul dual properad $\LoB^\Koz$, its
cobar construction $\Omega(\LoB^\Koz)$ and prove that the natural surjection
 $\Omega(\LoB^\Koz)\rar \LoB$ is a quasi-isomorphism. Anticipating this conclusion, we
 often use the symbol $\LoB_\infty$ as a shorthand for  $\Omega(\LoB^\Koz)$.

\subsection{An explicit description of the dg properad $\LoB_\infty$}\label{sec:explicit Koszul dual of Lob}
The Koszul dual of $\LoB$ is a coproperad $\LoB^\Koz$ whose linear dual, $(\LoB^\Koz)^*$,
 is the properad generated by  degree $1$ corollas,
$$
 \begin{xy}
 <0mm,-0.55mm>*{};<0mm,-2.5mm>*{}**@{-},
 <0.5mm,0.5mm>*{};<2.2mm,2.2mm>*{}**@{-},
 <-0.48mm,0.48mm>*{};<-2.2mm,2.2mm>*{}**@{-},
 <0mm,0mm>*{\bu};<0mm,0mm>*{}**@{},
 <0mm,-0.55mm>*{};<0mm,-3.8mm>*{_1}**@{},
 <0.5mm,0.5mm>*{};<2.7mm,2.8mm>*{^2}**@{},
 <-0.48mm,0.48mm>*{};<-2.7mm,2.8mm>*{^1}**@{},
 \end{xy}
=-
\begin{xy}
 <0mm,-0.55mm>*{};<0mm,-2.5mm>*{}**@{-},
 <0.5mm,0.5mm>*{};<2.2mm,2.2mm>*{}**@{-},
 <-0.48mm,0.48mm>*{};<-2.2mm,2.2mm>*{}**@{-},
 <0mm,0mm>*{\bu};<0mm,0mm>*{}**@{},
 <0mm,-0.55mm>*{};<0mm,-3.8mm>*{_1}**@{},
 <0.5mm,0.5mm>*{};<2.7mm,2.8mm>*{^1}**@{},
 <-0.48mm,0.48mm>*{};<-2.7mm,2.8mm>*{^2}**@{},
 \end{xy}\ \ \ , \ \ \
 \begin{xy}
 <0mm,0.66mm>*{};<0mm,3mm>*{}**@{-},
 <0.39mm,-0.39mm>*{};<2.2mm,-2.2mm>*{}**@{-},
 <-0.35mm,-0.35mm>*{};<-2.2mm,-2.2mm>*{}**@{-},
 <0mm,0mm>*{\bu};<0mm,0mm>*{}**@{},
   <0mm,0.66mm>*{};<0mm,3.4mm>*{^1}**@{},
   <0.39mm,-0.39mm>*{};<2.9mm,-4mm>*{^2}**@{},
   <-0.35mm,-0.35mm>*{};<-2.8mm,-4mm>*{^1}**@{},
\end{xy}=-
\begin{xy}
 <0mm,0.66mm>*{};<0mm,3mm>*{}**@{-},
 <0.39mm,-0.39mm>*{};<2.2mm,-2.2mm>*{}**@{-},
 <-0.35mm,-0.35mm>*{};<-2.2mm,-2.2mm>*{}**@{-},
 <0mm,0mm>*{\bu};<0mm,0mm>*{}**@{},
   <0mm,0.66mm>*{};<0mm,3.4mm>*{^1}**@{},
   <0.39mm,-0.39mm>*{};<2.9mm,-4mm>*{^1}**@{},
   <-0.35mm,-0.35mm>*{};<-2.8mm,-4mm>*{^2}**@{},
\end{xy}
$$
with the following relations,
$$
\Ba{c}
\begin{xy}
 <0mm,0mm>*{\bu};<0mm,0mm>*{}**@{},
 <0mm,-0.49mm>*{};<0mm,-3.0mm>*{}**@{-},
 <0.49mm,0.49mm>*{};<1.9mm,1.9mm>*{}**@{-},
 <-0.5mm,0.5mm>*{};<-1.9mm,1.9mm>*{}**@{-},
 <-2.3mm,2.3mm>*{\bu};<-2.3mm,2.3mm>*{}**@{},
 <-1.8mm,2.8mm>*{};<0mm,4.9mm>*{}**@{-},
 <-2.8mm,2.9mm>*{};<-4.6mm,4.9mm>*{}**@{-},
   <0.49mm,0.49mm>*{};<2.7mm,2.3mm>*{^3}**@{},
   <-1.8mm,2.8mm>*{};<0.4mm,5.3mm>*{^2}**@{},
   <-2.8mm,2.9mm>*{};<-5.1mm,5.3mm>*{^1}**@{},
 \end{xy}\Ea
\ =- \
\Ba{c}
\begin{xy}
 <0mm,0mm>*{\bu};<0mm,0mm>*{}**@{},
 <0mm,-0.49mm>*{};<0mm,-3.0mm>*{}**@{-},
 <0.49mm,0.49mm>*{};<1.9mm,1.9mm>*{}**@{-},
 <-0.5mm,0.5mm>*{};<-1.9mm,1.9mm>*{}**@{-},
 <2.3mm,2.3mm>*{\bu};<-2.3mm,2.3mm>*{}**@{},
 <1.8mm,2.8mm>*{};<0mm,4.9mm>*{}**@{-},
 <2.8mm,2.9mm>*{};<4.6mm,4.9mm>*{}**@{-},
   <0.49mm,0.49mm>*{};<-2.7mm,2.3mm>*{^1}**@{},
   <-1.8mm,2.8mm>*{};<0mm,5.3mm>*{^2}**@{},
   <-2.8mm,2.9mm>*{};<5.1mm,5.3mm>*{^3}**@{},
 \end{xy}\Ea, \ \ \ \ \
 \Ba{c}\begin{xy}
 <0mm,0mm>*{\bu};<0mm,0mm>*{}**@{},
 <0mm,0.69mm>*{};<0mm,3.0mm>*{}**@{-},
 <0.39mm,-0.39mm>*{};<2.4mm,-2.4mm>*{}**@{-},
 <-0.35mm,-0.35mm>*{};<-1.9mm,-1.9mm>*{}**@{-},
 <-2.4mm,-2.4mm>*{\bu};<-2.4mm,-2.4mm>*{}**@{},
 <-2.0mm,-2.8mm>*{};<0mm,-4.9mm>*{}**@{-},
 <-2.8mm,-2.9mm>*{};<-4.7mm,-4.9mm>*{}**@{-},
    <0.39mm,-0.39mm>*{};<3.3mm,-4.0mm>*{^3}**@{},
    <-2.0mm,-2.8mm>*{};<0.5mm,-6.7mm>*{^2}**@{},
    <-2.8mm,-2.9mm>*{};<-5.2mm,-6.7mm>*{^1}**@{},
 \end{xy}\Ea
\ =- \
 \Ba{c}\begin{xy}
 <0mm,0mm>*{\bu};<0mm,0mm>*{}**@{},
 <0mm,0.69mm>*{};<0mm,3.0mm>*{}**@{-},
 <0.39mm,-0.39mm>*{};<2.4mm,-2.4mm>*{}**@{-},
 <-0.35mm,-0.35mm>*{};<-1.9mm,-1.9mm>*{}**@{-},
 <2.4mm,-2.4mm>*{\bu};<-2.4mm,-2.4mm>*{}**@{},
 <2.0mm,-2.8mm>*{};<0mm,-4.9mm>*{}**@{-},
 <2.8mm,-2.9mm>*{};<4.7mm,-4.9mm>*{}**@{-},
    <0.39mm,-0.39mm>*{};<-3mm,-4.0mm>*{^1}**@{},
    <-2.0mm,-2.8mm>*{};<0mm,-6.7mm>*{^2}**@{},
    <-2.8mm,-2.9mm>*{};<5.2mm,-6.7mm>*{^3}**@{},
 \end{xy}\Ea,\ \ \ \ \ \
 \begin{xy}
 <0mm,2.47mm>*{};<0mm,0.12mm>*{}**@{-},
 <0.5mm,3.5mm>*{};<2.2mm,5.2mm>*{}**@{-},
 <-0.48mm,3.48mm>*{};<-2.2mm,5.2mm>*{}**@{-},
 <0mm,3mm>*{\bu};<0mm,3mm>*{}**@{},
  <0mm,-0.8mm>*{\bu};<0mm,-0.8mm>*{}**@{},
<-0.39mm,-1.2mm>*{};<-2.2mm,-3.5mm>*{}**@{-},
 <0.39mm,-1.2mm>*{};<2.2mm,-3.5mm>*{}**@{-},
     <0.5mm,3.5mm>*{};<2.8mm,5.7mm>*{^2}**@{},
     <-0.48mm,3.48mm>*{};<-2.8mm,5.7mm>*{^1}**@{},
   <0mm,-0.8mm>*{};<-2.7mm,-5.2mm>*{^1}**@{},
   <0mm,-0.8mm>*{};<2.7mm,-5.2mm>*{^2}**@{},
\end{xy}
\  =- \
\begin{xy}
 <0mm,-1.3mm>*{};<0mm,-3.5mm>*{}**@{-},
 <0.38mm,-0.2mm>*{};<2.0mm,2.0mm>*{}**@{-},
 <-0.38mm,-0.2mm>*{};<-2.2mm,2.2mm>*{}**@{-},
<0mm,-0.8mm>*{\bu};<0mm,0.8mm>*{}**@{},
 <2.4mm,2.4mm>*{\bu};<2.4mm,2.4mm>*{}**@{},
 <2.77mm,2.0mm>*{};<4.4mm,-0.8mm>*{}**@{-},
 <2.4mm,3mm>*{};<2.4mm,5.2mm>*{}**@{-},
     <0mm,-1.3mm>*{};<0mm,-5.3mm>*{^1}**@{},
     <2.5mm,2.3mm>*{};<5.1mm,-2.6mm>*{^2}**@{},
    <2.4mm,2.5mm>*{};<2.4mm,5.7mm>*{^2}**@{},
    <-0.38mm,-0.2mm>*{};<-2.8mm,2.5mm>*{^1}**@{},
    \end{xy}.
$$
Hence the following graphs
\Beq\label{2: basis of inv Frob}
\resizebox{10mm}{!}{
\xy
(0,-6)*{};
(0,0)*+{a}*\cir{}
**\dir{-};
(0,6)*{};
(0,0)*+{a}*\cir{}
**\dir{-};
(0,6)*-{\bu}="1",
(-3,9)*-{\bu}="2",
(3,9)*{}="2'",
(-6,12)*{}="3",
(-0,12)*{}="3'",
(-6.7,12.7)*{\cdot};
(-7.7,13.7)*{\cdot};
(-9,15)*-{\bu}="4",
(-12,18)*{}="4'",
(-6,18)*{}="4''",
(0,-6)*-{\bu}="-1",
(-3,-9)*-{\bu}="-2",
(3,-9)*{}="-2'",
(-6,-12)*{}="-3",
(-0,-12)*{}="-3'",
(-6.7,-12.7)*{\cdot};
(-7.7,-13.7)*{\cdot};
(-9,-15)*-{\bu}="-4",
(-12,-18)*{}="-4'",
(-6,-18)*{}="-4''",
\ar @{-} "1";"2" <0pt>
\ar @{-} "1";"2'" <0pt>
\ar @{-} "2";"3" <0pt>
\ar @{-} "2";"3'" <0pt>
\ar @{-} "4";"4'" <0pt>
\ar @{-} "4";"4''" <0pt>
\ar @{-} "-1";"-2" <0pt>
\ar @{-} "-1";"-2'" <0pt>
\ar @{-} "-2";"-3" <0pt>
\ar @{-} "-2";"-3'" <0pt>
\ar @{-} "-4";"-4'" <0pt>
\ar @{-} "-4";"-4''" <0pt>
\endxy
}
\Eeq
where
\Beq\label{2: a weight in LieB-Koszul}
\resizebox{3.1mm}{!}{\xy
(0,-5)*{};
(0,0)*+{_a}*\cir{}
**\dir{-};
(0,5)*{};
(0,0)*+{_a}*\cir{}
**\dir{-};
\endxy}:=\Ba{c}\resizebox{4mm}{!}{\xy
 (0,0)*-{\bu}="a",
(0,6)*-{\bu}="b",
(3,3)*{}="c",
(-3,3)*{}="d",
 (0,9)*{}="b'",
(0,-3)*{}="a'",
\ar @{-} "a";"c" <0pt>
\ar @{-} "a";"d" <0pt>
\ar @{-} "a";"a'" <0pt>
\ar @{-} "b";"c" <0pt>
\ar @{-} "b";"d" <0pt>
\ar @{-} "b";"b'" <0pt>
\endxy}\\
\cdot\vspace{-2mm}\\
\cdot\\
\resizebox{4mm}{!}{\xy
 (0,0)*-{\bu}="a",
(0,6)*-{\bu}="b",
(3,3)*{}="c",
(-3,3)*{}="d",
 (0,9)*{}="b'",
(0,-3)*{}="a'",
\ar @{-} "a";"c" <0pt>
\ar @{-} "a";"d" <0pt>
\ar @{-} "a";"a'" <0pt>
\ar @{-} "b";"c" <0pt>
\ar @{-} "b";"d" <0pt>
\ar @{-} "b";"b'" <0pt>
\endxy}
\Ea
\Eeq
is the composition of $a=0,1,2,\dots$ graphs of the form $\resizebox{4mm}{!}{\xy
 (0,0)*-{\bu}="a",
(0,6)*-{\bu}="b",
(3,3)*{}="c",
(-3,3)*{}="d",
 (0,9)*{}="b'",
(0,-3)*{}="a'",
\ar @{-} "a";"c" <0pt>
\ar @{-} "a";"d" <0pt>
\ar @{-} "a";"a'" <0pt>
\ar @{-} "b";"c" <0pt>
\ar @{-} "b";"d" <0pt>
\ar @{-} "b";"b'" <0pt>
\endxy}$\, ,
form  a basis of $(\LoB^\Koz)^*$.
If the graph (\ref{2: basis of inv Frob}) has $n$ input legs and $m$ output legs,
then it has $m+n+2a-2$ vertices and its degree is equal to
$m+n+2a-2$. Hence the properad
$\Omega((\LoB)^\Koz)=\cF ree\langle \overline{(\LoB)^\Koz}[1]\rangle$ is a free properad generated
by the following skewsymmetric corollas of degree $3-m-n-2a$,
\Beq\label{2: generating corollas of LoB infty}
\resizebox{15mm}{!}{
\xy
(-9,-6)*{};
(0,0)*+{a}*\cir{}
**\dir{-};
(-5,-6)*{};
(0,0)*+{a}*\cir{}
**\dir{-};
(9,-6)*{};
(0,0)*+{a}*\cir{}
**\dir{-};
(5,-6)*{};
(0,0)*+{a}*\cir{}
**\dir{-};
(0,-6)*{\ldots};
(-10,-8)*{_1};
(-6,-8)*{_2};
(10,-8)*{_n};
(-9,6)*{};
(0,0)*+{a}*\cir{}
**\dir{-};
(-5,6)*{};
(0,0)*+{a}*\cir{}
**\dir{-};
(9,6)*{};
(0,0)*+{a}*\cir{}
**\dir{-};
(5,6)*{};
(0,0)*+{a}*\cir{}
**\dir{-};
(0,6)*{\ldots};
(-10,8)*{_1};
(-6,8)*{_2};
(10,8)*{_m};
\endxy}
=(-1)^{\sigma+\tau}
\resizebox{18mm}{!}{
\xy
(-9,-6)*{};
(0,0)*+{a}*\cir{}
**\dir{-};
(-5,-6)*{};
(0,0)*+{a}*\cir{}
**\dir{-};
(9,-6)*{};
(0,0)*+{a}*\cir{}
**\dir{-};
(5,-6)*{};
(0,0)*+{a}*\cir{}
**\dir{-};
(0,-6)*{\ldots};
(-12,-8)*{_{\tau(1)}};
(-6,-8)*{_{\tau(2)}};
(12,-8)*{_{\tau(n)}};
(-9,6)*{};
(0,0)*+{a}*\cir{}
**\dir{-};
(-5,6)*{};
(0,0)*+{a}*\cir{}
**\dir{-};
(9,6)*{};
(0,0)*+{a}*\cir{}
**\dir{-};
(5,6)*{};
(0,0)*+{a}*\cir{}
**\dir{-};
(0,6)*{\ldots};
(-12,8)*{_{\sigma(1)}};
(-6,8)*{_{\sigma(2)}};
(12,8)*{_{\sigma(m)}};
\endxy}\ \ \ \forall \sigma\in \bS_m, \forall \tau\in \bS_n,
\Eeq
where $m+n+ a\geq 3$, $m\geq 1$, $n\geq 1$, $a\geq 0$. The non-negative number $a$ is
called the {\em weight}\, of the generating corolla (\ref{2: generating corollas of LoB
infty}). The differential in
$\Omega((\LoB)^\Koz)$ is given by\footnote{The precise sign factors in this formula can be
determined via a usual trick: the analogous differential in the degree shifted properad
$\Omega((\LoB)^\Koz)\{1\}$ must be given by the same formula but with all sign factors equal to
$+1$.}
\Beq\label{2: d on Lie inv infty}
\delta\Ba{c}
\resizebox{15mm}{!}{\xy
(-9,-6)*{};
(0,0)*+{a}*\cir{}
**\dir{-};
(-5,-6)*{};
(0,0)*+{a}*\cir{}
**\dir{-};
(9,-6)*{};
(0,0)*+{a}*\cir{}
**\dir{-};
(5,-6)*{};
(0,0)*+{a}*\cir{}
**\dir{-};
(0,-6)*{\ldots};
(-10,-8)*{_1};
(-6,-8)*{_2};
(10,-8)*{_n};
(-9,6)*{};
(0,0)*+{a}*\cir{}
**\dir{-};
(-5,6)*{};
(0,0)*+{a}*\cir{}
**\dir{-};
(9,6)*{};
(0,0)*+{a}*\cir{}
**\dir{-};
(5,6)*{};
(0,0)*+{a}*\cir{}
**\dir{-};
(0,6)*{\ldots};
(-10,8)*{_1};
(-6,8)*{_2};
(10,8)*{_m};
\endxy}\Ea=
\sum_{a=b+c+l-1}\sum_{[m]=I_1\sqcup I_2\atop
[n]=J_1\sqcup J_2} \pm
\Ba{c}
%
%
\resizebox{20mm}{!}{\xy
(0,0)*+{b}*\cir{}="b",
(10,10)*+{c}*\cir{}="c",
%
(-9,6)*{}="1",
(-7,6)*{}="2",
(-2,6)*{}="3",
(-3.5,5)*{...},
(-4,-6)*{}="-1",
(-2,-6)*{}="-2",
(4,-6)*{}="-3",
(1,-5)*{...},
(0,-8)*{\underbrace{\ \ \ \ \ \ \ \ }},
(0,-11)*{_{J_1}},
(-6,8)*{\overbrace{ \ \ \ \ \ \ }},
(-6,11)*{_{I_1}},
(6,16)*{}="1'",
(8,16)*{}="2'",
(14,16)*{}="3'",
(11,15)*{...},
(11,6)*{}="-1'",
(16,6)*{}="-2'",
(18,6)*{}="-3'",
(13.5,6)*{...},
(15,4)*{\underbrace{\ \ \ \ \ \ \ }},
(15,1)*{_{J_2}},
(10,18)*{\overbrace{ \ \ \ \ \ \ \ \ }},
(10,21)*{_{I_2}},
%
(0,2)*-{};(8.0,10.0)*-{}
**\crv{(0,10)};
(0.5,1.8)*-{};(8.5,9.0)*-{}
**\crv{(0.4,7)};
(1.5,0.5)*-{};(9.1,8.5)*-{}
**\crv{(5,1)};
(1.7,0.0)*-{};(9.5,8.6)*-{}
**\crv{(6,-1)};
(5,5)*+{...};
\ar @{-} "b";"1" <0pt>
\ar @{-} "b";"2" <0pt>
\ar @{-} "b";"3" <0pt>
\ar @{-} "b";"-1" <0pt>
\ar @{-} "b";"-2" <0pt>
\ar @{-} "b";"-3" <0pt>
\ar @{-} "c";"1'" <0pt>
\ar @{-} "c";"2'" <0pt>
\ar @{-} "c";"3'" <0pt>
\ar @{-} "c";"-1'" <0pt>
\ar @{-} "c";"-2'" <0pt>
\ar @{-} "c";"-3'" <0pt>
\endxy}
\Ea
\Eeq
where  the parameter $l$ counts the number of internal edges connecting
the two vertices
on the right-hand side. We have, in particular,
$$
\delta\Ba{c}
\resizebox{5mm}{!}{\xy
(-3,-4)*{};
(0,0)*+{_0}*\cir{}
**\dir{-};
(3,-4)*{};
(0,0)*+{_0}*\cir{}
**\dir{-};
(0,5)*{};
(0,0)*+{_0}*\cir{}
**\dir{-};
\endxy}\Ea
=0, \ \ \ \ \ \delta\Ba{c} \resizebox{5mm}{!}{\xy
(-3,4)*{};
(0,0)*+{_0}*\cir{}
**\dir{-};
(3,4)*{};
(0,0)*+{_0}*\cir{}
**\dir{-};
(0,-5)*{};
(0,0)*+{_0}*\cir{}
**\dir{-};
\endxy}\Ea=0,\ \ \ \
\delta\Ba{c} \resizebox{3mm}{!}{\xy
(0,5)*{};
(0,0)*+{_1}*\cir{}
**\dir{-};
(0,-5)*{};
(0,0)*+{_1}*\cir{}
**\dir{-};
\endxy}\Ea=
\Ba{c}\resizebox{4mm}{!}{\xy
(-3,0)*{};
(0,4)*+{_0}*\cir{}
**\dir{-};
(3,0)*{};
(0,4)*+{_0}*\cir{}
**\dir{-};
(0,9)*{};
(0,4)*+{_0}*\cir{}
**\dir{-};
(-3,0)*{};
(0,-4)*+{_0}*\cir{}
**\dir{-};
(3,0)*{};
(0,-4)*+{_0}*\cir{}
**\dir{-};
(0,-9)*{};
(0,-4)*+{_0}*\cir{}
**\dir{-};
\endxy}\Ea
$$
so that the map
\Beq\label{2: pi quasi-iso}
\pi: \LoB_\infty \lon \LoB,
\Eeq
which sends to zero all generators of $\LoB_\infty$ except the following ones,
$$
\pi\left(\Ba{c}\resizebox{5mm}{!}{ \xy
(-3,-4)*{};
(0,0)*+{_0}*\cir{}
**\dir{-};
(3,-4)*{};
(0,0)*+{_0}*\cir{}
**\dir{-};
(0,5)*{};
(0,0)*+{_0}*\cir{}
**\dir{-};
\endxy}\Ea\right)
=\Ba{c}\resizebox{3.6mm}{!}{\begin{xy}
 <0mm,0.66mm>*{};<0mm,4mm>*{}**@{-},
 <0.39mm,-0.39mm>*{};<2.2mm,-3.2mm>*{}**@{-},
 <-0.35mm,-0.35mm>*{};<-2.2mm,-3.2mm>*{}**@{-},
 <0mm,0mm>*{\circ};<0mm,0mm>*{}**@{},
\end{xy}}\Ea\ ,\ \ \ \ \ \
\pi\left(\Ba{c}\resizebox{5mm}{!}{ \xy
(-3,4)*{};
(0,0)*+{_0}*\cir{}
**\dir{-};
(3,4)*{};
(0,0)*+{_0}*\cir{}
**\dir{-};
(0,-5)*{};
(0,0)*+{_0}*\cir{}
**\dir{-};
\endxy}\Ea\right)
=\Ba{c}\resizebox{3.6mm}{!}{\begin{xy}
 <0mm,-0.66mm>*{};<0mm,-4mm>*{}**@{-},
 <0.4mm,0.4mm>*{};<2.2mm,3.2mm>*{}**@{-},
 <-0.4mm,0.4mm>*{};<-2.2mm,3.2mm>*{}**@{-},
 <0mm,-0.1mm>*{\circ};<0mm,0mm>*{}**@{},
\end{xy}}\Ea\ ,
$$
is a morphism of dg properads, as expected. To prove that the properad $\LoB$ is Koszul is
equivalent to showing that the map $\pi$ is a quasi-isomorphism. As $\LoB_\infty$ is non-negatively
graded, the map $\pi$ is a quasi-isomorphism if and only if the cohomology of the dg
properad
$\LoB_\infty$ is concentrated in degree zero. We shall prove this property below in \S
{\ref{2: Theorem on Koszulness}}
with the help of several auxiliary constructions which we discuss next.

\subsection{A decomposition of the complex  $\LoB_\infty$}\label{sec:decomposition}
As a vector space the properad  $\LoB_\infty=\{\LoB_\infty(m,n)\}_{m,n\geq 1}$ is spanned
by oriented graphs built from corollas (\ref{2: generating corollas of LoB infty}). For
such a
graph $\Ga \in  \LoB_\infty$ we set
$$
||\Ga||:= g(\Ga)+
w(\Ga)\in \N
$$
where  $g(\Ga)$ is its genus and  $w(\Ga)$ is its total weight defined as the sum of
weights of its vertices(-corollas). It is obvious that the differential $\delta$ in
$\LoB_\infty$ respects this total grading,
$$
|| \delta \Ga||=||\Ga||.
$$
Therefore each complex $(\LoB_\infty(m,n),\delta)$ decomposes into a direct sum of
subcomplexes,
$$
\LoB_\infty(m,n)=\sum_{s\geq 0}\LoB_\infty(m,n)^{(s)}
$$
where
$$
\LoB_\infty(m,n)^{(s)}:=\mbox{span}\langle\left\{  \Ga : \mbox{graphs}\ \Ga\in
\LoB_\infty(m,n)\ \mbox{with}\
||\Ga||=s\right\}\rangle
$$
\subsubsection{\bf Lemma}
{\em For any fixed  $m,n\geq 1$ and $s\geq 0$ the subcomplex $\LoB_\infty(m,n)^{(s)}$ is
finite-dimensional}.
\begin{proof}
The number of bivalent vertices in every graph $\Ga$ with $||\Ga||=s$ is finite.
As the genus of the graph $\Ga$ is also finite, it must have a finite number of vertices of
valence
$\geq 3$ as well.
\end{proof}

 This lemma guarantees convergence of all spectral sequences which we consider below in the
 context of computing cohomology of $\LoB_\infty$  and which, for general dg free
 properads, can be ill-behaved.

\subsection{An auxiliary graph complex}\label{2 subsect on aux graph comlxes} Let us
consider a graph complex,
$$
C =\bigoplus_{n\geq 1} C^n.
$$
where  $C^n$ is spanned by graphs of the form,
$\Ba{c}\resizebox{27mm}{!}{\xy
(-8,2)*+{};
(0,2)*+{_{a_1}}*\cir{}
**\dir{-};
(10,2)*+{_{a_2}}*\cir{}
**\dir{-};
(20,2)*+{...}
**\dir{-};
(30,2)*+{_{a_{n}}}*\cir{};
**\dir{-};
(31.5,2)*+{};(37,2)
**\dir{-};
\endxy}\Ea$\ , \ \ with $a_1,\ldots, a_n\in \N$.
The differential is given on the generators of the graphs (viewed as
elements of a $\frac{1}{2}$-prop) by
$$
d\Ba{c}\resizebox{8.0
mm}{!}{\xy(0,1)*+{_{a}}*\cir{};
(5,1)
**\dir{-};(0,1)*+{_{a}}*\cir{};
(-5,1)
**\dir{-};
\endxy}\Ea = \sum_{a=b+c\atop b\geq 1, c\geq 1}\Ba{c}\resizebox{15mm}{!}{\xy(0,1)*+{_{b}}*\cir{};
(8,1)*+{_{c}}*\cir{}
**\dir{-};
(0,1)*+{_{b}}*\cir{};
(-5,1)**\dir{-};
(8,1)*+{_{c}}*\cir{};
(13,1)**\dir{-};
\endxy}\Ea.
$$

\subsubsection{\bf Proposition}\label{2: propos on aux graph complexes} {\em One has}
$H^\bu(C)= \mbox{span} \langle \xy(0,0)*+{_{1}}*\cir{};
(5,0)
**\dir{-};(0,0)*+{_1}*\cir{};
(-5,0)
**\dir{-};
\endxy   \rangle$.

\begin{proof} It is well-known that the cohomology of the cobar construction
  $\Omega(T^c(V))$ of the free tensor coalgebra $T^c(V)$ generated by any vector space $V$
  over a field $\K$ equals to $\K\oplus V$, so that the cohomology of the reduced cobar
  construction, $\overline{\Omega}(T^c(V))$, equals $V$. The complex $C$ is
  isomorphic to  $\overline{\Omega}(T^c(V))$ for a one-dimensional vector space $V$ via the
  following identification
  $$
\xy(0,0)*+{_{a}}*\cir{};
\endxy \cong  V^{\ot a}.
  $$
  Hence the claim.
\end{proof}

\subsection{An auxiliary dg properad}\label{2: subsection on P} Let $\cP$ be a dg properad
 generated  generated by a degree $-1$ corolla
$\begin{xy}
 <0mm,-0.55mm>*{};<0mm,-3mm>*{}**@{-},
 <0mm,0.5mm>*{};<0mm,3mm>*{}**@{-},
 <0mm,0mm>*{\bu};<0mm,0mm>*{}**@{},
 \end{xy}$\, and degree zero corollas,
$
\Ba{c}
 \begin{xy}
 <0mm,-0.55mm>*{};<0mm,-2.5mm>*{}**@{-},
 <0.5mm,0.5mm>*{};<2.2mm,2.2mm>*{}**@{-},
 <-0.48mm,0.48mm>*{};<-2.2mm,2.2mm>*{}**@{-},
 <0mm,0mm>*{\circ};<0mm,0mm>*{}**@{},
 <0.5mm,0.5mm>*{};<2.7mm,2.8mm>*{^2}**@{},
 <-0.48mm,0.48mm>*{};<-2.7mm,2.8mm>*{^1}**@{},
 \end{xy}
=-
\begin{xy}
 <0mm,-0.55mm>*{};<0mm,-2.5mm>*{}**@{-},
 <0.5mm,0.5mm>*{};<2.2mm,2.2mm>*{}**@{-},
 <-0.48mm,0.48mm>*{};<-2.2mm,2.2mm>*{}**@{-},
 <0mm,0mm>*{\circ};<0mm,0mm>*{}**@{},
 <0.5mm,0.5mm>*{};<2.7mm,2.8mm>*{^1}**@{},
 <-0.48mm,0.48mm>*{};<-2.7mm,2.8mm>*{^2}**@{},
 \end{xy}\Ea
 $ and $\Ba{c}\begin{xy}
 <0mm,0.66mm>*{};<0mm,3mm>*{}**@{-},
 <0.39mm,-0.39mm>*{};<2.2mm,-2.2mm>*{}**@{-},
 <-0.35mm,-0.35mm>*{};<-2.2mm,-2.2mm>*{}**@{-},
 <0mm,0mm>*{\circ};<0mm,0mm>*{}**@{},
   <0.39mm,-0.39mm>*{};<2.9mm,-4mm>*{^2}**@{},
   <-0.35mm,-0.35mm>*{};<-2.8mm,-4mm>*{^1}**@{},
\end{xy}=-
\begin{xy}
 <0mm,0.66mm>*{};<0mm,3mm>*{}**@{-},
 <0.39mm,-0.39mm>*{};<2.2mm,-2.2mm>*{}**@{-},
 <-0.35mm,-0.35mm>*{};<-2.2mm,-2.2mm>*{}**@{-},
 <0mm,0mm>*{\circ};<0mm,0mm>*{}**@{},
   <0.39mm,-0.39mm>*{};<2.9mm,-4mm>*{^1}**@{},
   <-0.35mm,-0.35mm>*{};<-2.8mm,-4mm>*{^2}**@{},
\end{xy}\Ea$,
 modulo relations \Beq\label{2: relation in P}
\xy
(0,-1.9)*{\bu}="0",
 (0,1.9)*{\bu}="1",
(0,-5)*{}="d",
(0,5)*{}="u",
\ar @{-} "0";"u" <0pt>
\ar @{-} "0";"1" <0pt>
\ar @{-} "1";"d" <0pt>
\endxy=0, \ \ \ \
\xy
(0,-1.9)*{\circ}="0",
 (0,1.9)*{\bu}="1",
(-2.5,-5)*{}="d1",
(2.5,-5)*{}="d2",
(0,5)*{}="u",
\ar @{-} "1";"u" <0pt>
\ar @{-} "0";"1" <0pt>
\ar @{-} "0";"d1" <0pt>
\ar @{-} "0";"d2" <0pt>
\endxy  -
\xy
(-2.5,-1.9)*{\bu}="0",
 (0,1.9)*{\circ}="1",
(-2.5,-1.9)*{}="d1",
(2.5,-1.9)*{}="d2",
(-2.5,-5)*{}="d",
(0,5)*{}="u",
\ar @{-} "1";"u" <0pt>
\ar @{-} "0";"1" <0pt>
\ar @{-} "0";"d" <0pt>
\ar @{-} "1";"d1" <0pt>
\ar @{-} "1";"d2" <0pt>
\endxy
  -
\xy
(2.5,-1.9)*{\bu}="0",
 (0,1.9)*{\circ}="1",
(-2.5,-1.9)*{}="d1",
(2.5,-1.9)*{}="d2",
(2.5,-5)*{}="d",
(0,5)*{}="u",
\ar @{-} "1";"u" <0pt>
\ar @{-} "0";"1" <0pt>
\ar @{-} "0";"d" <0pt>
\ar @{-} "1";"d1" <0pt>
\ar @{-} "1";"d2" <0pt>
\endxy
=0\ , \ \ \ \
\xy
(0,1.9)*{\circ}="0",
 (0,-1.9)*{\bu}="1",
(-2.5,5)*{}="d1",
(2.5,5)*{}="d2",
(0,-5)*{}="u",
\ar @{-} "1";"u" <0pt>
\ar @{-} "0";"1" <0pt>
\ar @{-} "0";"d1" <0pt>
\ar @{-} "0";"d2" <0pt>
\endxy  -
\xy
(-2.5,1.9)*{\bu}="0",
 (0,-1.9)*{\circ}="1",
(-2.5,1.9)*{}="d1",
(2.5,1.9)*{}="d2",
(-2.5,5)*{}="d",
(0,-5)*{}="u",
\ar @{-} "1";"u" <0pt>
\ar @{-} "0";"1" <0pt>
\ar @{-} "0";"d" <0pt>
\ar @{-} "1";"d1" <0pt>
\ar @{-} "1";"d2" <0pt>
\endxy
  -
\xy
(2.5,1.9)*{\bu}="0",
 (0,-1.9)*{\circ}="1",
(-2.5,1.9)*{}="d1",
(2.5,1.9)*{}="d2",
(2.5,5)*{}="d",
(0,-5)*{}="u",
\ar @{-} "1";"u" <0pt>
\ar @{-} "0";"1" <0pt>
\ar @{-} "0";"d" <0pt>
\ar @{-} "1";"d1" <0pt>
\ar @{-} "1";"d2" <0pt>
\endxy
=0.
\Eeq and the first three relations in (\ref{R for
 LieB}).
 The differential in $\cP$ is given on the generators by
\Beq\label{2: differential in P}
d\, \begin{xy}
 <0mm,0.66mm>*{};<0mm,3mm>*{}**@{-},
 <0.39mm,-0.39mm>*{};<2.2mm,-2.2mm>*{}**@{-},
 <-0.35mm,-0.35mm>*{};<-2.2mm,-2.2mm>*{}**@{-},
 <0mm,0mm>*{\circ};<0mm,0mm>*{}**@{},
\end{xy}=0, \ \ \ \
d\, \begin{xy}
 <0mm,-0.55mm>*{};<0mm,-2.5mm>*{}**@{-},
 <0.5mm,0.5mm>*{};<2.2mm,2.2mm>*{}**@{-},
 <-0.48mm,0.48mm>*{};<-2.2mm,2.2mm>*{}**@{-},
 <0mm,0mm>*{\circ};<0mm,0mm>*{}**@{},
 \end{xy} =0, \ \ \ \
d\, \begin{xy}
 <0mm,-0.55mm>*{};<0mm,-3mm>*{}**@{-},
 <0mm,0.5mm>*{};<0mm,3mm>*{}**@{-},
 <0mm,0mm>*{\bu};<0mm,0mm>*{}**@{},
 \end{xy}=\Ba{c}\xy
 (0,0)*{\circ}="a",
(0,6)*{\circ}="b",
(3,3)*{}="c",
(-3,3)*{}="d",
 (0,9)*{}="b'",
(0,-3)*{}="a'",
\ar@{-} "a";"c" <0pt>
\ar @{-} "a";"d" <0pt>
\ar @{-} "a";"a'" <0pt>
\ar @{-} "b";"c" <0pt>
\ar @{-} "b";"d" <0pt>
\ar @{-} "b";"b'" <0pt>
\endxy\Ea.
\Eeq
\begin{theorem}\label{2: proposition on nu quasi-iso} { The surjective morphism of
dg properads,
\Beq\label{2: nu quasi-iso to P}
\nu: \LoB_\infty \lon \cP,
\Eeq
which sends all generators to zero except for the following ones
\Beq\label{2: map from P to P}
\nu\left( \xy
(-3,-4)*{};
(0,0)*+{_0}*\cir{}
**\dir{-};
(3,-4)*{};
(0,0)*+{_0}*\cir{}
**\dir{-};
(0,5)*{};
(0,0)*+{_0}*\cir{}
**\dir{-};
\endxy\right)
=\begin{xy}
 <0mm,0.66mm>*{};<0mm,4mm>*{}**@{-},
 <0.39mm,-0.39mm>*{};<2.2mm,-3.2mm>*{}**@{-},
 <-0.35mm,-0.35mm>*{};<-2.2mm,-3.2mm>*{}**@{-},
 <0mm,0mm>*{\circ};<0mm,0mm>*{}**@{},
\end{xy}\ ,\ \ \ \ \ \
\nu\left( \xy
(-3,4)*{};
(0,0)*+{_0}*\cir{}
**\dir{-};
(3,4)*{};
(0,0)*+{_0}*\cir{}
**\dir{-};
(0,-5)*{};
(0,0)*+{_0}*\cir{}
**\dir{-};
\endxy\right)
=\begin{xy}
 <0mm,-0.66mm>*{};<0mm,-4mm>*{}**@{-},
 <0.4mm,0.4mm>*{};<2.2mm,3.2mm>*{}**@{-},
 <-0.4mm,0.4mm>*{};<-2.2mm,3.2mm>*{}**@{-},
 <0mm,-0.1mm>*{\circ};<0mm,0mm>*{}**@{},
\end{xy}\ \ , \ \
\nu\left(
\xy
(0,5)*{};
(0,0)*+{_1}*\cir{}
**\dir{-};
(0,-5)*{};
(0,0)*+{_1}*\cir{}
**\dir{-};
\endxy\right)=\begin{xy}
 <0mm,-0.55mm>*{};<0mm,-3mm>*{}**@{-},
 <0mm,0.5mm>*{};<0mm,3mm>*{}**@{-},
 <0mm,0mm>*{\bu};<0mm,0mm>*{}**@{},
 \end{xy}
\Eeq
is a quasi-isomorphism.}
\end{theorem}

\begin{proof} The argument is based on several converging spectral sequences.
\mip

{\em Step 1: An exact functor}.

We define the following functor:
$$
F : \mbox{category of dg}\ \frac{1}{2}\mbox{-props} \lon \mbox{category of dg properads},
$$
by
$$
F(s)(m,n) = \bigoplus_{\Gamma \in \overline{\text{Gr}}(m,n)} \left(\bigotimes_{v\in v(\Gamma)} s(Out(v),In(v)) \otimes \odot H^1(\Gamma, \partial \Gamma)\right) _{Aut(\Gamma)},
$$
where $\overline{\text{Gr}}(m,n)$ represents the set of all (isomorphism classes of) directed acyclic graphs with $n$ output legs and $m$ input legs that are irreducible in the sense that they do not allow any $\frac{1}{2}\mbox{-prop}$-ic contractions. The differential acts trivially on the $H^1(\Gamma, \partial \Gamma)$ part.

\begin{lemma}\label{2: exact functor}
$F$ is an exact functor.
\end{lemma}
\begin{proof}
Since the differential preserves the underlying graph, we get
\begin{eqnarray}\label{functor1}
H_\bullet(F(s)(m,n)) = \bigoplus_{\Gamma \in \overline{\text{Gr}}(m,n)} H_\bullet \left(\bigotimes_{v\in v(\Gamma)} s(Out(v),In(v)) \otimes \odot H^1(\Gamma, \partial \Gamma)\right) _{Aut(\Gamma)}
\end{eqnarray}
Since the differential commutes with elements of $Aut(\Gamma)$, $Aut(\Gamma)$ is finite and $\K$ is a field of characteristic zero, by Maschke's Theorem we have
\begin{eqnarray}\label{functor2}
(\ref{functor1}) = \bigoplus_{\Gamma \in \overline{\text{Gr}}(m,n)} \left( H_\bullet \left(\bigotimes_{v\in v(\Gamma)} s(Out(v),In(v)) \otimes \odot H^1(\Gamma, \partial \Gamma)\right) \right) _{Aut(\Gamma)}
\end{eqnarray}
Applying the K\"{u}nneth formula twice, together with the fact that the differential is trivial on $H^1(\Gamma, \partial \Gamma)$  we get
\begin{eqnarray*}
(\ref{functor2}) = \bigoplus_{\Gamma \in \overline{\text{Gr}}(m,n)}  \left(\bigotimes_{v\in v(\Gamma)} H_\bullet(s(Out(v),In(v))) \otimes \odot H^1(\Gamma, \partial \Gamma)\right)  _{Aut(\Gamma)} = F(H_\bullet(s))(m,n).
\end{eqnarray*}
\end{proof}

\sip

{\em Step 2: A genus filtration}.

 Consider the genus filtration of  $(\LoB_\infty, \delta)$,
and denote by
 $(\gr\LoB_\infty, \delta^{gen})$ the associated  graded properad.   The
 differential $\delta^{gen}$ in the
 complex $\gr\LoB_\infty$ is given by the formula,
 \Beq\label{2: delta_0 in E_0 diamond}
\delta^{gen}
\resizebox{14mm}{!}{\xy
(-9,-6)*{};
(0,0)*+{_a}*\cir{}
**\dir{-};
(-5,-6)*{};
(0,0)*+{_a}*\cir{}
**\dir{-};
(9,-6)*{};
(0,0)*+{_a}*\cir{}
**\dir{-};
(5,-6)*{};
(0,0)*+{_a}*\cir{}
**\dir{-};
(0,-6)*{\ldots};
(-10,-8)*{_1};
(-6,-8)*{_2};
(10,-8)*{_n};
(-9,6)*{};
(0,0)*+{_a}*\cir{}
**\dir{-};
(-5,6)*{};
(0,0)*+{_a}*\cir{}
**\dir{-};
(9,6)*{};
(0,0)*+{_a}*\cir{}
**\dir{-};
(5,6)*{};
(0,0)*+{_a}*\cir{}
**\dir{-};
(0,6)*{\ldots};
(-10,8)*{_1};
(-6,8)*{_2};
(10,8)*{_m};
\endxy}=
\sum_{a=b+c}\sum_{[m]=I_1\sqcup I_2\atop
[n]=J_1\sqcup J_2} \pm
\Ba{c}
%
%
\resizebox{18mm}{!}{\xy
(0,0)*+{_b}*\cir{}="b",
(10,10)*+{_c}*\cir{}="c",
%
(-9,6)*{}="1",
(-7,6)*{}="2",
(-2,6)*{}="3",
(-3.5,5)*{...},
(-4,-6)*{}="-1",
(-2,-6)*{}="-2",
(4,-6)*{}="-3",
(1,-5)*{...},
(0,-8)*{\underbrace{\ \ \ \ \ \ \ \ }},
(0,-11)*{_{J_1}},
(-6,8)*{\overbrace{ \ \ \ \ \ \ }},
(-6,11)*{_{I_1}},
(6,16)*{}="1'",
(8,16)*{}="2'",
(14,16)*{}="3'",
(11,15)*{...},
(11,6)*{}="-1'",
(16,6)*{}="-2'",
(18,6)*{}="-3'",
(13.5,6)*{...},
(15,4)*{\underbrace{\ \ \ \ \ \ \ }},
(15,1)*{_{J_2}},
(10,18)*{\overbrace{ \ \ \ \ \ \ \ \ }},
(10,21)*{_{I_2}},
%
\ar @{-} "b";"c" <0pt>
\ar @{-} "b";"1" <0pt>
\ar @{-} "b";"2" <0pt>
\ar @{-} "b";"3" <0pt>
\ar @{-} "b";"-1" <0pt>
\ar @{-} "b";"-2" <0pt>
\ar @{-} "b";"-3" <0pt>
\ar @{-} "c";"1'" <0pt>
\ar @{-} "c";"2'" <0pt>
\ar @{-} "c";"3'" <0pt>
\ar @{-} "c";"-1'" <0pt>
\ar @{-} "c";"-2'" <0pt>
\ar @{-} "c";"-3'" <0pt>
\endxy}
\Ea
\Eeq
Consider also the genus filtration of the dg properad $\cP$ and denote by
$(\gr\LoB_\infty\cP, 0)$ the associated graded. The morphism (\ref{2: nu quasi-iso
to P}) of filtered complexes
induces a sequence of morphisms of the associated graded complexes,
\Beq\label{2: morphism of 1st terms of genus spectral seq}
\nu: \gr\LoB_\infty \lon \gr\cP,
\Eeq
Thanks to the Spectral Sequences Comparison Theorem, Proposition {\ref{2: proposition
on nu quasi-iso}} will be proven if we  show that the map $\nu$ is a quasi-isomorphism of complexes. We shall compute below
 the cohomology $H^\bu(\gr \LoB_\infty,\delta^{gen})$ which will make it evident
 that the map
 $
 \nu
 $ is an isomorphism indeed.
\bip

{\em Step 3: An auxiliary prop}.
Let us consider a properad $\cQ = F( \Omega_{\frac{1}{2}} ( \LB^{\text{!`}}_{\frac{1}{2}}))$, where $\LB_{\frac{1}{2}}$ is the $\frac{1}{2}$-prop governing Lie bialgebras. Explicitly, $\cQ$ is generated by corollas with either $a=m=n=1$ or $a=0$ and $m+n\geq 3$ subject to the relations
$$
\Ba{c}
\resizebox{3mm}{!}{\xy
(0,0)*+{_1}*\cir{}="b",
(0,6)*+{_1}*\cir{}="c",
(0,-4)*{}="-1",
(0,10)*{}="1'",
\ar @{-} "b";"c" <0pt>
\ar @{-} "b";"-1" <0pt>
\ar @{-} "c";"1'" <0pt>
\endxy}
\Ea=0\ \ \ , \ \ \
\sum_{i=1}^m
\Ba{c}
\resizebox{17mm}{!}{\xy
(-9,12.5)*{^{1}},
(-4,12.5)*{^{^{i-1}}},
(0,21)*{^{^{i}}},
(5,12.5)*{^{^{i+1}}},
(9.9,12.5)*{^{^m}},
(-9,-1.5)*{_{_1}},
(-4.5,-1.5)*{_{_{2}}},
(4.5,-1.5)*{_{_{n-1}}},
(9.9,-1.5)*{_{_n}},
(-5,11)*{...},
(5,11)*{...},
(0,1)*{...},
(0,6)*+{_0}*\cir{}="b",
(0,15)*+{_1}*\cir{}="c",
(0,-4)*{}="-1",
(-9,12)*{}="1'",
(-3.5,12)*{}="2'",
(3.5,12)*{}="3'",
(9,12)*{}="4'",
(0,21)*{}="u",
(-9,0)*{}="-1",
(-4.5,0)*{}="-2",
(4.5,0)*{}="-3",
(9,0)*{}="-4",
\ar @{-} "u";"c" <0pt>
\ar @{-} "b";"c" <0pt>
\ar @{-} "b";"-1" <0pt>
\ar @{-} "b";"1'" <0pt>
\ar @{-} "b";"2'" <0pt>
\ar @{-} "b";"3'" <0pt>
\ar @{-} "b";"4'" <0pt>
\ar @{-} "b";"-1" <0pt>
\ar @{-} "b";"-2" <0pt>
\ar @{-} "b";"-3" <0pt>
\ar @{-} "b";"-4" <0pt>
\endxy}
\Ea
-
\sum_{i=1}^n
\Ba{c}
\resizebox{17mm}{!}{\xy
(-9,-13.5)*{_{_1}},
(-4,-13.5)*{_{_{i-1}}},
(0,-22)*{_{_{i}}},
(5,-13.5)*{_{_{i+1}}},
(9.9,-13.5)*{_{_n}},
(-9,1.5)*{_{_1}},
(-4.5,1.5)*{_{_{2}}},
(4.5,1.5)*{^{_{m-1}}},
(9.9,1.5)*{_{_m}},
(-5,-11)*{...},
(5,-11)*{...},
(0,-1)*{...},
(0,-6)*+{_0}*\cir{}="b",
(0,-15)*+{_1}*\cir{}="c",
(0,4)*{}="-1",
(-9,-12)*{}="1'",
(-3.5,-12)*{}="2'",
(3.5,-12)*{}="3'",
(9,-12)*{}="4'",
(0,-21)*{}="u",
(-9,0)*{}="-1",
(-4.5,0)*{}="-2",
(4.5,0)*{}="-3",
(9,0)*{}="-4",
\ar @{-} "u";"c" <0pt>
\ar @{-} "b";"c" <0pt>
\ar @{-} "b";"-1" <0pt>
\ar @{-} "b";"1'" <0pt>
\ar @{-} "b";"2'" <0pt>
\ar @{-} "b";"3'" <0pt>
\ar @{-} "b";"4'" <0pt>
\ar @{-} "b";"-1" <0pt>
\ar @{-} "b";"-2" <0pt>
\ar @{-} "b";"-3" <0pt>
\ar @{-} "b";"-4" <0pt>
\endxy}
\Ea
=0
$$
With this description we see that $\nu$ factors through $\cQ$, $\nu:\gr\LoB_\infty \overset{p}\twoheadrightarrow \cQ \overset{q}\twoheadrightarrow  \gr\cP$.
Furthermore we claim that $q$ is a quasi-isomorphism.
First notice that  $\gr\cP = F(\LB_{\frac{1}{2}})$.
The $\frac{1}{2}$-prop $\LB_{\frac{1}{2}}$ is Koszul, i.~e., the natural projection
$\Omega_{\frac{1}{2}} ( \LB^{\text{!`}}_{\frac{1}{2}}) \twoheadrightarrow \LB_{\frac{1}{2}}$ is a quasi-isomorphism.
The result follows by applying the functor $F$ to this map and by Lemma {\ref{2: exact functor}}.
\bip

{\em Step 4: $p$ is a quasi-isomorphism}.

Consider a filtration of $\gr\LoB_\infty$ given by the sum of all decorations of non-bivalent vertices.
On the $0$-th page of this spectral sequence the differential acts only by splitting bivalent vertices. Then Proposition $\ref{2: propos on aux graph complexes}$ tells us that the first page of this spectral sequence consists of graphs with no bivalent vertices such that every vertex is decorated by a number $a\in \mathbb Z^+$ and every edge has either a decoration $\xy(0,0)*+{_{1}}*\cir{};
\endxy$ or no decoration.
The differential acts by
\begin{equation}\label{equ:diffspecs}
\resizebox{15mm}{!}{\xy
(-9,-6)*{};
(0,0)*+{_a}*\cir{}
**\dir{-};
(-5,-6)*{};
(0,0)*+{_a}*\cir{}
**\dir{-};
(9,-6)*{};
(0,0)*+{_a}*\cir{}
**\dir{-};
(5,-6)*{_a};
(0,0)*+{a}*\cir{}
**\dir{-};
(0,-6)*{\ldots};
(-10,-8)*{_1};
(-6,-8)*{_2};
(10,-8)*{_n};
(-9,6)*{};
(0,0)*+{_a}*\cir{}
**\dir{-};
(-5,6)*{};
(0,0)*+{_a}*\cir{}
**\dir{-};
(9,6)*{};
(0,0)*+{_a}*\cir{}
**\dir{-};
(5,6)*{};
(0,0)*+{_a}*\cir{}
**\dir{-};
(0,6)*{\ldots};
(-10,8)*{_1};
(-6,8)*{_2};
(10,8)*{_m};
\endxy}
\lon
\sum_{i=1}^m
\Ba{c}
\resizebox{17mm}{!}{\xy
(-9,12.5)*{^{1}},
(-4,12.5)*{^{^{i-1}}},
(0,21)*{^{^{i}}},
(5,12.5)*{^{^{i+1}}},
(9.9,12.5)*{^{^m}},
(-9,-1.5)*{_{_1}},
(-4.5,-1.5)*{_{_{2}}},
(4.5,-1.5)*{_{_{n-1}}},
(9.9,-1.5)*{_{_n}},
(-5,11)*{...},
(5,11)*{...},
(0,1)*{...},
(0,6)*+{_{a-1}}*\cir{}="b",
(0,15)*+{_1}*\cir{}="c",
(0,-4)*{}="-1",
(-9,12)*{}="1'",
(-3.5,12)*{}="2'",
(3.5,12)*{}="3'",
(9,12)*{}="4'",
(0,21)*{}="u",
(-9,0)*{}="-1",
(-4.5,0)*{}="-2",
(4.5,0)*{}="-3",
(9,0)*{}="-4",
\ar @{-} "u";"c" <0pt>
\ar @{-} "b";"c" <0pt>
\ar @{-} "b";"-1" <0pt>
\ar @{-} "b";"1'" <0pt>
\ar @{-} "b";"2'" <0pt>
\ar @{-} "b";"3'" <0pt>
\ar @{-} "b";"4'" <0pt>
\ar @{-} "b";"-1" <0pt>
\ar @{-} "b";"-2" <0pt>
\ar @{-} "b";"-3" <0pt>
\ar @{-} "b";"-4" <0pt>
\endxy}
\Ea
-
\sum_{i=1}^n
\Ba{c}
\resizebox{17mm}{!}{\xy
(-9,-13.5)*{_{_1}},
(-4,-13.5)*{_{_{i-1}}},
(0,-22)*{_{_{i}}},
(5,-13.5)*{_{_{i+1}}},
(9.9,-13.5)*{_{_n}},
(-9,1.5)*{_{_1}},
(-4.5,1.5)*{_{_{2}}},
(4.5,1.5)*{^{_{m-1}}},
(9.9,1.5)*{_{_m}},
(-5,-11)*{...},
(5,-11)*{...},
(0,-1)*{...},
(0,-6)*+{_{a-1}}*\cir{}="b",
(0,-15)*+{_1}*\cir{}="c",
(0,4)*{}="-1",
(-9,-12)*{}="1'",
(-3.5,-12)*{}="2'",
(3.5,-12)*{}="3'",
(9,-12)*{}="4'",
(0,-21)*{}="u",
(-9,0)*{}="-1",
(-4.5,0)*{}="-2",
(4.5,0)*{}="-3",
(9,0)*{}="-4",
\ar @{-} "u";"c" <0pt>
\ar @{-} "b";"c" <0pt>
\ar @{-} "b";"-1" <0pt>
\ar @{-} "b";"1'" <0pt>
\ar @{-} "b";"2'" <0pt>
\ar @{-} "b";"3'" <0pt>
\ar @{-} "b";"4'" <0pt>
\ar @{-} "b";"-1" <0pt>
\ar @{-} "b";"-2" <0pt>
\ar @{-} "b";"-3" <0pt>
\ar @{-} "b";"-4" <0pt>
\endxy}
\Ea
\end{equation}

The complex we obtain is precisely
$$
 \bigoplus_{\Gamma \in \overline{\text{Gr}}(m,n)} \left(\bigotimes_{v\in v(\Gamma)} \Omega_{\frac{1}{2}} ( \LB^{\text{!`}}_{\frac{1}{2}})(Out(v),In(v)) \otimes \odot C^*(\Gamma, \partial \Gamma)\right) _{Aut(\Gamma)},
$$
where $C^*(\Gamma, \partial \Gamma)$ are the simplicial co-chains of $\Gamma$ relative to its boundary; the differential in this complex is given by the standard differential in  $C^*(\Gamma, \partial \Gamma)$.
Indeed, we may identify $C^0(\Gamma, \partial \Gamma)\cong \K[V(\Ga)]$ and $C^1(\Gamma, \partial \Gamma)\cong \K[E(\Ga)]$.
A vertex $v=\xy(0,0)*+{_{a_v}}*\cir{};
\endxy$ with weight $a_v$ corresponds to the $a_v$-th power of the cochain representing the vertex, and an edge decorated
with the symbol $\xy(0,0)*+{_{1}}*\cir{};
\endxy$ corresponds to the cochain representing the edge.
The differential $d$ on $C^1(\Gamma, \partial \Gamma)$ is the map dual to the standard boundary map $\p: C_1(\Gamma, \partial \Gamma)\cong \K[E(\Ga)]
\lon C_0(\Gamma, \partial \Gamma)\cong \K[V(\Ga)]$. It is given, on a vertex $v\in V(\Ga)$, by
$$
dv =\sum_{e_v'\in Out(v)} e_v' -
 \sum_{e_v'\in In(v)} e_v'',
$$
where $Out(v)$ is the set of edges outgoing from $v$ and $In(v)$ is the set of edges ingoing to $v$. This exactly matches the
differential \eqref{equ:diffspecs} on the first page of the spectral sequence.
As $H^0(\Ga,\p \Ga)=0$ 
and since the symmetric product functor $\odot$ is exact we obtain $\cQ$ on the second page of the spectral sequence,
 $$
H^\bu(\gr\LoB_\infty)\cong
 \bigoplus_{\Gamma \in \overline{\text{Gr}}(m,n)} \left(\bigotimes_{v\in v(\Gamma)} \Omega_{\frac{1}{2}} ( \LB^{\text{!`}}_{\frac{1}{2}})(Out(v),In(v)) \otimes \odot H^1(\Gamma, \partial \Gamma)\right)_{Aut(\Gamma)}=
  F( \Omega_{\frac{1}{2}} ( \LB^{\text{!`}}_{\frac{1}{2}}))
  \cong Q
$$
 thus showing that $p$ is a quasi-isomorphism.
\end{proof}

\subsection{Auxiliary complexes} \label{sec:extracomplexes}
Let $\cA_n$ be a quadratic algebra generated by $x_1,\dots, x_n$ with relations $x_ix_{i+1}=x_{i+1}x_i$ for $i=1,\dots, n-1$.
We denote by $C_n=\cA_n^{\text{!`}}$ the Koszul dual coalgebra. Notice that $\cA_n$ and $C_n$ are weight graded and the weight $k$ component of $C_n$, $C_n^{(k)}$ is zero if $k\geq 3$, while $C_n^{(1)}=\vecspan \{x_1,\dots, x_n\}$ and
$C_n^{(2)}=\vecspan\{u_{1,2}=x_1x_2-x_2x_1,u_{2,3}=x_2x_3-x_3x_2, \dots, u_{n-1,n}=x_{n-1}x_n-x_nx_{n-1}\}$.

\begin{proposition}\label{2: toy problem}
The algebra $\cA_n$ is Koszul. In particular, the canonical projection map $A_n:=\Omega(C_n)\to \cA_n$ from the cobar construction of $C_n$ is a quasi-isomorphism.
\end{proposition}

The proof of this proposition is given in Appendix \ref{app:koszulnessproof}.


\sip

Proposition {\ref{2: toy problem}} in particular implies that the homology of the $A_n$ vanishes in positive degree. The complex
$A_n$ is naturally multigraded by the amount of times each index $j$ appears on each word and the differential respects this multigrading.
We will be interested in particular in the subcomplex $\cA_n^{1,1,\dots,1}$ of $A_n$ that is spanned by words in $x_j$ and $u_{i,i+1}$ such that each index occurs exactly once. Since $A_n^{1,1,\dots,1}$ is a direct summand of $A_n$, its homology also vanishes in positive degree.

Let us define a Lie algebra $\caL_n = Lie(x_1,\dots,x_n)/[x_i,x_{i+1}]$ and a complex $L_n=Lie(x_1,\dots,x_n, u_{1,2},\dots u_{n-1,n})$, with $dx_i=0$ and $d(u_{i,i+1})= x_ix_{i+1}-x_{i+1}x_i$. Here $Lie$ stands for the free Lie algebra functor.

\begin{lemma}\label{2: Ln quasi-iso}
The projection map $L_n \twoheadrightarrow \caL_n$ is a quasi-isomorphism.
\end{lemma}
\begin{proof}
It is clear that $H^0(L_n) = \caL_n$, therefore it is enough to see that the homology of $L_n$ vanishes in positive degree.
The Poincar\' e-Birkhoff-Witt Theorem gives us an isomorphism $$\odot(\caL  ie(x_1,\dots,x_n, u_{1,2},\dots u_{n-1,n}))=\odot (L_n) \overset{\sim}\longrightarrow \cA ss(x_1,\dots,x_n, u_{1,2},\dots u_{n-1,n})= A_n.$$
This map commutes with with the differentials, therefore we have an isomorphism in homology $H_{\bullet}(\odot (L_n))= H_{\bullet}(A_n)$. Since $\odot$ is an exact functor it commutes with taking homology and since the homology of $A_n$ vanishes in positive degree by {\ref{2: toy problem}} the result follows.
\end{proof}

Let us define $A_{n_1,\dots, n_r}$ as the  coproduct of $A_{n_1}, \dots, A_{n_r}$ in the category of associative algebras; $A_{n_1,\dots, n_k}$ consists of words in $x_{1}^1,x_2^1,\dots x_{n_1}^1, x_{1}^2,\dots, x_{n_2}^2,\dots ,x_1^r,\dots, x_{n_r}^r,
u_{1,2}^1,\dots, u_{n_1-1,n_1}^1,u_{1,2}^2,\dots , u_{n_r-1,n_r}^r.$ We define similarly $L_{n_1,\dots, n_r}$ and $\caL_{n_1,\dots, n_r}$.

\begin{lemma}
The homology of $A_{n_1,\dots, n_r}$ vanishes in positive degree.
\end{lemma}
\begin{proof}
We prove the statement for $r=2$ and the result follows by induction.

We can write $A_{n_1,n_2} = \bigoplus_\epsilon A_\epsilon$ where $\epsilon=(e_1,\dots e_{|\epsilon|})$ represents a string of alternating $n_1$'s and $n_2$'s of finite size $|\epsilon|$. The differential preserves this decomposition.
$A_\epsilon$ is naturally $|\epsilon|$-multigraded.

Let $x = \sum \alpha_I x_I \in A_\epsilon$, where $x_I$ represents some word in $x_i^1, x_i^2, u_{i,i+1}^1, u_{i,i+1}^2$ be a cycle of nonzero degree.

Then, when we apply the Leibniz rule on the word $x_I$ each parcel has a different multigrading, therefore the equation $dx=0$ splits into $|\epsilon|$ equations whose right hand side is zero. For each one of them we can use the fact that the homology of $A_{n_1}$ and $A_{n_2}$ vanishes in non-zero degree to conclude that the left hand side of each of these equations is actually a boundary. It follows that $x$ is a boundary as well.
\end{proof}

\begin{lemma}
The map $L_{n_1,\dots, n_r} \twoheadrightarrow \caL_{n_1,\dots, n_r}$ is a quasi-isomorphism.
\end{lemma}
\begin{proof}
The same argument from Lemma {\ref{2: Ln quasi-iso}} holds.
\end{proof}

\begin{corollary}\label{2: auxiliary corollary}
The map $L_{n_1,\dots, n_r}^{1,\dots 1} \twoheadrightarrow \caL_{n_1,\dots, n_r}^{1,\dots 1}$ is a quasi-isomorphism.
\end{corollary}

\subsection{Main Theorem}\label{2: Theorem on Koszulness} {\em The properad $\LoB$ is
Koszul, i.e.\ the natural surjection \eqref{2: pi quasi-iso}
a quasi-isomorphism.}

\begin{proof} The surjection  \eqref{2: pi quasi-iso} factors through the surjection
\eqref{2: nu quasi-iso to P},
$$
\pi: \LoB_\infty \stackrel{\nu}{\lon} \cP \stackrel{\rho}{\lon} \LoB.
$$
In view of Proposition {\ref{2: proposition on nu quasi-iso}}, the Main theorem is
proven once it is shown that the morphism $\rho$ is a quasi-isomorphism. The latter
statement is, in turn, proven once it is shown that the cohomology of the
non-positively graded dg properad $\cP$ is concentrated in degree zero. For notation reasons
it is suitable to work with the dg prop, $\cP \cP$, generated by the properad $\cP$. We also
denote by $\caL ie P$ the prop governing Lie algebras and by $\caL ie CP$ the prop
governing Lie coalgebras.

\sip

It is easy to see that the dg prop  $\cP \cP=\{\cP\cP(m,n)\}$ is isomorphic, as a dg $\mathbb{S}$-bimodule, to
the dg prop generated by a degree 1 corolla
$\begin{xy}
 <0mm,-0.55mm>*{};<0mm,-3mm>*{}**@{-},
 <0mm,0.5mm>*{};<0mm,3mm>*{}**@{-},
 <0mm,0mm>*{\bu};<0mm,0mm>*{}**@{},
 \end{xy}$\, , degree zero corollas
$
 \begin{xy}
 <0mm,-0.55mm>*{};<0mm,-2.5mm>*{}**@{-},
 <0.5mm,0.5mm>*{};<2.2mm,2.2mm>*{}**@{-},
 <-0.48mm,0.48mm>*{};<-2.2mm,2.2mm>*{}**@{-},
 <0mm,0mm>*{\circ};<0mm,0mm>*{}**@{},
 <0.5mm,0.5mm>*{};<2.7mm,2.8mm>*{^{_2}}**@{},
 <-0.48mm,0.48mm>*{};<-2.7mm,2.8mm>*{^{_1}}**@{},
 \end{xy}
=-
\begin{xy}
 <0mm,-0.55mm>*{};<0mm,-2.5mm>*{}**@{-},
 <0.5mm,0.5mm>*{};<2.2mm,2.2mm>*{}**@{-},
 <-0.48mm,0.48mm>*{};<-2.2mm,2.2mm>*{}**@{-},
 <0mm,0mm>*{\circ};<0mm,0mm>*{}**@{},
 <0.5mm,0.5mm>*{};<2.7mm,2.8mm>*{^{_1}}**@{},
 <-0.48mm,0.48mm>*{};<-2.7mm,2.8mm>*{^{_2}}**@{},
 \end{xy}
 $ and $\begin{xy}
 <0mm,0.66mm>*{};<0mm,3mm>*{}**@{-},
 <0.39mm,-0.39mm>*{};<2.2mm,-2.2mm>*{}**@{-},
 <-0.35mm,-0.35mm>*{};<-2.2mm,-2.2mm>*{}**@{-},
 <0mm,0mm>*{\circ};<0mm,0mm>*{}**@{},
   <0.39mm,-0.39mm>*{};<2.9mm,-4mm>*{^{_2}}**@{},
   <-0.35mm,-0.35mm>*{};<-2.8mm,-4mm>*{^{_1}}**@{},
\end{xy}=-
\begin{xy}
 <0mm,0.66mm>*{};<0mm,3mm>*{}**@{-},
 <0.39mm,-0.39mm>*{};<2.2mm,-2.2mm>*{}**@{-},
 <-0.35mm,-0.35mm>*{};<-2.2mm,-2.2mm>*{}**@{-},
 <0mm,0mm>*{\circ};<0mm,0mm>*{}**@{},
   <0.39mm,-0.39mm>*{};<2.9mm,-4mm>*{^{_1}}**@{},
   <-0.35mm,-0.35mm>*{};<-2.8mm,-4mm>*{^{_2}}**@{},
\end{xy}$,
 modulo the first three relations of \eqref{R for LieB}
and the following ones,
\Beq\label{2: relations for EP}
\xy
(0,-1.9)*{\bu}="0",
 (0,1.9)*{\bu}="1",
(0,-5)*{}="d",
(0,5)*{}="u",
\ar @{-} "0";"u" <0pt>
\ar @{-} "0";"1" <0pt>
\ar @{-} "1";"d" <0pt>
\endxy=0, \ \ \ \
\xy
(0,-1.9)*{\circ}="0",
 (0,1.9)*{\bu}="1",
(-2.5,-5)*{}="d1",
(2.5,-5)*{}="d2",
(0,5)*{}="u",
\ar @{-} "1";"u" <0pt>
\ar @{-} "0";"1" <0pt>
\ar @{-} "0";"d1" <0pt>
\ar @{-} "0";"d2" <0pt>
\endxy
=0\ , \ \ \ \
\xy
(0,1.9)*{\circ}="0",
 (0,-1.9)*{\bu}="1",
(-2.5,5)*{}="d1",
(2.5,5)*{}="d2",
(0,-5)*{}="u",
\ar @{-} "1";"u" <0pt>
\ar @{-} "0";"1" <0pt>
\ar @{-} "0";"d1" <0pt>
\ar @{-} "0";"d2" <0pt>
\endxy\ =\ 0\ .
\Eeq
 The latter complex can in turn be identified with the following
collection of dg  vector spaces,
\begin{equation}\label{2: basis in PP}
   W(n,m):= \bigoplus_N \left( \caL ieP(n,N) \otimes V^{\otimes N} \otimes \caL ieCP(N,m)
   \right)_{\bS_N}
 \end{equation}
where $V$ is a two-dimensional vector space $V_0 \oplus V_1$, where $V_0 = \vecspan\langle  \ \xy
(0,0)*{}="0",
(0,-4)*{}="d",
(0,4)*{}="u",
\ar @{-} "0";"u" <0pt>
\ar @{-} "0";"d" <0pt>
\endxy \ \rangle$ and
$V_1 = \vecspan\langle \xy
(0,0)*{\bu}="0",
(0,-4)*{}="d",
(0,4)*{}="u",
\ar @{-} "0";"u" <0pt>
\ar @{-} "0";"d" <0pt>
\endxy \rangle$
and the differential in this complex is given by formulae (\ref{2: differential in P}).\\

Let us consider a slightly different complex
\Beq \label{equ:V_mn}
\begin{aligned}
V_{n,m}
&=
\bigoplus_N \left( \caL ieP(n,N) \otimes V^{\otimes N} \otimes \cA ssCP(N,m)
\right)_{\bS_N}\\
&\cong  \bigoplus_N \bigoplus_{N=n_1+...+n_m} \left( \caL ieP(n,N) \otimes
V^{\otimes N}
\otimes \cA ssC(n_1) \ot \dots \ot \cA ssC(n_m)\right)_{\bS_{n_1}\times \dots \times \bS_{n_k}},
\end{aligned}
\Eeq
where  $\cA ssCP$ is the prop governing coassociative coalgebras.

The operad $\cA ss = \cC om \circ \caL ie = \bigoplus_k (\cC om \circ \caL ie)^{(k)}$ is naturally graded with respect to the arity in $\cC om $.
This decomposition induces a multigrading in $V_{n,m} = \displaystyle\bigoplus_{(k_1,\dots, k_m)} V_{n;k_1,\dots, k_m}$. It is clear that this decomposition is actually a splitting of complexes and in fact the direct summand $V_{1,\dots,1}$ is just $W(n,m)$, therefore to show the theorem it suffices to show that the cohomology of $V_{n,m}$ is zero in positive degree.\\

There is a natural identification $\caL ieP(n,N) \cong \left(\left(\caL ie(y_1,\cdots, y_N)\right)^{\ot n}\right)^{1,\dots,1}$ where, as before, we use the notation $1,\dots,1$ to represent the subspace spanned by words such that each index appears exactly once.

$\left(\caL ieP(n,N) \ot \cA ssC(n_1) \ot \cdots \cA ssC(n_m)\right)_{S_{n_1}\times \dots \times S_{n_k}}$ is isomorphic to  $\caL ieP(n,N) $ but there is a more natural identification than the one above, namely the $y_j$ can be gathered by blocks of size $n_j$, according to the action of $S_{1}\times\dots\times S_{n_m}$ on $\caL ieP(n,N)$ and can be relabeled accordingly:
\begin{eqnarray*}
y_1, y_2, \dots ,y_{N} \leadsto y_1^1,\dots , y_{n_1}^1, y_1^2,\dots , y_{n_2}^2 ,\dots , y_1^m,\dots , y_{n_m}^m.
\end{eqnarray*}

Since $V = V_0 \oplus V_1$, there is a natural decomposition of $V^{\ot N} = \displaystyle\bigoplus_{\epsilon} V_{\epsilon_1} \ot \dots \ot V_{\epsilon_N}$, where $\epsilon= (\epsilon_1,\dots,\epsilon_N)$ runs through all strings of $0$'s and $1$ of length $N$.
Then,
\begin{eqnarray*}
\left(\caL ieP(n,N) \otimes
V^{\otimes N}
\otimes \cA ssC(n_1) \ot \dots \ot \cA ssC(n_m)\right)_{\bS_{n_1}\times \dots \times \bS_{n_k}} =\hspace{55mm} \\
\ \ \ \ \ \ \ \ \ \ \bigoplus_{\epsilon}\left( \caL ieP(n,N) \otimes
V_{\epsilon_1} \ot \dots \ot V_{\epsilon_N}
\otimes \cA ssC(n_1) \ot \dots \ot \cA ssC(n_m)\right)_{\bS_{n_1}\times \dots \times \bS_{n_k}}
\end{eqnarray*}

Given a fixed a string $\epsilon$ we look at each summand individually. If $\epsilon_k$ is zero we wish to correspond $V_{\epsilon_k}$ to an element of the form $x_j^i$ and if $\epsilon_k=1$ we wish to correspond $V_{\epsilon_k}$ to an element of the form $u_{j,j+1}^i$. We do this recursively, to determine the correct indices:\\

\Bi
\item[-] For $k=1$: $\left\{\Ba{c} \mathrm{if}\ \epsilon_1 = 0, \mathrm{then}\ V_{\epsilon_1}= \text{span } \langle x_{1}^1\rangle;\\
 \mathrm{if}\ \epsilon_1 = 1\  \mathrm{then}\ V_{\epsilon_1}= \text{span } \langle u_{1,2}^1\rangle.
\Ea\right.$

\item[-] For $k>1$:
\Bi
\item[]	if $\epsilon_k = 0$ then
\Bi
\item[]		if $V_{\epsilon_{k-1}}= \text{span } \langle x_{j}^i\rangle$ or $V_{\epsilon_{k-1}}= \text{span } \langle u_{j-1,j}^i\rangle$ and the total number of variables with upper script $i$ is still smaller then $n_i$, then $V_{\epsilon_{k}} = \text{span }\langle x_{j+1}^i\rangle$.
\item[]
    Otherwise $V_{\epsilon_{k}} = \text{span }\langle x_{1}^{i+1}\rangle$.
\Ei

\item[]	if $\epsilon_k = 1$ then
\Bi
\item[]		if $V_{\epsilon_{k-1}}= \text{span } \langle x_{j}^i\rangle$ or $V_{\epsilon_{k-1}}= \text{span } \langle u_{j-1,j}^i\rangle$ and the total number of variables with upper script $i$ is still smaller then $n_i$, then $V_{\epsilon_{k}} = \text{span } \langle u_{j+1,j+2}^i \rangle$.
    \item[] Otherwise $V_{\epsilon_{k}} = \text{span } \langle u_{1,2}^{i+1}\rangle$.
\Ei
\Ei
\Ei
With this notation, the total space

$$V_{n,m} = \bigoplus_N \bigoplus_\epsilon \bigoplus_{N=n_1+...+n_m} \left( \caL ieP(n,N) \otimes
V_{\epsilon_1} \ot \dots \ot V_{\epsilon_N}
\otimes \cA ssC(n_1) \ot \dots \ot \cA ssC(n_m)\right)_{\bS_{n_1}\times \dots \times \bS_{n_k}}$$

can be seen as a sum of spaces of the form
$$\left(\left(\caL ie(x_1^1,\dots , x_{\tilde n_1}^1,\dots , x_1^m,\dots , x_{\tilde n_m}^m, u_{1,2}^1,\dots, u_{\tilde n_1-1,\tilde n_1}^1,u_{1,2}^2,\dots , u_{\tilde n_m-1,\tilde n_m}^m )\right)^{\ot n}\right)^{1,\dots,1},$$
where $\tilde n_i$ is the biggest index produced by the algorithm described above.

 Explicitly $\tilde n_i = n_i + \displaystyle\sum_{j=n_1+...+n_{i-1}+1}^{n_1+...+n_{i}} \epsilon_j$.

   For example, consider the following element of $\caL ieP(2,5)\otimes V^{\otimes 5}$:
   $$
\resizebox{22mm}{!}{   \xy
(-18,0)*{}="L",
(15,0)*{}="R",
(-5,15)*{}="u1",
(-5,10)*{\circ}="a1",
(-9,5)*{\circ}="a2",
(-13,0)*{}="L1",
(-5,0)*{\bu}="L2",
(5,14)*{}="u2",
(5,9)*{\circ}="b1",
(10,0)*{\bu}="L3",
(-5,-12)*{}="-u1",
(-5,-7)*{\circ}="-a1",
(5,-11)*{}="-u2",
(5,-6)*{\circ}="-b1",
\ar @{-} "u1";"a1" <0pt>
\ar @{-} "a1";"a2" <0pt>
\ar @{-} "a2";"L1" <0pt>
\ar @{-} "a2";"L2" <0pt>
\ar @{-} "u2";"b1" <0pt>
\ar @{-} "b1";"L3" <0pt>
\ar @{-} "-u1";"-a1" <0pt>
\ar @{-} "L1";"-a1" <0pt>
\ar @{-} "L2";"-a1" <0pt>
\ar @{-} "b1";"-a1" <0pt>
\ar @{-} "-b1";"-u2" <0pt>
\ar @{-} "-b1";"L3" <0pt>
\ar @{-} "-b1";"a1" <0pt>
\ar @{--} "L";"R" <0pt>
\endxy}
   $$

This element is mapped to the expression $[[x_1^1,u^1_{2,3}],x^2_1] \otimes [x^1_4,u^2_{2,3}]$.

Under this identification the differential sends the elements $u^i_{j,j+1}$ to $[x_j^i,x_{j+1}^i]$ and it is zero on the elements $x_j^i$.
Then the differential preserves the $\tilde n_i$'s therefore it preserves this direct sum.

\sip

We conclude that the complex $V_{n,m}$ splits as a sum of tensor products of complexes of the form
$L_{p_1,\dots, p_k}^{1,\dots, 1}$, so from Corollary {\ref{2: auxiliary corollary}} we obtain that its cohomology is concentrated in degree zero.
The proof of the Main theorem is completed.
\end{proof}

\subsection{Remark}\label{rem:degreeshifted} In applications of the theory of involutive Lie bialgebras to string topology, contact topology and quantum $\cA ss_\infty$ algebras  one is often interested in a version of the properad $\LoB$  in which degrees of Lie and coLie operations differ by an even number,
$$
|[\ ,\ ]| - |\vartriangle|=2d,\ \ \ \ d\in \N.
$$
The arguments proving Koszulness of $\LoB$ work also for such degree shifted versions of $\LoB$.
The same remark applies to the Koszul dual properads below.

\subsection{Properads of Frobenius algebras}
The properad of non-unital Frobenius algebras $\cF rob_d$ in dimension $d$ is the properad generated by operations
$
 \begin{xy}
 <0mm,-0.55mm>*{};<0mm,-2.5mm>*{}**@{-},
 <0.5mm,0.5mm>*{};<2.2mm,2.2mm>*{}**@{-},
 <-0.48mm,0.48mm>*{};<-2.2mm,2.2mm>*{}**@{-},
 <0mm,0mm>*{\circ};<0mm,0mm>*{}**@{},
 <0.5mm,0.5mm>*{};<2.7mm,2.8mm>*{^{_2}}**@{},
 <-0.48mm,0.48mm>*{};<-2.7mm,2.8mm>*{^{_1}}**@{},
 \end{xy}
=
(-1)^d
\begin{xy}
 <0mm,-0.55mm>*{};<0mm,-2.5mm>*{}**@{-},
 <0.5mm,0.5mm>*{};<2.2mm,2.2mm>*{}**@{-},
 <-0.48mm,0.48mm>*{};<-2.2mm,2.2mm>*{}**@{-},
 <0mm,0mm>*{\circ};<0mm,0mm>*{}**@{},
 <0.5mm,0.5mm>*{};<2.7mm,2.8mm>*{^{_1}}**@{},
 <-0.48mm,0.48mm>*{};<-2.7mm,2.8mm>*{^{_2}}**@{},
 \end{xy}
 $ (graded co-commutative comultiplication) of degree $d$
 and
 $
 \begin{xy}
 <0mm,0.66mm>*{};<0mm,3mm>*{}**@{-},
 <0.39mm,-0.39mm>*{};<2.2mm,-2.2mm>*{}**@{-},
 <-0.35mm,-0.35mm>*{};<-2.2mm,-2.2mm>*{}**@{-},
 <0mm,0mm>*{\circ};<0mm,0mm>*{}**@{},
   <0.39mm,-0.39mm>*{};<2.9mm,-4mm>*{^{_2}}**@{},
   <-0.35mm,-0.35mm>*{};<-2.8mm,-4mm>*{^{_1}}**@{},
\end{xy}=
\begin{xy}
 <0mm,0.66mm>*{};<0mm,3mm>*{}**@{-},
 <0.39mm,-0.39mm>*{};<2.2mm,-2.2mm>*{}**@{-},
 <-0.35mm,-0.35mm>*{};<-2.2mm,-2.2mm>*{}**@{-},
 <0mm,0mm>*{\circ};<0mm,0mm>*{}**@{},
   <0.39mm,-0.39mm>*{};<2.9mm,-4mm>*{^{_1}}**@{},
   <-0.35mm,-0.35mm>*{};<-2.8mm,-4mm>*{^{_2}}**@{},
\end{xy}
$ (graded commutative multiplication) of degree 0,
modulo the ideal generated by the following relations,
\begin{equation}\label{equ:Frobrelations}
\Ba{c}
\begin{xy}
 <0mm,0mm>*{\circ};<0mm,0mm>*{}**@{},
 <0mm,-0.49mm>*{};<0mm,-3.0mm>*{}**@{-},
 <0.49mm,0.49mm>*{};<1.9mm,1.9mm>*{}**@{-},
 <-0.5mm,0.5mm>*{};<-1.9mm,1.9mm>*{}**@{-},
 <-2.3mm,2.3mm>*{\circ};<-2.3mm,2.3mm>*{}**@{},
 <-1.8mm,2.8mm>*{};<0mm,4.9mm>*{}**@{-},
 <-2.8mm,2.9mm>*{};<-4.6mm,4.9mm>*{}**@{-},
   <0.49mm,0.49mm>*{};<2.7mm,2.3mm>*{^3}**@{},
   <-1.8mm,2.8mm>*{};<0.4mm,5.3mm>*{^2}**@{},
   <-2.8mm,2.9mm>*{};<-5.1mm,5.3mm>*{^1}**@{},
 \end{xy}\Ea
\ = \
\Ba{c}
\begin{xy}
 <0mm,0mm>*{\circ};<0mm,0mm>*{}**@{},
 <0mm,-0.49mm>*{};<0mm,-3.0mm>*{}**@{-},
 <0.49mm,0.49mm>*{};<1.9mm,1.9mm>*{}**@{-},
 <-0.5mm,0.5mm>*{};<-1.9mm,1.9mm>*{}**@{-},
 <2.3mm,2.3mm>*{\circ};<-2.3mm,2.3mm>*{}**@{},
 <1.8mm,2.8mm>*{};<0mm,4.9mm>*{}**@{-},
 <2.8mm,2.9mm>*{};<4.6mm,4.9mm>*{}**@{-},
   <0.49mm,0.49mm>*{};<-2.7mm,2.3mm>*{^1}**@{},
   <-1.8mm,2.8mm>*{};<0mm,5.3mm>*{^2}**@{},
   <-2.8mm,2.9mm>*{};<5.1mm,5.3mm>*{^3}**@{},
 \end{xy}\Ea, \ \ \ \ \
 \Ba{c}\begin{xy}
 <0mm,0mm>*{\circ};<0mm,0mm>*{}**@{},
 <0mm,0.69mm>*{};<0mm,3.0mm>*{}**@{-},
 <0.39mm,-0.39mm>*{};<2.4mm,-2.4mm>*{}**@{-},
 <-0.35mm,-0.35mm>*{};<-1.9mm,-1.9mm>*{}**@{-},
 <-2.4mm,-2.4mm>*{\circ};<-2.4mm,-2.4mm>*{}**@{},
 <-2.0mm,-2.8mm>*{};<0mm,-4.9mm>*{}**@{-},
 <-2.8mm,-2.9mm>*{};<-4.7mm,-4.9mm>*{}**@{-},
    <0.39mm,-0.39mm>*{};<3.3mm,-4.0mm>*{^3}**@{},
    <-2.0mm,-2.8mm>*{};<0.5mm,-6.7mm>*{^2}**@{},
    <-2.8mm,-2.9mm>*{};<-5.2mm,-6.7mm>*{^1}**@{},
 \end{xy}\Ea
\ = \
 \Ba{c}\begin{xy}
 <0mm,0mm>*{\circ};<0mm,0mm>*{}**@{},
 <0mm,0.69mm>*{};<0mm,3.0mm>*{}**@{-},
 <0.39mm,-0.39mm>*{};<2.4mm,-2.4mm>*{}**@{-},
 <-0.35mm,-0.35mm>*{};<-1.9mm,-1.9mm>*{}**@{-},
 <2.4mm,-2.4mm>*{\circ};<-2.4mm,-2.4mm>*{}**@{},
 <2.0mm,-2.8mm>*{};<0mm,-4.9mm>*{}**@{-},
 <2.8mm,-2.9mm>*{};<4.7mm,-4.9mm>*{}**@{-},
    <0.39mm,-0.39mm>*{};<-3mm,-4.0mm>*{^1}**@{},
    <-2.0mm,-2.8mm>*{};<0mm,-6.7mm>*{^2}**@{},
    <-2.8mm,-2.9mm>*{};<5.2mm,-6.7mm>*{^3}**@{},
 \end{xy}\Ea,\ \ \ \ \ \
 \begin{xy}
 <0mm,2.47mm>*{};<0mm,0.12mm>*{}**@{-},
 <0.5mm,3.5mm>*{};<2.2mm,5.2mm>*{}**@{-},
 <-0.48mm,3.48mm>*{};<-2.2mm,5.2mm>*{}**@{-},
 <0mm,3mm>*{\circ};<0mm,3mm>*{}**@{},
  <0mm,-0.8mm>*{\circ};<0mm,-0.8mm>*{}**@{},
<-0.39mm,-1.2mm>*{};<-2.2mm,-3.5mm>*{}**@{-},
 <0.39mm,-1.2mm>*{};<2.2mm,-3.5mm>*{}**@{-},
     <0.5mm,3.5mm>*{};<2.8mm,5.7mm>*{^2}**@{},
     <-0.48mm,3.48mm>*{};<-2.8mm,5.7mm>*{^1}**@{},
   <0mm,-0.8mm>*{};<-2.7mm,-5.2mm>*{^1}**@{},
   <0mm,-0.8mm>*{};<2.7mm,-5.2mm>*{^2}**@{},
\end{xy}
\  = \
\begin{xy}
 <0mm,-1.3mm>*{};<0mm,-3.5mm>*{}**@{-},
 <0.38mm,-0.2mm>*{};<2.0mm,2.0mm>*{}**@{-},
 <-0.38mm,-0.2mm>*{};<-2.2mm,2.2mm>*{}**@{-},
<0mm,-0.8mm>*{\circ};<0mm,0.8mm>*{}**@{},
 <2.4mm,2.4mm>*{\circ};<2.4mm,2.4mm>*{}**@{},
 <2.77mm,2.0mm>*{};<4.4mm,-0.8mm>*{}**@{-},
 <2.4mm,3mm>*{};<2.4mm,5.2mm>*{}**@{-},
     <0mm,-1.3mm>*{};<0mm,-5.3mm>*{^1}**@{},
     <2.5mm,2.3mm>*{};<5.1mm,-2.6mm>*{^2}**@{},
    <2.4mm,2.5mm>*{};<2.4mm,5.7mm>*{^2}**@{},
    <-0.38mm,-0.2mm>*{};<-2.8mm,2.5mm>*{^1}**@{},
    \end{xy}.
\end{equation}

For the purposes of this paper we will define the properad of non-unital Frobenius algebras to be
\[
 \cF rob :=\cF rob_2.
\]
For example, the cohomology $H^2(\Sigma)$ of any closed Riemann surface $\Sigma$ is a Frobenius algebra in this sense.
Comparing with section \ref{sec:explicit Koszul dual of Lob} we see that the properad $\cF rob$ is isomorphic to the Koszul dual properad of $\LoB$, up to a degree shift
\[
 \LoB^\Koz \cong \cF rob^*\{1\}.
\]

By Koszul duality theory of properads \cite{Va}, one hence obtains from Theorem {\ref{2: Theorem on Koszulness}} the following result.

 \begin{corollary}\label{2: corollary on Frob}
 The properad of non-unital (symmetric) Frobenius algebras $\cF rob$ is Koszul.
 \end{corollary}

 By adding the additional relation
 \[
 \xy
 (0,0)*{\circ}="a",
(0,6)*{\circ}="b",
(3,3)*{}="c",
(-3,3)*{}="d",
 (0,9)*{}="b'",
(0,-3)*{}="a'",
\ar@{-} "a";"c" <0pt>
\ar @{-} "a";"d" <0pt>
\ar @{-} "a";"a'" <0pt>
\ar @{-} "b";"c" <0pt>
\ar @{-} "b";"d" <0pt>
\ar @{-} "b";"b'" <0pt>
\endxy =0
 \]
(which is automatic for $d$ odd) to the presentation of $\Frob_d$ we obtain the properad(s) of involutive Frobenius algebras $\invFrob_d$. They are Koszul dual to the operads governing degree shifted Lie bialgebras (cf. Remark \ref{rem:degreeshifted}), and in particular $\LieBi^\Koz\cong \invcoFrobtwo\{1\}$. It then follows from the Koszulness of $\LieBi$ (and its degre shifted relatives) that the properads $\invFrob_d$ are Koszul, as noted in \cite{JF}.

 \sip

The properad $uc\cF rob$ of unital-counital Frobenius algebras is, by definition, a quotient of the free properad
generated by degree zero corollas $\Ba{c}
\xy
 <0mm,-2mm>*{\circ};<0mm,2mm>*{}**@{-},
 \endxy\Ea
$ (unit),
 $\Ba{c}
\xy
 <0mm,2mm>*{\circ};<0mm,-2mm>*{}**@{-},
 \endxy\Ea
$ (counit),
$
 \begin{xy}
 <0mm,-0.55mm>*{};<0mm,-2.5mm>*{}**@{-},
 <0.5mm,0.5mm>*{};<2.2mm,2.2mm>*{}**@{-},
 <-0.48mm,0.48mm>*{};<-2.2mm,2.2mm>*{}**@{-},
 <0mm,0mm>*{\circ};<0mm,0mm>*{}**@{},
 <0.5mm,0.5mm>*{};<2.7mm,2.8mm>*{^{_2}}**@{},
 <-0.48mm,0.48mm>*{};<-2.7mm,2.8mm>*{^{_1}}**@{},
 \end{xy}
=
\begin{xy}
 <0mm,-0.55mm>*{};<0mm,-2.5mm>*{}**@{-},
 <0.5mm,0.5mm>*{};<2.2mm,2.2mm>*{}**@{-},
 <-0.48mm,0.48mm>*{};<-2.2mm,2.2mm>*{}**@{-},
 <0mm,0mm>*{\circ};<0mm,0mm>*{}**@{},
 <0.5mm,0.5mm>*{};<2.7mm,2.8mm>*{^{_1}}**@{},
 <-0.48mm,0.48mm>*{};<-2.7mm,2.8mm>*{^{_2}}**@{},
 \end{xy}
 $ (graded co-commutative comultiplication)
 and
 $
 \begin{xy}
 <0mm,0.66mm>*{};<0mm,3mm>*{}**@{-},
 <0.39mm,-0.39mm>*{};<2.2mm,-2.2mm>*{}**@{-},
 <-0.35mm,-0.35mm>*{};<-2.2mm,-2.2mm>*{}**@{-},
 <0mm,0mm>*{\circ};<0mm,0mm>*{}**@{},
   <0.39mm,-0.39mm>*{};<2.9mm,-4mm>*{^{_2}}**@{},
   <-0.35mm,-0.35mm>*{};<-2.8mm,-4mm>*{^{_1}}**@{},
\end{xy}=
\begin{xy}
 <0mm,0.66mm>*{};<0mm,3mm>*{}**@{-},
 <0.39mm,-0.39mm>*{};<2.2mm,-2.2mm>*{}**@{-},
 <-0.35mm,-0.35mm>*{};<-2.2mm,-2.2mm>*{}**@{-},
 <0mm,0mm>*{\circ};<0mm,0mm>*{}**@{},
   <0.39mm,-0.39mm>*{};<2.9mm,-4mm>*{^{_1}}**@{},
   <-0.35mm,-0.35mm>*{};<-2.8mm,-4mm>*{^{_2}}**@{},
\end{xy}
$ (graded commutative multiplication)
modulo the ideal generated by the relations \eqref{equ:Frobrelations}
and the additional relations
\Beq\label{2: uFrob relations}
\begin{xy}
<0mm,-0.55mm>*{};<0mm,-2.5mm>*{}**@{-},
<0.5mm,0.5mm>*{};<2.2mm,2.2mm>*{}**@{-},
<-0.48mm,0.48mm>*{};<-2.5mm,2.5mm>*{}**@{-},
<0mm,0mm>*{\circ};
<-2.78mm,2.78mm>*{\circ};
 \end{xy} \ - \
 \begin{xy}
<0mm,-2.2mm>*{};<0mm,2.2mm>*{}**@{-},
 \end{xy}\ =0 \ \ \ , \ \ \
 \begin{xy}
<0mm,0.55mm>*{};<0mm,2.5mm>*{}**@{-},
<0.5mm,-0.5mm>*{};<2.2mm,-2.2mm>*{}**@{-},
<-0.48mm,-0.48mm>*{};<-2.5mm,-2.5mm>*{}**@{-},
<0mm,0mm>*{\circ};
<-2.78mm,-2.78mm>*{\circ};
 \end{xy} \ - \
 \begin{xy}
<0mm,-2.2mm>*{};<0mm,2.2mm>*{}**@{-}
 \end{xy}\ =0 \ \ \ .
\Eeq
where the vertical line\ $\begin{xy}
<0mm,-2.2mm>*{};<0mm,2.2mm>*{}**@{-}
 \end{xy}$ \ stands for the unit in the properad $\cF rob$.
Similarly one defines a properad $u\cF rob$ of unital Frobenius algebras, and a properad
$c\cF rob$ of counital algebras. Clearly, $u\cF rob$ and $c\cF rob$ are subproperads of $uc\cF rob$.

 \begin{theorem}
 The properads\ \, $u\cF rob$, $c\cF rob$ and $uc\cF rob$ are Koszul.
 \end{theorem}
 \begin{proof} By curved Koszul duality theory \cite{HM}, it is enough to prove Koszulness
 of the associated quadratic properads, $qu\cF rob$, $qc\cF rob$ and $quc\cF rob$,
obtained from $u\cF rob$, $c\cF rob$ and, respectively, $uc\cF rob$ by replacing  inhomogeneous relations
(\ref{2: uFrob relations}) by the following ones (see \S 3.6 in \cite{HM}),
$$
\begin{xy}
<0mm,-0.55mm>*{};<0mm,-2.5mm>*{}**@{-},
<0.5mm,0.5mm>*{};<2.2mm,2.2mm>*{}**@{-},
<-0.48mm,0.48mm>*{};<-2.5mm,2.5mm>*{}**@{-},
<0mm,0mm>*{\circ};
<-2.78mm,2.78mm>*{\circ};
 \end{xy} \ = \ 0 \ \ \ , \ \ \
 \begin{xy}
<0mm,0.55mm>*{};<0mm,2.5mm>*{}**@{-},
<0.5mm,-0.5mm>*{};<2.2mm,-2.2mm>*{}**@{-},
<-0.48mm,-0.48mm>*{};<-2.5mm,-2.5mm>*{}**@{-},
<0mm,0mm>*{\circ};
<-2.78mm,-2.78mm>*{\circ};
 \end{xy} \ = \ 0
$$
so that we have  decompositions into direct sums of $\bS$-bimodules,
$$
qu\cF rob= \mathrm{span}\left\langle \hspace{-1.0mm} \Ba{c}
\resizebox{7mm}{!}{\xy
(0,-5)*{\circ};
(0,0)*+{a}*\cir{}
**\dir{-};
(0,6)*{_{co\cC om}};
(0,0)*+{a}*\cir{}
**\dir{-};
\endxy}
\Ea \hspace{-1.0mm} \right\rangle \  \oplus\  \cF rob
 \ \ \ \ , \ \ \ \
 qc\cF rob= \mathrm{span}\left\langle \hspace{-1.0mm}\Ba{c}
\resizebox{5mm}{!}{\xy
(0,5)*{\circ};
(0,0)*+{_a}*\cir{}
**\dir{-};
(0,-6)*{_{\cC om}};
(0,0)*+{_a}*\cir{}
**\dir{-};
\endxy}
\Ea  \hspace{-1.0mm} \right\rangle\ \oplus\   \cF rob
$$
and
$$
 quc\cF rob=
  \mathrm{span}\left\langle \hspace{-1.0mm} \Ba{c}
\resizebox{7mm}{!}{\xy
(0,-5)*{\circ};
(0,0)*+{_a}*\cir{}
**\dir{-};
(0,6)*{_{co\cC om}};
(0,0)*+{_a}*\cir{}
**\dir{-};
\endxy}
\Ea \hspace{-1.0mm} \right\rangle \
 \oplus \ \mathrm{span}\left\langle \hspace{-1.0mm}\Ba{c}
\resizebox{5mm}{!}{\xy
(0,5)*{\circ};
(0,0)*+{_a}*\cir{}
**\dir{-};
(0,-6)*{_{\cC om}};
(0,0)*+{_a}*\cir{}
**\dir{-};
\endxy}
\Ea  \hspace{-1.0mm} \right\rangle\ \ \oplus \ \
 \cF rob.
$$
where $\resizebox{4mm}{!}{
\xy
(0,-4)*{};
(0,0)*+{_a}*\cir{}
**\dir{-};
(0,4)*{};
(0,0)*+{_a}*\cir{}
**\dir{-};
\endxy}
$ stands for the graph given in (\ref{2: a weight in LieB-Koszul}).

\sip

Consider, for example, the properad  $qc\cF rob$ (proofs of Koszulness of properads $qu\cF rob$ and $quc\cF rob$ can be given by a similar argument).
 Its Koszul dual properad $qc\cF rob^!=:qc \LoB$
is generated by the properads $\LoB$ and   $\left\langle\hspace{-1.8mm}  \Ba{c}\xy
 <0mm,2mm>*{\circ};<0mm,-2mm>*{}**@{-},
 \endxy\Ea \hspace{-1.8mm} \right\rangle$  modulo the following relation,
 $$
  \begin{xy}
<0mm,0.55mm>*{};<0mm,3.5mm>*{\circ}**@{-},
<0.5mm,-0.5mm>*{};<2.2mm,-2.2mm>*{}**@{-},
<-0.48mm,-0.48mm>*{};<-2.5mm,-2.5mm>*{}**@{-},
<0mm,0mm>*{\circ};
<-2.78mm,-2.78mm>*{};
 \end{xy} \ = \ 0.
 $$
We have,
$$
qc\LoB^\Koz= (qc\cF rob)^*\{1\} \cong
 \mathrm{span}\left\langle \hspace{-1.0mm}\Ba{c}
\resizebox{4mm}{!}{\xy
(0,5)*{\circ};
(0,0)*+{_a}*\cir{}
**\dir{-};
(0,-6)*{_{\caL ie^\Koz}};
(0,0)*+{_a}*\cir{}
**\dir{-};
\endxy}
\Ea  \hspace{-1.0mm} \right\rangle\
\ \oplus \ \LoB^\Koz  .
 $$
It will suffice to show that the properad $qc\LoB$ is Koszul. To this end consider the dg properad $\Omega(qc\LoB^\Koz)$ which is a free properad generated by corollas
 (\ref{2: generating corollas of LoB infty}) and the following ones,
\Beq\label{2: 0-generators of qucLoB}
\Ba{c}
\resizebox{10mm}{!}{\xy
(-7,-8.5)*{_{_1}},
(-3.5,-8.5)*{_{_2}},
(7.5,-8.5)*{_{_n}},
(0,0)*+{_a}*\cir{}="o",
(-7,-7)*{}="d1",
(-3.5,-7)*{}="d2",
(1.6,-5.5)*{...},
(7,-7)*{}="d4",
\ar @{-} "o";"d1" <0pt>
\ar @{-} "o";"d2" <0pt>
\ar @{-} "o";"d4" <0pt>
\endxy}
\Ea=(-1)^\sigma
\Ba{c}
\resizebox{14mm}{!}{\xy
(-8.5,-8.5)*{_{_{\sigma(1)}}},
(-3,-8.5)*{_{_{\sigma(2)}}},
(7.9,-8.5)*{_{_{\sigma(n)}}},
(0,0)*+{_a}*\cir{}="o",
(-7,-7)*{}="d1",
(-3.5,-7)*{}="d2",
(1.6,-5.5)*{...},
(7,-7)*{}="d4",
\ar @{-} "o";"d1" <0pt>
\ar @{-} "o";"d2" <0pt>
\ar @{-} "o";"d4" <0pt>
\endxy}
\Ea
\Eeq
where $a\geq 0$, $n\geq 1$ and $\sigma\in \bS_n$ is an arbitrary permutation. The differential is given
on corollas  (\ref{2: generating corollas of LoB infty}) by the standard formula (\ref{2: d on Lie inv infty}) and on $(0,n)$-generators by
$$
d\Ba{c}
\resizebox{10mm}{!}{\xy
(-7,-8.5)*{_{_1}},
(-3.5,-8.5)*{_{_2}},
(7.5,-8.5)*{_{_n}},
(0,0)*+{_a}*\cir{}="o",
(-7,-7)*{}="d1",
(-3.5,-7)*{}="d2",
(1.6,-5.5)*{...},
(7,-7)*{}="d4",
\ar @{-} "o";"d1" <0pt>
\ar @{-} "o";"d2" <0pt>
\ar @{-} "o";"d4" <0pt>
\endxy}
\Ea
=
\sum_{a=b+c+l-1}\sum_{
[n]=J_1\sqcup J_2\atop \# J_1\geq 1} \pm \hspace{-6mm}
\Ba{c}
%
%
\resizebox{18mm}{!}{\xy
(0,0)*+{_b}*\cir{}="b",
(10,10)*+{_c}*\cir{}="c",
%
(-9,6)*{}="1",
(-7,6)*{}="2",
(-2,6)*{}="3",
(-4,-6)*{}="-1",
(-2,-6)*{}="-2",
(4,-6)*{}="-3",
(1,-5)*{...},
(0,-8)*{\underbrace{\ \ \ \ \ \ \ \ }},
(0,-11)*{_{J_1}},
(6,16)*{}="1'",
(8,16)*{}="2'",
(14,16)*{}="3'",
(11,6)*{}="-1'",
(16,6)*{}="-2'",
(18,6)*{}="-3'",
(13.5,6)*{...},
(15,4)*{\underbrace{\ \ \ \ \ \ \ }},
(15,1)*{_{J_2}},
%
(0,2)*-{};(8.0,10.0)*-{}
**\crv{(0,10)};
(0.5,1.8)*-{};(8.5,9.0)*-{}
**\crv{(0.4,7)};
(1.5,0.5)*-{};(9.1,8.5)*-{}
**\crv{(5,1)};
(1.7,0.0)*-{};(9.5,8.6)*-{}
**\crv{(6,-1)};
(5,5)*+{...};
\ar @{-} "b";"-1" <0pt>
\ar @{-} "b";"-2" <0pt>
\ar @{-} "b";"-3" <0pt>
\ar @{-} "c";"-1'" <0pt>
\ar @{-} "c";"-2'" <0pt>
\ar @{-} "c";"-3'" <0pt>
\endxy}
\Ea
$$
where $l$ counts the number of internal edges connecting
the two vertices
on the right-hand side. There is a natural morphism of properads
$$
\Omega(qc\LoB^\Koz) \lon qc\LoB,
$$
which is a quasi-isomorphism if and only if  $qc\LoB$ is Koszul. Thus to prove Koszulness
of $qc\LoB$ it is enough to establish an isomorphism $H^\bu(\Omega(qc\LoB^\Koz)) \cong qc\LoB$ of $\bS$-bimodules.

\sip

To do this, one may closely follow the proof of Theorem {\ref{2: Theorem on Koszulness}}, adjusting it slightly so as to allow for the additional $(0,n)$-ary generators.
First, we define a properad $\tilde \cP$ which is generated by the properad $\cP$ of section \ref{2: subsection on P}, together with an additional generator of arity $(0,1)$, in pictures $\hspace{-1.8mm}  \Ba{c}\xy
 <0mm,2mm>*{\circ};<0mm,-2mm>*{}**@{-},
 \endxy\Ea \hspace{-1.8mm}$, with the additional relations
\begin{align*}
   \begin{xy}
<0mm,0.55mm>*{};<0mm,3.5mm>*{\circ}**@{-},
<0.5mm,-0.5mm>*{};<2.2mm,-2.2mm>*{}**@{-},
<-0.48mm,-0.48mm>*{};<-2.5mm,-2.5mm>*{}**@{-},
<0mm,0mm>*{\circ};
<-2.78mm,-2.78mm>*{};
 \end{xy} &= \ 0
 &
 \xy
(0,-1.4)*{\bu}="0",
 (0,2.4)*{\circ}="1",
(0,-4.5)*{}="d2",
\ar @{-} "0";"1" <0pt>
\ar @{-} "0";"d2" <0pt>
\endxy
&=
0\,.
\end{align*}

The map $\Omega(qc\LoB^\Koz)\to qc\LoB$ clearly factors through $\tilde\cP$
\Beq\label{equ:Koszul factoring}
 \Omega(qc\LoB^\Koz) \to \tilde\cP \to qc\LoB
\Eeq
and it suffices to show that both of the above maps are quasi-isomorphisms.
Consider first the left-hand map.
The fact that this map is a quasi-isomorphism may be proven by copying the proof of Theorem \ref{2: proposition on nu quasi-iso}, except that now the functor $F$ (as in section \ref{2: subsection on P}) is applied not to the
cobar construction, $\Omega_{\frac{1}{2}} ( \LB^{\text{!`}}_{\frac{1}{2}}))$
but to $\Omega_{\frac{1}{2}} ( qc\LB^{\text{!`}}_{\frac{1}{2}}))$. Here
\[
 qc\LB_{\frac{1}{2}}(m,n)=
 \begin{cases}
  \LB_{\frac{1}{2}}(m,n) & \text{if $(m,n)\neq (0,1)$} \\
  \K & \text{if $(m,n)= (0,1)$}
 \end{cases}
\]
is the $\frac 1 2$-prop governing Lie bialgebras with a counit operation killed by the cobracket.
More concretely, $qc\LB^{\text{!`}}_{\frac{1}{2}}(m,n)$ is the same as $\LB^{\text{!`}}_{\frac{1}{2}}(m,n)$ in all arities $(m,n)$ with $m,n>0$, but $qc\LB^{\text{!`}}_{\frac{1}{2}}(0,n)$ is one-dimensional, the extra operations corresponding to corollas
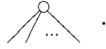
\[
 \Ba{c}
\resizebox{10mm}{!}{\xy
(-7,-8.5)*{_{}},
(-3.5,-8.5)*{_{}},
(7.5,-8.5)*{_{}},
(0,0)*+{}*\cir{}="o",
(-7,-7)*{}="d1",
(-3.5,-7)*{}="d2",
(1.6,-5.5)*{...},
(7,-7)*{}="d4",
\ar @{-} "o";"d1" <0pt>
\ar @{-} "o";"d2" <0pt>
\ar @{-} "o";"d4" <0pt>
\endxy}
\Ea\, .
\]
One can check\footnote{The piece of the $\frac 1 2$-prop $\Omega_{\frac{1}{2}} ( qc\LB^{\text{!`}}_{\frac{1}{2}}))$ involving the additional generators is isomorphic to the complex $(E_1^{\mathsf{Lie}+},d_1^{\mathsf{Lie}+})$ from \cite{Me2} (see page 344). According to loc. cit. its cohomology is one-dimensional.} that the $\frac{1}{2}$-prop $qc\LB^{\text{!`}}_{\frac{1}{2}}$ is Koszul, i.~e., that
\[
 H(\Omega_{\frac{1}{2}}( qc\LB^{\text{!`}}_{\frac{1}{2}}))\cong qc\LB_{\frac{1}{2}}.
\]
The properad $\tilde\cP$ is obtained by applying the exact functor $F$ to this $\frac 1 2$-prop, and hence, by essentially the same arguments as in the proof of Theorem \ref{2: nu quasi-iso to P} the left-hand map of \eqref{equ:Koszul factoring} is a quasi-isomorphism.

\sip

Next consider the right hand map of \eqref{equ:Koszul factoring}. It can be shown to be a quasi-isomorphism along the lines of the proof of Theorem \ref{2: Theorem on Koszulness}.
Again, it is clear that the degree zero cohomology of $\tilde\cP$ is $qc\LoB$, so it will suffice to show that $H^{>0}(\tilde\cP)=0$.
First, let $\widetilde{\cP\cP}$ be the prop generated by the properad $\tilde\cP$. As a dg $\bS$-bimodule it is isomorphic to (cf. \eqref{2: basis in PP})
\begin{equation*}
   \tilde W(n,m):= \bigoplus_{N,M} \left( \caL ieP(n,N) \otimes V^{\otimes N} \otimes \K^{\otimes M}\otimes \caL ieCP(N+M,m)
   \right)_{\bS_N\times \bS_M}
\end{equation*}
where $V$ is as in \eqref{2: basis in PP}.
The above complex $\tilde W(n,m)$ is a direct summand of the complex (cf. \eqref{equ:V_mn})
\[
\tilde V_{n,m}
:=
\bigoplus_{N,M} \left( \caL ieP(n,N) \otimes V^{\otimes N} \otimes \K^{\otimes M} \otimes \cA ssCP(N+M,m)
\right)_{\bS_N\times \bS_M}
\]
by arguments similar to those following \eqref{equ:V_mn}.
Then again by the Koszulness results of section \ref{sec:extracomplexes} it follows that the above complex has no cohomology in positive degrees, hence neither can $\widetilde{\cP\cP}$ have cohomology in positive degrees. Hence we can conclude that the properad $qc\LoB$ is Koszul.

\end{proof}

\bip

{\large
\section{\bf Deformation complexes}\label{3: Def complexes}
}

As one application of the Koszulness of $\LoB$ and $\Frob$ we obtain minimal models $\LoB_\infty=\Omega(\LoB^\Koz)=\Omega(\coFrob\{1\})$ and $\Frob_\infty=\Omega(\invcoLieBi\{1\})$ of these properads and hence minimal models for their deformation complexes and for the deformation complexes of their algebras.

\sip

\subsection{\bf A deformation complex of an involutive Lie bialgebra}\label{3: Section on Def complex}

According to the general theory \cite{MV}, $\LoB_\infty$-algebra structures in a dg
vector space $(\fg, d)$ can be identified with Maurer-Cartan elements,
  $$
  \mathcal{MC}\left({\mathsf{InvLieB}}(\fg)\right):= \left\{\Ga\in \mathsf{InvLieB}(\fg)
  : |\Ga|=3\ \mbox{and}\  [\Ga,\Ga]_{CE}=0\right\},
  $$
  of a  graded Lie algebra,
\Beq\label{3:  InvLieB(g)}
{\mathsf{InvLieB}}(\fg):=\Def\left(\LoB_\infty \stackrel{0}{\lon} \cE nd_\fg\right)[-2],
\Eeq
which controls deformations of the zero morphism from $\LoB_\infty$ to the endomorphism
properad $\cE nd_\fg=\{\Hom(\fg^{\ot n}, \fg^{\ot m})\}$.
 As a $\Z$-graded vector space ${\mathsf{InvLieB}}(\fg)$ can be identified with the vector
space of homomorphisms of $\bS$-bimodules,
\Beqrn
{\mathsf{InvLieB}}(\fg)&=&\Hom_\bS\left((\LoB)^\Koz,  \cE nd_\fg\right)[-2]\\
&=&\prod_{a\geq0, m,n\geq 1\atop
m+n+a\geq 3} \Hom_{\bS_m\times \bS_n}\left(\sgn_n\ot\sgn_m[m+n+2a-2],
\Hom(\fg^{\ot  n}, \fg^{\ot m}) \right)[-2]\\
&=& \prod_{a\geq 0, m,n\geq 1\atop
m+n+a\geq 3} \Hom(\odot^n(\fg[-1]),  \odot^m(\fg[-1]) )[-2a]\\
&\subset& \displaystyle  \widehat{\odot^\bu}\left(\fg[-1]\oplus \fg^*[-1]\oplus
\K[-2]\right)\simeq \K[[\eta^i,\psi_i,\hbar]]
\Eeqrn

where $ \hbar$ a formal  parameter of degree $2$  (a basis vector of the summand $\K[-2]$  above), and, for a basis
$(e_1, e_2,\ldots, e_i,\ldots  )$ in $\fg$ and the associated dual basis
$(e^1, e^2,\ldots, e^i,\ldots  )$ in $\fg^*$ we set $\eta^i:=s\ e^i$,
$\psi_i:=s\ e_i$,
where $s: V\rar V[-1]$ is the suspension map. Therefore the Lie algebra ${\mathsf{InvLieB}}(\fg)$
has a canonical structure of a module over the algebra $\K[[\hbar]]$;
moreover, for finite dimensional $\alg g$ its elements can be identified with
formal power series\footnote{In fact, this is true for a class of infinite-dimensional vector spaces. Consider a
category
of graded vector spaces which are inverse limits of finite dimensional ones (with the corresponding topology and  with the completed tensor product), and also a category
of graded vector spaces which are direct limits of finite dimensional ones. If $\alg g$  belongs to one of these categories, then ${\alg g}^*$ belongs (almost by definition) to the other, and we have isomorphisms of the type $({\alg g} \ot {\alg g})^*= {\alg g}^* \ot {\alg g }^*$ which are required for the ``local coordinate" formulae to work.}, $f$, in variables $\psi_i$, $\eta^i$ and $\hbar$, which  satisfy the ``boundary"
conditions,
\Beq\label{Appendix: boundary conditions}
f(\psi,\eta,\hbar)|_{\psi_i=0}=0,\ \ \ \ f(\psi,\eta,\hbar)|_{\eta^i=0}=0,\ \ \ \
f(\psi,\eta,\hbar)|_{\hbar=0}\in I^3
\Eeq
where $I$ is the maximal ideal in $\K[[\psi,\eta]]$. The Lie brackets
in  ${\mathsf{InvLieB}}(\fg)$ can be read off either from the coproperad structure in
$(\LoB)^\Koz$
or directly from the formula (\ref{2: d on Lie inv infty}) for the differential,
and are given explicitly by (cf.\ \cite{DCTT}),
\Beq\label{3: brackets in InvLieB(g)}
[f,g]_\hbar:=f *_\hbar g - (-1)^{|f||g|} g *_\hbar f
\Eeq
where (up to  Koszul signs),
$$
f *_\hbar g:=\sum_{k=0}^\infty \frac{\hbar^{k-1}}{k!}\sum_{i_1,\ldots, i_k}
\pm \frac{\p^ k f}{\p \eta^{i_1}\cdots \eta^{i_k}}\frac{\p^ k g}{\p\psi_{i_1}\cdots
\p\psi_{i_k}}
$$
is an associative product. Note that the differential
$d_\fg$ in $\fg$ gives rise to a quadratic element, $D_\fg=\sum_{i,j}\pm d_j^i
\psi_i\eta^j$, of homological degree $3$ in $\K[[\eta^i,\psi_i,\hbar]]$, where $d_j^i$ are the
structure constants of $d_\fg$ in the chosen basis,
 $d_\fg(e_i)=:\sum_{j} d_i^j e_j$.

\bip

Finally, we can identify $\LoB_\infty$ structures in a finite dimensional dg vector space $(\fg,d_\fg)$ with
a homogenous formal power series,
$$
\Ga:= D_\fg + f \in  \K[[\eta^i,\psi_i,\hbar]],
$$
of homological degree 3 such that
\Beq\label{Appendix: equation for Gamma-h}
\Ga \star_\hbar \Ga=\sum_{k=0}^\infty \frac{\hbar^{k-1}}{k!}\sum_{i_1,\ldots, i_k}
\pm \frac{\p^ k \Ga}{\p \eta^{i_1}\cdots \eta^{i_k}}\frac{\p^ k \Ga}{\p\psi_{i_1}\cdots
\p\psi_{i_k}}=0,
\Eeq
 and the summand
 $f$ satisfies boundary conditions (\ref{Appendix: boundary conditions}).

\bip

For example, let
$$
\left(\bigtriangleup:V\rar \wedge^2 V,\ \ \  [\ ,\ ]:\wedge^2V \rar V\right)
$$
be a Lie bialgebra structure in a  vector space $V$ which we assume for simplicity to be concentrated in degree $0$. Let
$C_{ij}^k$ and $\Phi_k^{ij}$ be the associated structure constants,
$$
[x_i,x_j]=:\sum_{k\in I} C_{ij}^k x_k,\ \ \ \ \bigtriangleup(x_k)=:\sum_{i,j\in I} \Phi_k^{ij} x_i\wedge x_j.
$$
Then it is easy to check that all the involutive Lie bialgebra axioms (\ref{R for LieB}) get encoded into a single equation $\Ga *_\hbar \Ga=0$ for
$
\Ga:=\sum_{i,j,k\in I}\left( C_{ij}^k \psi_k\eta^i\eta^j + \Phi_k^{ij}\eta^k \psi_i\psi_j
\right).
$
\sip

Note that all the above formulae taken modulo the ideal generated by the formal variable $\hbar$
give us a Lie algebra,
\Beq\label{3:  LieB(g)}
{\mathsf{LieB}}(\fg):=\Def\left(\caL ie\cB_\infty \stackrel{0}{\lon} \cE nd_\fg\right)[-2]\cong \K[[\psi_i,\eta^i]]
\Eeq
controlling the deformation theory of (not-necessarily involutive) Lie bialgebra structures in a dg space $\fg$.  Lie brackets in (\ref{3:  LieB(g)}) are given in coordinates by the standard Poisson formula,
\Beq\label{3: Poisson brackets in LieB(g)}
\{f, g\}=\sum_{i\in I} (-1)^{|f||\eta^i|} \frac{\p f}{\p  \psi_i} \frac{\p g}{\p \eta^i} -
 (-1)^{|f||\psi_i|} \frac{\p f}{\p \eta^i} \frac{\p g}{\p \psi_i}
\Eeq
for any $f,g\in \K[[\psi_i,\eta^i]]$. Formal power series, $f\in  \K[[\psi_i,\eta^i]]$, which have homological degree $3$ and satisfy the equations,
$$
\{f,f\}=0, \ \ \ \ f(\psi,\eta)|_{\psi_i=0}=0,\ \ \ \ f(\psi,\eta)|_{\eta^i=0}=0,
$$
are in one-to-one correspondence with strongly homotopy Lie bialgebra structures in a finite dimensional dg vector space $\fg$.

\subsection{Deformation complexes of properads} \label{sec:defcomplexes}
The deformation complex of a properad $\cP$ is by definition the dg Lie algebra $\Der(\tilde P)$ of derivations of a cofibrant resolution $\tilde \cP\stackrel{\sim}{\to}\cP$. It may be identified as a complex with the deformation complex of the identity map $\tilde \cP\to \tilde \cP$ (which controls deformations of $\cP$-algebras) up to a degree shift:
\[
 \Der(\tilde \cP) \cong \Def(\tilde \cP\to \tilde \cP)[1].
\]
Note however that both $\Der(\cP)$ and $\Def(\tilde \cP\to \tilde \cP)$ have natural dg Lie (or $\Lie_\infty$) algebra structures that are \emph{not} preserved by the above map. Furthermore, there is a quasi-isomorphism of dg Lie algebras
\begin{equation}\label{equ:Defsimpl}
 \Def(\tilde \cP\to \tilde \cP)\to \Def(\tilde \cP\to \cP)
\end{equation}

The zeroth cohomology $H^0(\Der(\tilde \cP))$ is of particular importance. It is a differential graded Lie algebra whose elements act on the space of $\tilde \cP$ algebra structures on any vector space. We shall see that in the examples we are interested in this dg Lie algebra is very rich, and that it acts non-trivially in general.

\sip
Using the Koszulness of the properads $\LieBi$, $\invFrob$ from \cite{MaVo, Ko} and the Koszulness of $\invLieBi$ and $\Frob$ from Theorem \ref{2: Theorem on Koszulness} and Corollary \ref{2: corollary on Frob} we can write down the following models for the deformation complexes.
\begin{align*}
\Der(\LieBi_\infty) &=
\prod_{n,m\geq 1} \Hom_{\bS_n\times \bS_m}(\invcoFrob\{1\}(n,m), \LieBi_\infty(n,m))[1]
\\
&\cong \prod_{n,m\geq 1} (\LieBi_\infty(n,m) \otimes \sgn_n\otimes \sgn_m)^{\bS_n\times \bS_m}[-n-m+1]
\\
\Der(\LoB_\infty) &=  \prod_{n,m\geq 1} \Hom_{\bS_n\times \bS_m}(\coFrob\{1\}(n,m), \LoB_\infty(n,m))[1]
\\
&\cong \prod_{n,m\geq 1} (\LoB_\infty(n,m)\otimes \sgn_n\otimes \sgn_m )^{\bS_n\times \bS_m} [[\hbar]][-n-m+1]
 \\
\Der(\Frob_\infty) &=\prod_{n,m\geq 1} \Hom_{\bS_n\times \bS_m\geq 1}(\invcoLieBi\{1\}(n,m), \Frob_\infty(n,m))[1]
\\
\Der(\invFrob_\infty) &=\prod_{n,m\geq 1} \Hom_{\bS_n\times \bS_m\geq 1}(\coLieBi\{1\}(n,m), \invFrob_\infty(n,m))[1]
\end{align*}
Here $\hbar$ is a formal variable of degree 2.
Each of the models on the right has a natural combinatorial interpretation as a graph complex, cf. also \cite[section 1.7 ]{MaVo}.
For example $\Der(\LieBi_\infty)$ may be interpreted as a complex of directed  graphs which have incoming and outgoing legs but have no closed paths of directed edges. The differential is obtained by splitting vertices and by attaching new vertices at one of the external legs, see Figure \ref{fig:hairyGC}.

\begin{figure}
 \centering
 \begin{align*} &
\resizebox{15mm}{!}{ \xy
(0,0)*{\bu}="d1",
(10,0)*{\bu}="d2",
(-5,-5)*{}="dl",
(5,-5)*{}="dc",
(15,-5)*{}="dr",
(0,10)*{\bu}="u1",
(10,10)*{\bu}="u2",
(5,15)*{}="uc",
(15,15)*{}="ur",
(0,15)*{}="ul",
\ar @{->} "d1";"d2" <0pt>
\ar @{->} "d1";"dl" <0pt>
\ar @{->} "d1";"dc" <0pt>
\ar @{->} "d2";"dc" <0pt>
\ar @{->} "d2";"dr" <0pt>
\ar @{->} "u1";"d1" <0pt>
\ar @{->} "u1";"d2" <0pt>
\ar @{->} "u2";"d2" <0pt>
\ar @{->} "u2";"d1" <0pt>
\ar @{->} "uc";"u2" <0pt>
\ar @{->} "ur";"u2" <0pt>
\ar @{->} "ul";"u1" <0pt>
\endxy}
 %
 &
 \delta \Gamma
 &=
 \delta_{\LieBi_\infty }\Gamma
 \pm
 \sum\Ba{c}
 \resizebox{9mm}{!}{ \xy
 (0,0)*+{\Ga}="Ga",
(-5,5)*{\bu}="0",
(-8,2)*{}="-1",
(-8,8)*{}="1",
(-5,8)*{}="2",
(-2,8)*{}="3",
\ar @{-} "0";"Ga" <0pt>
\ar @{-} "0";"-1" <0pt>
\ar @{-} "0";"1" <0pt>
\ar @{-} "0";"2" <0pt>
\ar @{-} "0";"3" <0pt>
 \endxy}\Ea
 %
 %
  \pm
 \sum\Ba{c}
\resizebox{9mm}{!}{  \xy
 (0,0)*+{\Ga}="Ga",
(-5,-5)*{\bu}="0",
(-8,-2)*{}="-1",
(-8,-8)*{}="1",
(-5,-2)*{}="2",
(-2,-8)*{}="3",
\ar @{-} "0";"Ga" <0pt>
\ar @{-} "0";"-1" <0pt>
\ar @{-} "0";"1" <0pt>
\ar @{-} "0";"2" <0pt>
\ar @{-} "0";"3" <0pt>
 \endxy}\Ea
 &
 %
 \end{align*}
 \caption{\label{fig:hairyGC} A graph interpretation of an element of $\Der(\LieBi_\infty)$, and the pictorial description of the differential.}
\end{figure}

\sip

Similarly, $\Der(\LoB_\infty)$ may be interpreted as a complex of $\hbar$-power series of graphs with weighted vertices. The differential is obtained by splitting vertices and attaching vertices at external legs as indicated in Figure~ \ref{fig:hairyweightedGC}.

\begin{figure}
 \centering
 \begin{align*}
  &
\resizebox{15mm}{!}{  \xy
(0,0)*+{_3}*\cir{}="d1",
(10,0)*+{_2}*\cir{}="d2",
(-5,-5)*{}="dl",
(5,-5)*{}="dc",
(15,-5)*{}="dr",
(0,10)*+{_0}*\cir{}="u1",
(10,10)*+{_3}*\cir{}="u2",
(5,15)*{}="uc",
(15,15)*{}="ur",
(0,15)*{}="ul",
\ar @{->} "d1";"d2" <0pt>
\ar @{->} "d1";"dl" <0pt>
\ar @{->} "d1";"dc" <0pt>
\ar @{->} "d2";"dc" <0pt>
\ar @{->} "d2";"dr" <0pt>
\ar @{->} "u1";"d1" <0pt>
\ar @{->} "u1";"d2" <0pt>
\ar @{->} "u2";"d2" <0pt>
\ar @{->} "u2";"d1" <0pt>
\ar @{->} "uc";"u2" <0pt>
\ar @{->} "ur";"u2" <0pt>
\ar @{->} "ul";"u1" <0pt>
\endxy}
%
 &
 \delta \Gamma
 &=
 \delta_{\LoB_\infty }\Gamma
\pm
 \sum \hbar^{p+k-1}\Ba{c}
\resizebox{10mm}{!}{  \xy
  (-3,4)*+{_k},
 (0,0)*+{\Ga}="Ga",
(-7,7)*+{_p}*\cir{}="0",
(-10,4)*{}="-1",
(-10,11)*{}="1",
(-7,11)*{}="2",
(-4,11)*{}="3",
(-5.3,6.8)*-{};(-0.2,1.5)*-{}
**\crv{(0.4,5)};
(-6.9,5.1)*-{};(-1.4,0.3)*-{}
**\crv{(-5,-1.5)};
(-6.0,5.4)*-{};(-1.4,0.7)*-{}
**\crv{(-4,-0.5)};
%
\ar @{-} "0";"-1" <0pt>
\ar @{-} "0";"1" <0pt>
\ar @{-} "0";"2" <0pt>
\ar @{-} "0";"3" <0pt>
 \endxy}\Ea
 %
  \pm
 \sum \hbar^{p+k-1}\Ba{c}
 \resizebox{10mm}{!}{  \xy
  (-3,-4)*+{_k},
 (0,0)*+{\Ga}="Ga",
(-7,-7)*+{_p}*\cir{}="0",
(-10,-4)*{}="-1",
(-10,-11)*{}="1",
(-7,-11)*{}="2",
(-4,-11)*{}="3",
(-5.3,-6.8)*-{};(-0.2,-1.5)*-{}
**\crv{(0.4,-5)};
(-6.9,-5.1)*-{};(-1.4,-0.3)*-{}
**\crv{(-5,1.5)};
(-6.0,-5.4)*-{};(-1.4,-0.7)*-{}
**\crv{(-4,-0.5)};
%
\ar @{-} "0";"-1" <0pt>
\ar @{-} "0";"1" <0pt>
\ar @{-} "0";"2" <0pt>
\ar @{-} "0";"3" <0pt>
 \endxy}\Ea
 %
 \end{align*}
 \caption{\label{fig:hairyweightedGC} A graph interpretation of an element of $\Der(\LoB_\infty)$, and the pictorial description of the differential.}
\end{figure}

The Lie bracket is combinatorially obtained by inserting graphs into vertices of another.
We leave it to the reader to work out the structure of the graph complexes and the differentials for the complexes $\Der(\Frob_\infty)$ and $\Der(\invFrob_\infty)$.

\sip

The cohomology of all these graph complexes is hard to compute. We may however simplify the computation by using formula (\ref{equ:Defsimpl}) and equivalently compute instead
\begin{align*}
 \Def(\LieBi_\infty\to \LieBi) &= \prod_{n,m} \Hom_{\bS_n\times \bS_m}(\invcoFrob\{1\}(n,m), \LieBi(n,m))
 \\
 &\cong \prod_{n,m}  (\LieBi(n,m)\otimes \sgn_n\otimes \sgn_m)^{\bS_n\times \bS_m}[-n-m]
  \\
  \Def(\LoB_\infty\to \LoB) &=  \prod_{n,m} \Hom_{\bS_n\times \bS_m}(\coFrob\{1\}(n,m), \LoB(n,m))
 \\
 &\cong \prod_{n,m} (\LoB(n,m)\otimes \sgn_n\otimes \sgn_m)^{\bS_n\times \bS_m}[-n-m] [[\hbar]]
 \\
 \Def(\Frob_\infty\to \Frob) &=\prod_{n,m} \Hom_{\bS_m\times \bS_n}(\invcoLieBi\{1\}(n,m), \Frob(n,m))
 \\
 \Def(\invFrob_\infty\to \invFrob) &=\prod_{n,m} \Hom_{\bS_n\times \bS_m\geq 1}(\coLieBi\{1\}(n,m), \invFrob(n,m))
 \, .
\end{align*}
Note however that here we lose the dg Lie algebra structure, or rather  there is a different Lie algebra structure on the above complexes.
The above complexes may again be interepreted as graph complexes. For example $\Def(\LieBi_\infty\to \LieBi)$ consists of directed trivalent acyclic graphs with incoming and outgoing legs, modulo the Jacobi and Drinfeld five term relations. The differential is obtained by attaching a trivalent vertex at one external leg in all possible ways.


\sip
Finally we note that of the above four deformation complexes only two are essentially different.
For example, note that $\Hom_{\bS_n\times \bS_m}(\coLieBi\{1\}(n,m), \invFrob(n,m))$ is just a completion of
\[
\Hom_{\bS_n\times \bS_m}(\invcoFrob\{1\}(n,m), \LieBi(n,m))\cong (\invFrob(n,m)\otimes \sgn_n\otimes \sgn_m) \otimes_{\bS_n\times \bS_n}\LieBi(n,m)[-n+m]
\]
Concretely, the completion is with respect to the genus grading of $\LieBi$, and the differential preserves the genus grading. Hence the cohomology of one complex is just the completion of the cohomology of the other with respect to the genus grading.

\sip

Similar arguments show that the cohomologies $\Def(\LoB_\infty\to \LoB)$ and $\Def(\Frob_\infty\to \Frob)$ are the same up to completion issues. Here the differential does not preserve the genus but preserves the quantity (genus)-($\hbar$-degree).
Hence it suffices to discuss one of each pair of deformation complexes. We will discuss $\Def(\LieBi_\infty\to \LieBi)$ and $\Def(\LoB_\infty\to \LoB)$ in the next section.

\bip


{\large
\section{\bf Oriented graph complexes and the universal $\grt_1$ action}
}

The goal of this section is to reduce the computation of the above deformation complexes to the computation of the cohomology of M. Kontsevich's graph complex.
By a result of one of the authors \cite{Wi1} the degree zero cohomology of this graph complex agrees with the Grotendieck-Teichm\"uller Lie algebra $\grt_1$.
This will allow us to conclude that the Grothendieck-Teichm\"uller group universally acts on $\Frob_\infty$ and $\LoB_\infty$ structures. This extends the well known result that the Grothendieck-Teichm\"uller group universally acts on Lie bialgebra structures.

\bip

\subsection{Grothendieck-Teichm\"uller group} The profinite and prounipotent Grothendieck-Teich\"uller groups were introduced by Vladimir Drinfeld in his
study of braid groups and quasi-Hopf algebras. They turned out to be one of the
most interesting and mysterious objects in modern mathematics. The profinite Grothendieck-Teichm\"uller group $\widehat{GT}$ plays an important role in number theory
and algebraic geometry. The pro-unipotent Grothendieck-Teichm\"uller group $GT$ (and its graded version $GRT$) over a field of characteristic zero was used by Pavel Etingof and David Kazhdan to solve the Drinfeld's
quantization conjecture for Lie bialgebras. Maxim Kontsevich's and Dmitry Tamarkin's formality theory
unravels the role of the group ${GRT}$  in the deformation quantization of Poisson structures.  Later Anton Alekseev and Charles Torossian applied $GRT$ to the Kashiwara-Vergne problem in Lie theory. The Grothendieck-Teichm\"uller group unifies different fields, and every time this group appears
in a mathematical theory, there follows a breakthrough in that theory. We refer to Hidekazu Furusho's lecture note \cite{Fu} for precise definitions and references.

\sip

In this paper we consider the Grothendieck-Teichm\"uller group $GRT_1$ which is the kernel of the canonical morphism of groups $GRT\rar \K^*$.   As $GRT_1$  is prounipotent, it is of the form $\exp(\alg{grt}_1)$ for some Lie algebra $\alg{grt}_1$ whose definition can be found, for example, in \S 6 of \cite{Wi1}. Therefore to understand representations of $GRT_1$ is the same as to understand representations of the  Grothendieck-Teichm\"uller Lie algebra $\alg{grt}_1$.

\subsection{Completed versions of $\LoB$ and $\LieBi$}\label{sec:completed versions}
The properads $\LieBi$ and $\LoB$ are naturally graded by the genus of the graphs describing the operations.
We will denote by $\hLoB$ and $\hLieBi$ the completions with respect to this grading.
Similarly, we denote by $\hLieBi_\infty$ the completion of $\LieBi_\infty$ with respect to the genus grading.
The natural map $\hLieBi_\infty\to\hLieBi$ is a quasi-isomorphism. Furthermore, we denote by $\hLoB_\infty$ the completion of $\LoB_\infty$ with respect to the genus plus the total weight-grading, i.~e., with respect to the grading $||\cdot||$ described in section {\ref{sec:decomposition}}. Then the map $\hLoB_\infty\to \hLoB$ is a quasi-isomorphism.

We will call a continuous representation of $\hLieBi$ (respectively of $\hLoB$) a \emph{genus complete} (involutive) Lie bialgebra.
Here the topoloy on $\hLieBi$ (respectively on $\hLoB$) is the one induced by the genus filtration (respectively the filtration $||\cdot||$).
For example, the involutive Lie bialgebra discussed in section {\ref{2: subsection on cyclic words}} is clearly genus complete since both the cobracket and the bracket reduce the lengths of the cyclic words.

Abusing notation slightly we will denote by $\Der(\hLieBi_\infty)$ (respectively by $\Der(\hLoB_\infty)$) the complex of \emph{continuous} derivations. 
The sub-properads $\LieBi_\infty\subset \hLieBi_\infty$ and $\LoB_\infty \subset \hLoB_\infty$ are dense by definition and hence any continuous derivation is determined by its restriction to these sub-properads. It also follows that the above complexes of derivations are isomorphic as complexes to $\Def(\LieBi_\infty\to \hLieBi_\infty)[1]$ and  $\Def(\LoB_\infty\to \hLoB_\infty)[1]$.
Finally we note that the cohomology of these complexes is merely the completion of the cohomology of the complexes $\Der(\LieBi_\infty)$ and $\Der(\LoB_\infty)$, since the differential respects the gradings.

\subsection{An operad of graphs $\cG ra^\uparrow$} A graph is called {\em directed}\, if its edges are equipped with directions as in the following examples,
$$
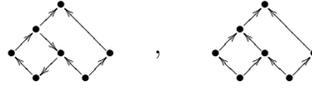
\Ba{c}
\resizebox{15mm}{!}{ \xy
(0,17)*{\bu}="u",
(-5,12)*{\bu}="0",
 (0,7)*{\bu}="a",
(-10,7)*{\bu}="L",
(10,7)*{\bu}="R",
(-5,2)*{\bu}="b_1",
(5,2)*{\bu}="b_2",
\ar @{<-} "a";"0" <0pt>
\ar @{->} "a";"b_1" <0pt>
\ar @{<-} "a";"b_2" <0pt>
\ar @{->} "b_1";"L" <0pt>
\ar @{<-} "0";"L" <0pt>
\ar @{->} "b_2";"R" <0pt>
\ar @{->} "R";"u" <0pt>
\ar @{->} "0";"u" <0pt>
\endxy}
\Ea
\ \ \ ,\ \ \
\Ba{c}
\resizebox{15mm}{!}{ \xy
(0,17)*{\bu}="u",
(-5,12)*{\bu}="0",
 (0,7)*{\bu}="a",
(-10,7)*{\bu}="L",
(10,7)*{\bu}="R",
(-5,2)*{\bu}="b_1",
(5,2)*{\bu}="b_2",
\ar @{->} "a";"0" <0pt>
\ar @{<-} "a";"b_1" <0pt>
\ar @{<-} "a";"b_2" <0pt>
\ar @{->} "b_1";"L" <0pt>
\ar @{<-} "0";"L" <0pt>
\ar @{->} "b_2";"R" <0pt>
\ar @{->} "R";"u" <0pt>
\ar @{->} "0";"u" <0pt>
\endxy}
\Ea
$$
A directed graph is called {\em oriented}\, or {\em acyclic}\, if it contains no  directed {\em closed}\, paths of edges.
For example, the second graph above is oriented while the first one is not.
 For arbitrary integers $n\geq 1$ and $l\geq 0$ let ${\sG}^\uparrow_{n,l}$ stand for the
set of  oriented graphs, $\{\Ga\}$, with $n$ vertices and $l$ edges
such that the vertices of $\Ga$ are labelled by elements of $[n]:=\{1,\ldots, n\}$, i.e.\ an isomorphism $V(\Ga)\rar [n]$ is fixed.

\sip

Let
 $\K\langle \sG_{n,l}^\uparrow\rangle$  be the vector space over a  field $\K$ of characteristic zero which is  spanned by graphs from
$\sG_{n,l}^\uparrow$,
and consider a  $\Z$-graded $\bS_n$-module,
$$
\cG ra^\uparrow (n):=\bigoplus_{l=0}^\infty \K\langle \sG^\uparrow_{n,l}\rangle[2l].
$$
For example,  $\xy
(0,2)*{_{1}},
(7,2)*{_{2}},
 (0,0)*{\bullet}="a",
(7,0)*{\bu}="b",
\ar @{->} "a";"b" <0pt>
\endxy$ is a degree $-2$ element in  $\cG ra^\uparrow(2)$.
 The $\bS$-module, $\cG ra^\uparrow :=\{\cG ra (n)^\uparrow\}_{n\geq 1}$, is naturally an operad with the
 operadic compositions given by
\Beq\label{3: operad comp in Gra}
\Ba{rccc}
\circ_i: & \cG ra^\uparrow (n)\ot \cG ra^\uparrow (m) & \lon &  \cG ra^\uparrow (m+n-1)\\
&  \Ga_1 \ot \Ga_2   &\lon & \sum_{\Ga\in \sG_{\Ga_1, \Ga_2}^i}  \Ga
\Ea
\Eeq
where $ \sG_{\Ga_1, \Ga_2}^i$ is the subset of $\sG^\uparrow_{n+m-1, \# E(\Ga_1) + \#E(\Ga_2)}$ consisting
of graphs, $\Ga$, satisfying the condition: the full subgraph of $\Ga$ spanned by the vertices labeled by
the set $\{i,i+1, \ldots, i+m-1\}$ is isomorphic to $\Ga_2$ and the quotient graph $\Ga/\Ga_2$ (which is obtained from $\Ga$
obtained by contracting that subgraph $\Ga_2$ to a single vertex) is isomorphic to $\Ga_1$, see, e.g.,  \S 7 in \cite{Me} or \S 2 in
\cite{Wi1}  for explicit examples of this kind of operadic compositions.
The unique element in $\sG_{1,0}^\uparrow$ serves as the unit in the operad  $\cG ra^\uparrow$.

\subsubsection{\bf A representation of $\cG ra^\uparrow$ in $\mathsf{LieB}(\fg)$}
\label{3: subsec on canonical repr of Gra}
For any graded vector space $\fg$ the operad  $\cG ra^\uparrow$ has a natural representation in the associated graded vector space $\mathsf{LieB}(\fg)$ (see (\ref{3:  LieB(g)})),
\Beq\label{3: Gra representation in g_V}
\Ba{rccc}
\rho: & \cG ra^\uparrow(n) & \lon & \cE  nd_{\mathsf{LieB}(\fg)}(n)=
\Hom( \mathsf{LieB}(\fg)^{\ot n},\mathsf{LieB}(\fg))\\
      & \Ga &\lon & \Phi_\Ga
\Ea
\Eeq
given by the formula,
$$
\Ba{rccc}
\Phi_\Ga: & \ot^n \mathsf{LieB}(\fg)   & \lon & \mathsf{LieB}(\fg)\\
& \ga_1\ot \ldots \ot \ga_n   &\lon &
\Phi_\Ga(\ga_1,\ldots, \ga_n) :=\mu\left(\prod_{e\in E(\Ga)}
\Delta_e \left(\ga_1(\psi)\ot \ga_2(\psi)\ot \ldots\ot
\ga_n(\psi) \right)\right)
\Ea
$$
where, for an edge $e=\Ba{c}\xy
(0,2)*{_{a}},
(6,2)*{_{b}},
 (0,0)*{\bullet}="a",
(6,0)*{\bu}="b",
\ar @{->} "a";"b" <0pt>
\endxy \Ea$ connecting a vertex labeled by $a\in [n]$ and to a vertex labelled by $b\in [n]$, we set
$$
\Delta_e \left(\ga_1\ot \ga_2\ot \ldots\ot
\ga_n \right)=
\left\{\Ba{cc}\underset{i\in I}{\sum}(-1)^{|\psi^i|(|\ga_a| + |\ga_{a+1}|+\ldots+ |\ga_{b-1}|)} \ga_1\ot ...\ot \frac{\p\ga_a}{\p \psi_i}\ot ...
\ot \frac{\p\ga_b}{\p \psi^i}\ot ... \ot
\ga_n & \mbox{for} \ a< b   \\
\underset{i\in I}{\sum}(-1)^{|\psi_i|(|\ga_b| + |\ga_{b+1}|+\ldots+ |\ga_{a-1}| +1 )} \ga_1\ot ... \ot \frac{\p\ga_b}{\p \psi^i}\ot ...
\ot \frac{\p\ga_a}{\p \psi_i}\ot ... \ot
\ga_n & \mbox{for} \ b< a
\Ea\right.
$$
and where  $\mu$ is the standard multiplication map in the ring  $\mathsf{LieB}(\fg)\subset \K[[\psi_i,\eta^i]]$,
$$
\Ba{rccc}
\mu:&   \mathsf{LieB}(\fg)^{\ot n} & \lon & \mathsf{LieB}(\fg)\\
   & \ga_1\ot \ga_2\ot \ldots \ot \ga_n &\lon & \ga_1 \ga_2 \cdots
   \ga_n.
\Ea
$$
Note that this representation makes sense for both finite- and {\em infinite}\, dimensional vector spaces $\fg$ as graphs from $\cG ra^\uparrow$ do not contain oriented cycles.

\begin{remark}
The above action of $\cG ra^\uparrow$ on $\Def\left(\caL ie\cB_\infty \stackrel{0}{\lon} \cE nd_\fg \right)[-2]$ only uses
the properadic compositions in $\cE nd_\fg$ and no further data. It follows that the same formulas may in fact be used to define an action of $\cG ra^\uparrow$ on the deformation complex
\[
 \Def\left(\caL ie\cB_\infty \stackrel{0}{\lon} \cP \right)[-2]
\]
for any properad $\cP$.
\end{remark}

\subsection{An oriented graph complex}
Let $\caL ie\{2\}$ be a (degree shifted) operad of Lie algebras, and let
 $\caL ie_\infty\{2\}$ be its minimal resolution. Thus $\caL ie\{2\}$ is a quadratic operad generated by degree $-2$ skewsymmetric binary operation,
$$
\begin{xy}
 <0mm,0.66mm>*{};<0mm,3mm>*{}**@{-},
 <0.39mm,-0.39mm>*{};<2.2mm,-2.2mm>*{}**@{-},
 <-0.35mm,-0.35mm>*{};<-2.2mm,-2.2mm>*{}**@{-},
 <0mm,0mm>*{\bu};<0mm,0mm>*{}**@{},
   <0.39mm,-0.39mm>*{};<2.9mm,-4mm>*{^{_2}}**@{},
   <-0.35mm,-0.35mm>*{};<-2.8mm,-4mm>*{^{_1}}**@{},
\end{xy}=-
\begin{xy}
 <0mm,0.66mm>*{};<0mm,3mm>*{}**@{-},
 <0.39mm,-0.39mm>*{};<2.2mm,-2.2mm>*{}**@{-},
 <-0.35mm,-0.35mm>*{};<-2.2mm,-2.2mm>*{}**@{-},
 <0mm,0mm>*{\bu};<0mm,0mm>*{}**@{},
   <0.39mm,-0.39mm>*{};<2.9mm,-4mm>*{^{_1}}**@{},
   <-0.35mm,-0.35mm>*{};<-2.8mm,-4mm>*{^{_2}}**@{},
\end{xy}
$$
modulo the Jacobi relations,
\Beq\label{3: Jacobi relation}
\Ba{c}
\begin{xy}
 <0mm,0mm>*{\bu};<0mm,0mm>*{}**@{},
 <0mm,0.69mm>*{};<0mm,3.0mm>*{}**@{-},
 <0.39mm,-0.39mm>*{};<2.4mm,-2.4mm>*{}**@{-},
 <-0.35mm,-0.35mm>*{};<-1.9mm,-1.9mm>*{}**@{-},
 <-2.4mm,-2.4mm>*{\bu};<-2.4mm,-2.4mm>*{}**@{},
 <-2.0mm,-2.8mm>*{};<0mm,-4.9mm>*{}**@{-},
 <-2.8mm,-2.9mm>*{};<-4.7mm,-4.9mm>*{}**@{-},
    <0.39mm,-0.39mm>*{};<3.3mm,-4.0mm>*{^{_3}}**@{},
    <-2.0mm,-2.8mm>*{};<0.5mm,-6.7mm>*{^{_2}}**@{},
    <-2.8mm,-2.9mm>*{};<-5.2mm,-6.7mm>*{^{_1}}**@{},
 \end{xy}
\ + \
 \begin{xy}
 <0mm,0mm>*{\bu};<0mm,0mm>*{}**@{},
 <0mm,0.69mm>*{};<0mm,3.0mm>*{}**@{-},
 <0.39mm,-0.39mm>*{};<2.4mm,-2.4mm>*{}**@{-},
 <-0.35mm,-0.35mm>*{};<-1.9mm,-1.9mm>*{}**@{-},
 <-2.4mm,-2.4mm>*{\bu};<-2.4mm,-2.4mm>*{}**@{},
 <-2.0mm,-2.8mm>*{};<0mm,-4.9mm>*{}**@{-},
 <-2.8mm,-2.9mm>*{};<-4.7mm,-4.9mm>*{}**@{-},
    <0.39mm,-0.39mm>*{};<3.3mm,-4.0mm>*{^{_2}}**@{},
    <-2.0mm,-2.8mm>*{};<0.5mm,-6.7mm>*{^{_1}}**@{},
    <-2.8mm,-2.9mm>*{};<-5.2mm,-6.7mm>*{^{_3}}**@{},
 \end{xy}
\ + \
 \begin{xy}
 <0mm,0mm>*{\bu};<0mm,0mm>*{}**@{},
 <0mm,0.69mm>*{};<0mm,3.0mm>*{}**@{-},
 <0.39mm,-0.39mm>*{};<2.4mm,-2.4mm>*{}**@{-},
 <-0.35mm,-0.35mm>*{};<-1.9mm,-1.9mm>*{}**@{-},
 <-2.4mm,-2.4mm>*{\bu};<-2.4mm,-2.4mm>*{}**@{},
 <-2.0mm,-2.8mm>*{};<0mm,-4.9mm>*{}**@{-},
 <-2.8mm,-2.9mm>*{};<-4.7mm,-4.9mm>*{}**@{-},
    <0.39mm,-0.39mm>*{};<3.3mm,-4.0mm>*{^{_1}}**@{},
    <-2.0mm,-2.8mm>*{};<0.5mm,-6.7mm>*{^{_3}}**@{},
    <-2.8mm,-2.9mm>*{};<-5.2mm,-6.7mm>*{^{_2}}**@{},
 \end{xy}\Ea=0
\Eeq
while  $\caL ie_\infty\{2\}$ is the free operad generated by an
$\bS$-module $E=\{E(n)\}_{n\geq 2}$,
$$
E(n):=sgn_n[3n-4]=\left\langle\Ba{c}
 \resizebox{19mm}{!}  {\xy
(1,-5)*{\ldots},
(-13,-7)*{_1},
(-8,-7)*{_2},
(-3,-7)*{_3},
(7,-7)*{_{n-1}},
(13,-7)*{_n},
 (0,0)*{\bu}="a",
(0,5)*{}="0",
(-12,-5)*{}="b_1",
(-8,-5)*{}="b_2",
(-3,-5)*{}="b_3",
(8,-5)*{}="b_4",
(12,-5)*{}="b_5",
\ar @{-} "a";"0" <0pt>
\ar @{-} "a";"b_2" <0pt>
\ar @{-} "a";"b_3" <0pt>
\ar @{-} "a";"b_1" <0pt>
\ar @{-} "a";"b_4" <0pt>
\ar @{-} "a";"b_5" <0pt>
\endxy}=(-1)^{\sigma}
\resizebox{19mm}{!}  {\xy
(1,-6)*{\ldots},
(-13,-7)*{_{\sigma(1)}},
(-6.7,-7)*{_{\sigma(2)}},
(13,-7)*{_{\sigma(n)}},
 (0,0)*{\bu}="a",
(0,5)*{}="0",
(-12,-5)*{}="b_1",
(-8,-5)*{}="b_2",
(-3,-5)*{}="b_3",
(8,-5)*{}="b_4",
(12,-5)*{}="b_5",
\ar @{-} "a";"0" <0pt>
\ar @{-} "a";"b_2" <0pt>
\ar @{-} "a";"b_3" <0pt>
\ar @{-} "a";"b_1" <0pt>
\ar @{-} "a";"b_4" <0pt>
\ar @{-} "a";"b_5" <0pt>
\endxy}\Ea
\right\rangle_{\sigma\in \bS_n}
$$
and equipped with the following differential,
\Beq\label{3: Lie_infty differential}
\p\hspace{-3mm}
\resizebox{19mm}{!}
{ \xy
(1,-5)*{\ldots},
(-13,-7)*{_1},
(-8,-7)*{_2},
(-3,-7)*{_3},
(7,-7)*{_{n-1}},
(13,-7)*{_n},
 (0,0)*{\bu}="a",
(0,5)*{}="0",
(-12,-5)*{}="b_1",
(-8,-5)*{}="b_2",
(-3,-5)*{}="b_3",
(8,-5)*{}="b_4",
(12,-5)*{}="b_5",
\ar @{-} "a";"0" <0pt>
\ar @{-} "a";"b_2" <0pt>
\ar @{-} "a";"b_3" <0pt>
\ar @{-} "a";"b_1" <0pt>
\ar @{-} "a";"b_4" <0pt>
\ar @{-} "a";"b_5" <0pt>
\endxy}
=
\sum_{ [n]=I_1\sqcup I_2\atop
\# I_1\geq 1, \# I_2\geq 1}(-1)^{\sigma(I_1\sqcup I_2) +|I_1||I_2|}
\Ba{c}
\resizebox{21mm}{!}{
\begin{xy}
<10mm,0mm>*{\bu},
<10mm,0.8mm>*{};<10mm,5mm>*{}**@{-},
<0mm,-10mm>*{...},
<14mm,-5mm>*{\ldots},
<13mm,-7mm>*{\underbrace{\ \ \ \ \ \ \ \ \ \ \ \ \  }},
<14mm,-10mm>*{_{I_2}};
<10.3mm,0.1mm>*{};<20mm,-5mm>*{}**@{-},
<9.7mm,-0.5mm>*{};<6mm,-5mm>*{}**@{-},
<9.9mm,-0.5mm>*{};<10mm,-5mm>*{}**@{-},
<9.6mm,0.1mm>*{};<0mm,-4.4mm>*{}**@{-},
<0mm,-5mm>*{\bu};
<-5mm,-10mm>*{}**@{-},
<-2.7mm,-10mm>*{}**@{-},
<2.7mm,-10mm>*{}**@{-},
<5mm,-10mm>*{}**@{-},
<0mm,-12mm>*{\underbrace{\ \ \ \ \ \ \ \ \ \ }},
<0mm,-15mm>*{_{I_1}},
\end{xy}}
\Ea
\Eeq
where $\sigma(I_1\sqcup I_2)$ is the sign of the shuffle $[n]\rar [I_1\sqcup I_2]$.

\subsubsection{\bf Proposition \cite{Wi2}}\label{3: Prop on map from Lie2 to Gra}
{\em There is a morphism of operads
$$
\varphi: \caL ie\{2\} \lon \cG ra^\uparrow
$$
given on the generators  by
$$
\Ba{c}
\xy
 <0mm,0.55mm>*{};<0mm,3.5mm>*{}**@{-},
 <0.5mm,-0.5mm>*{};<2.2mm,-2.2mm>*{}**@{-},
 <-0.48mm,-0.48mm>*{};<-2.2mm,-2.2mm>*{}**@{-},
 <0mm,0mm>*{\bu};<0mm,0mm>*{}**@{},
 <0.5mm,-0.5mm>*{};<2.7mm,-3.2mm>*{_2}**@{},
 <-0.48mm,-0.48mm>*{};<-2.7mm,-3.2mm>*{_1}**@{},
 \endxy\Ea
   \ \ \ \ \lon \ \ \ \ \xy
(0,2)*{_{1}},
(7,2)*{_{2}},
 (0,0)*{\bullet}="a",
(7,0)*{\bu}="b",
\ar @{->} "a";"b" <0pt>
\endxy\
- \
\xy
(0,2)*{_{2}},
(7,2)*{_{1}},
 (0,0)*{\bullet}="a",
(7,0)*{\bu}="b",
\ar @{->} "a";"b" <0pt>
\endxy
$$
}

\begin{proof} Using the definition of the operadic composition in $\cG ra^\uparrow$ we get
\Beqr
\varphi\left(\Ba{c}
 \begin{xy}
 <0mm,0mm>*{\bu};<0mm,0mm>*{}**@{},
 <0mm,0.69mm>*{};<0mm,3.0mm>*{}**@{-},
 <0.39mm,-0.39mm>*{};<2.4mm,-2.4mm>*{}**@{-},
 <-0.35mm,-0.35mm>*{};<-1.9mm,-1.9mm>*{}**@{-},
 <-2.4mm,-2.4mm>*{\bu};<-2.4mm,-2.4mm>*{}**@{},
 <-2.0mm,-2.8mm>*{};<0mm,-4.9mm>*{}**@{-},
 <-2.8mm,-2.9mm>*{};<-4.7mm,-4.9mm>*{}**@{-},
    <0.39mm,-0.39mm>*{};<3.3mm,-4.0mm>*{^3}**@{},
    <-2.0mm,-2.8mm>*{};<0.5mm,-6.7mm>*{^2}**@{},
    <-2.8mm,-2.9mm>*{};<-5.2mm,-6.7mm>*{^1}**@{},
 \end{xy}\Ea\right)&=&
\Ba{c}\xy
(0,2)*{_{1}},
(6,2)*{_{2}},
(12,2)*{_{3}},
 (0,0)*{\bullet}="a",
(6,0)*{\bu}="b",
(12,0)*{\bu}="c",
\ar @{->} "a";"b" <0pt>
\ar @{->} "b";"c" <0pt>
\endxy\Ea
\ - \
\Ba{c}\xy
(0,2)*{_{2}},
(6,2)*{_{1}},
(12,2)*{_{3}},
 (0,0)*{\bullet}="a",
(6,0)*{\bu}="b",
(12,0)*{\bu}="c",
\ar @{->} "a";"b" <0pt>
\ar @{->} "b";"c" <0pt>
\endxy\Ea
\ + \
\Ba{c}\xy
(0,2)*{_{2}},
(6,2)*{_{1}},
(12,2)*{_{3}},
 (0,0)*{\bullet}="a",
(6,0)*{\bu}="b",
(12,0)*{\bu}="c",
\ar @{<-} "a";"b" <0pt>
\ar @{->} "b";"c" <0pt>
\endxy\Ea
-
\Ba{c}\xy
(0,2)*{_{1}},
(6,2)*{_{2}},
(12,2)*{_{3}},
 (0,0)*{\bullet}="a",
(6,0)*{\bu}="b",
(12,0)*{\bu}="c",
\ar @{<-} "a";"b" <0pt>
\ar @{->} "b";"c" <0pt>
\endxy\Ea \label{3: morhism f into Gra}
\\
&&
\Ba{c}\xy
(0,2)*{_{1}},
(6,2)*{_{2}},
(12,2)*{_{3}},
 (0,0)*{\bullet}="a",
(6,0)*{\bu}="b",
(12,0)*{\bu}="c",
\ar @{<-} "a";"b" <0pt>
\ar @{<-} "b";"c" <0pt>
\endxy\Ea
\ - \
\Ba{c}\xy
(0,2)*{_{2}},
(6,2)*{_{1}},
(12,2)*{_{3}},
 (0,0)*{\bullet}="a",
(6,0)*{\bu}="b",
(12,0)*{\bu}="c",
\ar @{<-} "a";"b" <0pt>
\ar @{<-} "b";"c" <0pt>
\endxy\Ea
\ + \
\Ba{c}\xy
(0,2)*{_{2}},
(6,2)*{_{1}},
(12,2)*{_{3}},
 (0,0)*{\bullet}="a",
(6,0)*{\bu}="b",
(12,0)*{\bu}="c",
\ar @{->} "a";"b" <0pt>
\ar @{<-} "b";"c" <0pt>
\endxy\Ea
-
\Ba{c}\xy
(0,2)*{_{1}},
(6,2)*{_{2}},
(12,2)*{_{3}},
 (0,0)*{\bullet}="a",
(6,0)*{\bu}="b",
(12,0)*{\bu}="c",
\ar @{->} "a";"b" <0pt>
\ar @{<-} "b";"c" <0pt>
\endxy\Ea \nonumber
\Eeqr
which implies
$$
\varphi\left(\Ba{c}
 \begin{xy}
 <0mm,0mm>*{\bu};<0mm,0mm>*{}**@{},
 <0mm,0.69mm>*{};<0mm,3.0mm>*{}**@{-},
 <0.39mm,-0.39mm>*{};<2.4mm,-2.4mm>*{}**@{-},
 <-0.35mm,-0.35mm>*{};<-1.9mm,-1.9mm>*{}**@{-},
 <-2.4mm,-2.4mm>*{\bu};<-2.4mm,-2.4mm>*{}**@{},
 <-2.0mm,-2.8mm>*{};<0mm,-4.9mm>*{}**@{-},
 <-2.8mm,-2.9mm>*{};<-4.7mm,-4.9mm>*{}**@{-},
    <0.39mm,-0.39mm>*{};<3.3mm,-4.0mm>*{^3}**@{},
    <-2.0mm,-2.8mm>*{};<0.5mm,-6.7mm>*{^2}**@{},
    <-2.8mm,-2.9mm>*{};<-5.2mm,-6.7mm>*{^1}**@{},
 \end{xy}
\ + \
 \begin{xy}
 <0mm,0mm>*{\bu};<0mm,0mm>*{}**@{},
 <0mm,0.69mm>*{};<0mm,3.0mm>*{}**@{-},
 <0.39mm,-0.39mm>*{};<2.4mm,-2.4mm>*{}**@{-},
 <-0.35mm,-0.35mm>*{};<-1.9mm,-1.9mm>*{}**@{-},
 <-2.4mm,-2.4mm>*{\bu};<-2.4mm,-2.4mm>*{}**@{},
 <-2.0mm,-2.8mm>*{};<0mm,-4.9mm>*{}**@{-},
 <-2.8mm,-2.9mm>*{};<-4.7mm,-4.9mm>*{}**@{-},
    <0.39mm,-0.39mm>*{};<3.3mm,-4.0mm>*{^2}**@{},
    <-2.0mm,-2.8mm>*{};<0.5mm,-6.7mm>*{^1}**@{},
    <-2.8mm,-2.9mm>*{};<-5.2mm,-6.7mm>*{^3}**@{},
 \end{xy}
\ + \
 \begin{xy}
 <0mm,0mm>*{\bu};<0mm,0mm>*{}**@{},
 <0mm,0.69mm>*{};<0mm,3.0mm>*{}**@{-},
 <0.39mm,-0.39mm>*{};<2.4mm,-2.4mm>*{}**@{-},
 <-0.35mm,-0.35mm>*{};<-1.9mm,-1.9mm>*{}**@{-},
 <-2.4mm,-2.4mm>*{\bu};<-2.4mm,-2.4mm>*{}**@{},
 <-2.0mm,-2.8mm>*{};<0mm,-4.9mm>*{}**@{-},
 <-2.8mm,-2.9mm>*{};<-4.7mm,-4.9mm>*{}**@{-},
    <0.39mm,-0.39mm>*{};<3.3mm,-4.0mm>*{^1}**@{},
    <-2.0mm,-2.8mm>*{};<0.5mm,-6.7mm>*{^3}**@{},
    <-2.8mm,-2.9mm>*{};<-5.2mm,-6.7mm>*{^2}**@{},
 \end{xy}\Ea\right)=0.
$$
\end{proof}

All possible morphisms of dg operads,  $\caL ie_\infty\{2\} \lon \cG ra^\uparrow$, can be usefully encoded as Maurer-Cartan elements in the graded Lie algebra,
$$
\mathsf{f} \sG\sC_3^{or}:= \Def(\caL ie_\infty\{2\} \stackrel{0}{\lon} \cG ra^\uparrow),
$$
which controls deformation theory of the zero morphism (cf.\ \cite{MV}). As a graded vector space,
\Beq \label{equ:GCordef}
\mathsf{f} \sG\sC_3^{or}\cong \prod_{n\geq 2} \Hom_{\bS_n}(E(n), \cG ra^\uparrow(n))[-1]=  \prod_{n\geq 2} \cG ra^\uparrow(n)^{\bS_n}[3-3n],
\Eeq
so that its elements can be understood as ($\K$-linear series of) graphs $\Ga$ from $\cG ra^\uparrow$ whose vertex labels are skewsymmetrized (so that we can often forget numerical labels of vertices in our pictures), and which are assigned the homological degree
$$
|\Ga|= 3\# V(\Ga) -3 - 2\# E(\Ga),
$$
where $V(\Ga)$ (resp.\, $E(\Ga)$) stands for the set of vertices (resp., edges) of $\Ga$.

\mip

The Lie brackets, $[\ , \ ]_{\mathsf{gra}}$, in $\mathsf{f} \sG\sC_3^{or}$  can be either read from the differential (\ref{3: Lie_infty differential}), or, equivalently, from the following explicit Lie algebra structure \cite{KM} associated with the degree shifted operad $\cG ra_3^\uparrow\{3\}$ (and which makes sense for any operad),
$$
\Ba{rccc}
[\ ,\ ]:&  \sP \ot \sP & \lon & \sP\\
& (a\in \cP(n), b\in \cP(m)) & \lon &
[a, b]:= \sum_{i=1}^n a\circ_i b - (-1)^{|a||b|}\sum_{i=1}^m b\circ_i a
\Ea
$$
where
$
\sP:= \prod_{n\geq 1}\cG ra(n)^\uparrow[3-3n]$. These Lie brackets in $\sP$ induce
Lie brackets, $[\ ,\ ]_{\mathsf{gra}}$, in the subspace of $\bS$-invariants \cite{KM},
$$
\sP^\bS:=  \prod_{n\geq 1}\cG ra^\uparrow(n)[3-3n]^{\bS_n}= \fGCor_3 
$$
via the standard symmetrization map $\sP\rar \sP^\bS$.

\mip

The graph
$$
\xy
 (0,0)*{\bullet}="a",
(7,0)*{\bu}="b",
\ar @{->} "a";"b" <0pt>
\endxy:=
\xy
(0,2)*{_{1}},
(7,2)*{_{2}},
 (0,0)*{\bullet}="a",
(7,0)*{\bu}="b",
\ar @{->} "a";"b" <0pt>
\endxy\
- \
\xy
(0,2)*{_{2}},
(7,2)*{_{1}},
 (0,0)*{\bullet}="a",
(7,0)*{\bu}="b",
\ar @{->} "a";"b" <0pt>
\endxy
$$
is a degree $2\cdot 3-3-2=1$ element in $\sf \sG\sC_3^{or}$, which, in fact, is a Maurer-Cartan element,
$$
\left[\xy
 (0,0)*{\bullet}="a",
(7,0)*{\bu}="b",
\ar @{->} "a";"b" <0pt>
\endxy, \xy
 (0,0)*{\bullet}="a",
(7,0)*{\bu}="b",
\ar @{->} "a";"b" <0pt>
\endxy \right]= \mathrm{skewsymmetrization\ of\ the\ r.h.s.\ in}\ (\ref{3: morhism f into Gra})=0,
$$
which represents the above morphism $\varphi$ in the Lie algebra $\fGCor_3$. This element makes, therefore, $\mathsf{f}\sG\sC_3^{or}$ into a {\em differential}\, graded Lie algebra
with the differential
\Beq\label{3: d in GC_3}
d\Ga:= \left[\xy
 (0,0)*{\bullet}="a",
(7,0)*{\bu}="b",
\ar @{->} "a";"b" <0pt>
\endxy, \Ga\right]_{\mathsf{gra}}.
\Eeq
Let $\sG\sC_3^{or}$  be a subspace of $\mathsf{f}\sG\sC_3^{or}$ spanned by connected graphs whose vertices are at least bivalent, and if bivalent do not have one incoming and one outgoing edge. It is easy to see that this is a dg Lie subalgebra.

\begin{remark}
 This definition of $\sG\sC_3^{or}$ in \cite{Wi1} differs slightly from the present one in that all bivalent vertices are allowed in loc. cit. However, it is easy to check that this extra condition does not change the cohomology.
\end{remark}

The cohomology of the  {\em oriented graph complex}\, $(\sG\sC_3^{or}, d)$ was partially computed in \cite{Wi1}.

\subsubsection{\bf Theorem \cite{Wi2}}\label{3: Willwacher theorem on GC_3} (i) {\em  $H^0(\sG\sC_3^{or}, d)=\fg\fr\ft_1$, where
$\fg\fr\ft_1$ is the Lie algebra of the prounipotent Grothendieck-Teichm\"uller group $GRT_1$
introduced by Drinfeld in \cite{D2}.  }

(ii) {\em $H^{-1}(\sG\sC_3^{or}, d)\cong  \K$. The single class is represented by the graph
$
\Ba{c}\resizebox{4mm}{!}{   \xy
   \ar@/^0.6pc/(0,5)*{\bullet};(0,-5)*{\bullet}
   \ar@/^{-0.6pc}/(0,5)*{\bullet};(0,-5)*{\bullet}
 \endxy}\Ea
$.

(iii) {\em $H^i(\sG\sC_3^{or}, d)=0$ for all $i\leq-2$.}}

\mip

\subsection{Action on $\hLieBi_\infty$}
There is a natural action of $\sG\sC_3^{or}$ on the properad $\hLieBi_\infty$ by properadic derivations.
Concretely, for any graph $\Gamma$ we define the derivation $F(\Gamma)\in \Der(\hLieBi_\infty)$ sending the generator $\mu_{m,n}$ of $\hLieBi_\infty$ to the linear combination of graphs
\Beq \label{equ:def GC action 1}
\mu_{m,n}\cdot \Gamma=
 \sum
    \overbrace{
 \underbrace{ \Ba{c}\resizebox{9mm}{!}  {\xy
(0,4.5)*+{...},
(0,-4.5)*+{...},
(0,0)*+{\Ga}="o",
(-5,6)*{}="1",
(-3,6)*{}="2",
(3,6)*{}="3",
(5,6)*{}="4",
(-3,-6)*{}="5",
(3,-6)*{}="6",
(5,-6)*{}="7",
(-5,-6)*{}="8",
\ar @{<-} "o";"1" <0pt>
\ar @{<-} "o";"2" <0pt>
\ar @{<-} "o";"3" <0pt>
\ar @{<-} "o";"4" <0pt>
\ar @{->} "o";"5" <0pt>
\ar @{->} "o";"6" <0pt>
\ar @{->} "o";"7" <0pt>
\ar @{->} "o";"8" <0pt>
\endxy}\Ea
 }_{n\times}
 }^{m\times}
%
%
%
\Eeq
where the sum is taken over all ways of attaching the incoming and outgoing legs such that all vertices are at least trivalent and have at least one incoming and one outgoing edge.

\begin{lemma}\label{lem:GC3 action on LieBi}
 The above formula defines a right action of $\GC_3^{or}$ on $\hLieBi_\infty$.
\end{lemma}
\begin{proof}[Proof sketch]
 We denote by $\bullet$ the pre-Lie product on the deformation complex $\Def(\caL ie_\infty\{2\} \stackrel{0}{\lon} \cG ra^\uparrow)\supset \GC_3^{or}$, so that the Lie bracket on $\GC_3^{or}$ may be written as $[\Gamma, \Gamma']=\Gamma\bullet \Gamma' \pm \Gamma' \bullet \Gamma$ for $\Gamma,\Gamma'\in \GC_3^{or}$.
 Note that
 \[
\mu_{m,n} \cdot (\Gamma \bullet \Gamma')=(\mu_{m,n} \cdot \Gamma) \cdot \Gamma' ,
 \]
where it is important that we excluded graphs with bivalent vertices with one incoming and one outgoing edge from the definition of $\sG\sC_3^{or}$. It follows that the formula above defines an action of the graded Lie algebra $\GC_3^{or}$. We leave it to the reader to check that this action also commutes with the differential.
\end{proof}

Of course, by a change of sign the right action may be transformed into a left action and hence we obtain a map of Lie algebras
\[
 F\colon \sG\sC_3^{or}\to \Der(\hLieBi_\infty)\, .
\]
Interpreting the right hand side as a graph complex as in section \ref{sec:defcomplexes}, the map $F$ sends a graph $\Gamma\in \GC_3^{or}$ to the series of graphs
\[
\pm
\sum_{m,n}
 \sum
    \overbrace{
 \underbrace{\Ba{c}\resizebox{10mm}{!}  { \xy
(0,4.5)*+{...},
(0,-4.5)*+{...},
(0,0)*+{\Ga}="o",
(-5,6)*{}="1",
(-3,6)*{}="2",
(3,6)*{}="3",
(5,6)*{}="4",
(-3,-6)*{}="5",
(3,-6)*{}="6",
(5,-6)*{}="7",
(-5,-6)*{}="8",
\ar @{<-} "o";"1" <0pt>
\ar @{<-} "o";"2" <0pt>
\ar @{<-} "o";"3" <0pt>
\ar @{<-} "o";"4" <0pt>
\ar @{->} "o";"5" <0pt>
\ar @{->} "o";"6" <0pt>
\ar @{->} "o";"7" <0pt>
\ar @{->} "o";"8" <0pt>
\endxy}\Ea
 }_{n\times}
 }^{m\times}
%
%
%
 %
\]

\begin{theorem}\label{thm:Fqiso}
 The map $F : \GC_3^{or}\to \Der(\hLieBi_\infty)$ is a quasi-isomorphism, up to one class in  $\Der(\hLieBi_\infty)$ represented by the series
 \[
  \sum_{m,n}(m+n-2)
  \overbrace{
  \underbrace{
 \Ba{c}\resizebox{10mm}{!}  {\xy
(0,4.5)*+{...},
(0,-4.5)*+{...},
(0,0)*{\bu}="o",
(-5,5)*{}="1",
(-3,5)*{}="2",
(3,5)*{}="3",
(5,5)*{}="4",
(-3,-5)*{}="5",
(3,-5)*{}="6",
(5,-5)*{}="7",
(-5,-5)*{}="8",
\ar @{<-} "o";"1" <0pt>
\ar @{<-} "o";"2" <0pt>
\ar @{<-} "o";"3" <0pt>
\ar @{<-} "o";"4" <0pt>
\ar @{->} "o";"5" <0pt>
\ar @{->} "o";"6" <0pt>
\ar @{->} "o";"7" <0pt>
\ar @{->} "o";"8" <0pt>
\endxy}\Ea
 %
%
 }_{n\times}
 }^{m\times}.
 \]
\end{theorem}
The result will not be used directly in this paper. The proof is an adaptation of the proof of \cite[Proposition 3]{Wi2} and is given in Appendix \ref{app:defproof1}.

\mip

\subsection{$GRT_1$ action on Lie bialgebra structures}
The action of the Lie algebra of closed degree zero cocycles $\GCor_{3,cl}\subset \GCor_3$ on $\hLieBi_\infty$ by derivations may be integrated
to an action of the exponential group $\Exp\GCor_{3,cl}$ on $\hLieBi_\infty$ by (continuous) automorphisms.
Hence this exponential group acts on the set of $\hLieBi_\infty$ algebra structures on any dg vector space $\fg$, i.~e., on the set of morphisms of properads
\[
 \hLieBi_\infty \to \End_\fg
\]
by precomposition.
Furthermore, it follows that the cohomology Lie algebra $H^0(\GCor_3)\cong \grt_1$ maps into the the Lie algebra of continuous derivations up to homotopy $H^0(\Der(\hLieBi_\infty))$ and the the exponential group $\Exp H^0(\GCor_3)\cong GRT_1$ maps into the set of homotopy classes of continuous automorphisms of $\hLieBi_\infty$.

\begin{remark}\label{rem:incomplete}
Note also that one may define a non-complete version $\GCor_{3,inc}$ of the graph complex $\GCor_3$ by merely replacing the direct product by a direct sum in \eqref{equ:GCordef}. The zeroth cohomology of $\GCor_{3,inc}$ is a non-complete version of the Grothendieck-Teichm\"uller group. Furthermore $\GCor_{3,inc}$ acts on the non-completed operad $\LieBi_\infty$
by derivations, using the formulas of the previous subsection, and hence also on $\LieBi_\infty$ algebra structures. However, these actions can in general not be integrated, whence we work with the completed properad $\hLieBi_\infty$ above.
\end{remark}

Finally, let us describe the action $\GCor_{3,cl}$ on Lie bialgebra structures in yet another form.
We have a sequence of morphisms of dg Lie algebras,
$$
\sG\sC_3^{or}\  \ \lon \ \ \Def(\caL ie_\infty\{2\} \rar \cG ra^\uparrow) \ \ \stackrel{\rho}{\lon} \ \
 \Def(\caL ie_\infty\{2\} \stackrel{\{\ ,\ \}}{\lon} \cE nd_{\mathsf{LieB}(\fg)})
$$
where the first arrow is just the inclusion, and the second arrow is induced by the canonical representation (\ref{3: Gra representation in g_V}) and which obviously satisfies
$$
\rho\circ \varphi \left(\Ba{c}
\xy
 <0mm,0.55mm>*{};<0mm,3.5mm>*{}**@{-},
 <0.5mm,-0.5mm>*{};<2.2mm,-2.2mm>*{}**@{-},
 <-0.48mm,-0.48mm>*{};<-2.2mm,-2.2mm>*{}**@{-},
 <0mm,0mm>*{\bu};<0mm,0mm>*{}**@{},
 <0.5mm,-0.5mm>*{};<2.7mm,-3.2mm>*{_2}**@{},
 <-0.48mm,-0.48mm>*{};<-2.7mm,-3.2mm>*{_1}**@{},
 \endxy\Ea\right)= \{\ ,\ \}\in \Hom\left(\wedge^2 \mathsf{LieB}, \mathsf{LieB}[-2]\right)\subset \Def(\caL ie_\infty\{2\} \stackrel{\{\ ,\ \}}{\lon} \cE nd_{\mathsf{LieB}(\fg)}).
$$

The dg Lie algebra,
$$
\Def(\caL ie_\infty\{2\} \stackrel{\{\ ,\ \}}{\lon} \cE nd_{\mathsf{LieB}(\fg)}) =: CE^\bu(\mathsf{LieB}(\fg)),
$$
is nothing but the classical Chevalley-Eilenberg complex controlling deformations of the Poisson brackets (\ref{3: Poisson brackets in LieB(g)}) in $\mathsf{LieB}(\fg)$.
In particular for any closed degree zero element $g\in \sG\sC_3^{or}$, and in particular for representatives of elements of $\fg\fr\ft_1$ in $\sG\sC_3^{or}$, we obtain a $\Lie_\infty$ derivation of $\mathsf{LieB}(\fg)$. This derivation may be integrated into a $\Lie_\infty$ automorphism $\exp(ad_g)$.
For a Maurer-Cartan element $\ga$ in $\mathsf{LieB}(\fg)$ corresponding to a graded complete $\LieBi_\infty$ structure on $\fg$ the series
$$
\ga \lon exp(ad_g) \ga
$$
converges and defines again a graded complete homotopy Lie bialgebra structure on $\fg$.

%
%

\begin{remark}
 By degree reasons the above action on $\hLieBi_\infty$ structures maps $\hLieBi$ structures again to $\hLieBi$ structures. In other words, no higher homotopies are created if there were none before.
\end{remark}

\subsection{Another oriented graph complex}\label{sec:GChbar}
We shall introduce next a new oriented graph complex and then use ({\em both}\, results of) Theorem {\ref{3: Willwacher theorem on GC_3} } to compute partially its cohomology and then deduce formulae for an action
of $GRT_1$ on {\em involutive}\, Lie bialgebra structures.

Let $\hbar$ be a formal variable of homological degree $2$. The Lie brackets $[\ ,\
]_{\mathrm{gra}}$ in $\sG\sC_3^{or}$ extend $\hbar$-linearly to the topological
vector space $\sG\sC_3^{or}[[\hbar]]$. 

\subsection{Proposition} \label{prop:phihbar} {\em The element
$$
\Phi_\hbar:= \sum_{k=1}^\infty \hbar^{k-1} \underbrace{
\Ba{c}\resizebox{6mm}{!}  {\xy
(0,-5)*{...},
   \ar@/^1pc/(0,0)*{\bullet};(0,-10)*{\bullet}
   \ar@/^{-1pc}/(0,0)*{\bullet};(0,-10)*{\bullet}
   \ar@/^0.6pc/(0,0)*{\bullet};(0,-10)*{\bullet}
   \ar@/^{-0.6pc}/(0,0)*{\bullet};(0,-10)*{\bullet}
 \endxy}
 \Ea}_{k\ \mathrm{edges}}
$$
is a Maurer-Cartan element in the Lie algebra $(\mathsf{f}\sG\sC_3^{or}[[\hbar]], [\ ,\ ]_{\mathsf{gra}})$}.
\begin{proof}
\Beqrn
[\Phi_\hbar, \Phi_\hbar] &=& 
\sum_{k=1}^\infty \sum_{l=1}^\infty \hbar^{k+l-2}\sum_{k=k'+k''}
\Ba{c}\resizebox{12mm}{!}  {\xy
(0,-5)*{\stackrel{l}{...}},
(0,5)*{\stackrel{k'}{...}},
(8,0)*{\stackrel{k''}{...}},
   \ar@/^{-1pc}/(0,0)*{\bullet};(0,-10)*{\bullet}
   \ar@/^0.6pc/(0,0)*{\bullet};(0,-10)*{\bullet}
   \ar@/^{-0.6pc}/(0,0)*{\bullet};(0,-10)*{\bullet}
   \ar@/^{-1pc}/(0,10)*{\bullet};(0,0)*{\bullet}
   \ar@/^0.6pc/(0,10)*{\bullet};(0,0)*{\bullet}
   \ar@/^{-0.6pc}/(0,10)*{\bullet};(0,0)*{\bullet}
   \ar@/^{2.4pc}/(0,10)*{\bullet};(0,-10)*{\bullet}
   \ar@/^{1.3pc}/(0,10)*{\bullet};(0,-10)*{\bullet}
 \endxy}
 \Ea -
 \sum_{k=1}^\infty \sum_{l=1}^\infty \hbar^{k+l-2}\sum_{k=k'+k''}
 \Ba{c}\resizebox{12.5mm}{!}  {
 \xy
(0,-5)*{\stackrel{k'}{...}},
(0,5)*{\stackrel{l}{...}},
(8,0)*{\stackrel{k''}{...}},
   \ar@/^{-1pc}/(0,0)*{\bullet};(0,-10)*{\bullet}
   \ar@/^0.6pc/(0,0)*{\bullet};(0,-10)*{\bullet}
   \ar@/^{-0.6pc}/(0,0)*{\bullet};(0,-10)*{\bullet}
   \ar@/^{-1pc}/(0,10)*{\bullet};(0,0)*{\bullet}
   \ar@/^0.6pc/(0,10)*{\bullet};(0,0)*{\bullet}
   \ar@/^{-0.6pc}/(0,10)*{\bullet};(0,0)*{\bullet}
   \ar@/^{2.4pc}/(0,10)*{\bullet};(0,-10)*{\bullet}
   \ar@/^{1.3pc}/(0,10)*{\bullet};(0,-10)*{\bullet}
 \endxy}
 \Ea \\
 &=& 0.
\Eeqrn
\end{proof}

Hence the degree one continuous map
$$
\Ba{rccc}
d_\hbar: & \sG\sC_3^{or}[[\hbar]] &\lon & \sG\sC_3^{or}[[\hbar]]\\
         &  \Ga &  \lon & d_\hbar\Ga:= [\Phi_\hbar,\Ga]_{\mathrm{gra}}
\Ea
$$
is a differential in $\sG\sC_3^{or}[[\hbar]]$. The induced differential, $d$,  in
$\sG\sC_3^{or}=\sG\sC_3^{or}[[\hbar]]/ \hbar \sG\sC_3^{or}[[\hbar]]$ is precisely the
original
differential (\ref{3: d in GC_3}).

\subsection{Action on $\hLoB_\infty$}\label{sec:action on LoB}
The dg Lie algebra $(\GC_3^{or}[[\hbar]], d_\hbar)$
acts naturally on the properad $\hLoB_\infty$ by continuous properadic derivations. More precisely, let $\Gamma\in \GC_3^{or}$ be a graph.
Then to the element $\hbar^N\Gamma\in \GC_3^{or}[[\hbar]]$ we assign the derivation of $\hLoB_\infty$ that sends the generator $\mu_{m,n}^k$ to zero if $k<N$ and to
\Beq \label{equ:def GC action 2}
\mu_{m,n}^k \cdot (\hbar^N\Gamma)
=
 \sum
\Ba{c}\resizebox{13mm}{!}{ \xy
 (-5,7)*{^{_1}},
 (-3,7)*{^{_2}},
 (5.5,7)*{^{_m}},
 (-5,-7.9)*{^{_1}},
 (-3,-7.9)*{^{_2}},
 (5.5,-7.9)*{^{_n}},
(1,4.5)*+{...},
(1,-4.5)*+{...},
(0,0)*+{\Ga}="o",
(-5,6)*{}="1",
(-3,6)*{}="2",
(5,6)*{}="3",
(5,-6)*{}="4",
(-3,-6)*{}="5",
(-5,-6)*{}="6",
\ar @{<-} "o";"1" <0pt>
\ar @{<-} "o";"2" <0pt>
\ar @{<-} "o";"3" <0pt>
\ar @{->} "o";"4" <0pt>
\ar @{->} "o";"5" <0pt>
\ar @{->} "o";"6" <0pt>
\endxy}\Ea
 %
%
\Eeq
where the sum is over all graphs obtained by (i) attaching the external legs to $\Gamma$ in all possible ways as in \eqref{equ:def GC action 2} and (ii) assigning weights to the vertices in all possible ways such that the weights sum to $k-N$.

\begin{lemma}\label{lem:action of GC on LoB}
 The above formula defines a right action of $(\GC_3^{or}[[\hbar]], d_\hbar)$ on $\hLoB_\infty$.
\end{lemma}
\begin{proof}[Proof sketch.]
The proof is similar to that of Lemma \ref{lem:GC3 action on LieBi} after noting that
\[
\mu_{m,n}^k \cdot (\hbar^N \Gamma \bullet \hbar^M \Gamma')=(\mu_{m,n} \cdot \hbar^N\Gamma) \cdot \hbar^M \Gamma'
\]
for all $M,N$ and $\Gamma,\Gamma'\in \GC_3^{or}$.
\end{proof}

Again, by a change of sign the right action may be transformed into a left action and hence we obtain a map of Lie algebras
\[
 F_\hbar \colon \sG\sC_3^{or}[[\hbar]]\to \Der(\hLoB_\infty)\, .
\]

\begin{theorem}\label{thm:Fhbarqiso}
 The map $F_\hbar$ is a quasi-isomorphism, up to classes $T\K[[\hbar]]\subset \Der(\hLoB_\infty)$ where
 \[
  T=
  \sum_{m,n,p}(m+n+2p-2) \hbar^{p}
   \overbrace{
 \underbrace{
 \xy
(0,4.5)*+{...},
(0,-4.5)*+{...},
(0,0)*+{_p}*\cir{}="o",
(-5,5)*{}="1",
(-3,5)*{}="2",
(3,5)*{}="3",
(5,5)*{}="4",
(-3,-5)*{}="5",
(3,-5)*{}="6",
(5,-5)*{}="7",
(-5,-5)*{}="8",
\ar @{<-} "o";"1" <0pt>
\ar @{<-} "o";"2" <0pt>
\ar @{<-} "o";"3" <0pt>
\ar @{<-} "o";"4" <0pt>
\ar @{->} "o";"5" <0pt>
\ar @{->} "o";"6" <0pt>
\ar @{->} "o";"7" <0pt>
\ar @{->} "o";"8" <0pt>
\endxy
%
 }_{n\times}
 }^{m\times}.
 \]
\end{theorem}
The result will not be used in this paper, and the proof will be given in Appendix \ref{app:defproof2}.

\subsection{Action on involutive Lie bialgebra structures}
By the previous subsection the Lie algebra of degree 0 cocycles in $\sG\sC_3^{or}[[\hbar]]$ acts on the properad $\hLoB_\infty$ by derivations. The action may be integrated to an action of the corresponding exponential group on $\hLoB_\infty$ by continuous automorphisms, and hence also on the set of $\hLoB_\infty$ algebra structures on some dg vector space by precomposition.
Furthermore, the cohomology Lie algebra $H^0(\sG\sC_3^{or}[[\hbar]])$ maps into the the Lie algebra of continuous derivations up to homotopy $H^0(\Der(\hLoB_\infty))$, while the
the exponential group $\exp H^0(\GCor_3[[\hbar]])$ maps into the set of homotopy classes of continuous automorphisms of $\hLoB_\infty$.
By precomposition, we also have a map of Lie algebras $H^0(\sG\sC_3^{or}[[\hbar]])\to H^0(\Def(\hLoB_\infty\to \End_\fg))$ for any $\hLoB_\infty$ algebra $\fg$ and an action of $\exp H^0(\GCor_3[[\hbar]])$ on the set of such algebra structures up to homotopy.
Let us encode these findings in the following corollary.

\begin{corollary}  The Lie algebra $H^0(\GC_3^{or}[[\hbar]])$ and its exponential group $\exp( H^0(\GC_3^{or}[[\hbar]]))$ canonically act on the set of graded complete strong homotopy involutive Lie bialgebra structures up to homotopy.
\end{corollary}

\begin{remark}
Note again that, analogously to Remark \ref{rem:incomplete}, we may define a non-complete version of the graph complex $\sG\sC_3^{or}[[\hbar]]$ which acts on the non-complete properad $\LoB_\infty$ by derivations, and hence also on ordinary Lie bialgebra structures. However, these actions can in general not be integrated, whence we prefer to work with the complete version of the graph complex and $\hLoB_\infty$.
\end{remark}

Finally let us give another yet another description of the action of $\GCor_3[[\hbar]]$ on involutive Lie bialgebra structures, strengthening the above result a little.
As usual, the deformation complex of the zero morphism of dg props,
$\mathsf{InvLieB}(\fg):=
\Def(\LoB_\infty\stackrel{0}{\rar} \cE nd_\fg)$
has a canonical dg Lie algebra structure, with the Lie bracket $[\ ,\ ]_{\hbar}$ given explicitly by (\ref{3: brackets in InvLieB(g)}), such that the
Maurer-Cartan
elements are in 1-to-1 correspondence with $\LoB_\infty$-structures on the dg vector space $\fg$. The Maurer-Cartan element $\Phi_\hbar$ in $\mathsf{f}\sG\sC_3^{or}[[\hbar]]$ corresponds to a continuous morphism of operads,
$$
\varphi_\hbar: \caL ie\{2\}[[\hbar]] \lon \cG ra^\uparrow[[\hbar]],
$$
given on the generator of  $\caL ie\{2\}$ by the formula
$$
\varphi_\hbar \left(\Ba{c}
\xy
 <0mm,0.55mm>*{};<0mm,3.5mm>*{}**@{-},
 <0.5mm,-0.5mm>*{};<2.2mm,-2.2mm>*{}**@{-},
 <-0.48mm,-0.48mm>*{};<-2.2mm,-2.2mm>*{}**@{-},
 <0mm,0mm>*{\bu};<0mm,0mm>*{}**@{},
 <0.5mm,-0.5mm>*{};<2.7mm,-3.2mm>*{_2}**@{},
 <-0.48mm,-0.48mm>*{};<-2.7mm,-3.2mm>*{_1}**@{},
 \endxy\Ea\right)=
 \sum_{k=1}^\infty \hbar^{k-1}\left(
 \Ba{c}\resizebox{7mm}{!}{
\xy
(0,2.2)*{_1},
(0,-12.2)*{_2},
(0,-3.5)*{_k},
(0,-5)*{...},
   \ar@/^1pc/(0,0)*{\bullet};(0,-10)*{\bullet}
   \ar@/^{-1pc}/(0,0)*{\bullet};(0,-10)*{\bullet}
   \ar@/^0.6pc/(0,0)*{\bullet};(0,-10)*{\bullet}
   \ar@/^{-0.6pc}/(0,0)*{\bullet};(0,-10)*{\bullet}
 \endxy}
 \Ea -
  \Ba{c}\resizebox{6mm}{!} {\xy
(0,2.2)*{_2},
(0,-12.2)*{_1},
(0,-3.5)*{_k},
(0,-5)*{...},
   \ar@/^1pc/(0,0)*{\bullet};(0,-10)*{\bullet}
   \ar@/^{-1pc}/(0,0)*{\bullet};(0,-10)*{\bullet}
   \ar@/^0.6pc/(0,0)*{\bullet};(0,-10)*{\bullet}
   \ar@/^{-0.6pc}/(0,0)*{\bullet};(0,-10)*{\bullet}
 \endxy}
 \Ea
 \right)\, .
$$
The representation (\ref{3: Gra representation in g_V}) of the operad $\cG ra^\uparrow$ in the deformation complex $\Def(\LieBi_\infty\stackrel{0}{\to} \cP)[-2]$ extends $\hbar$-linearly to a representation $\rho_\hbar$ of $\cG ra^\uparrow[[\hbar]]$ in $\Def(\LoB_\infty\stackrel{0}{\to} \cP)[-2]$ for any properad $\cP$.
Furthermore it is almost immediate to see that the action of $\Lie\{2\}$ on the latter deformation complex factors through the map $\Lie\{2\}\to \cG ra^\uparrow[[\hbar]]$.

%
%
%
%
It follows that one has a morphism of dg Lie algebras induced by $\varphi_\hbar$
$$
\left(\sG\sC_3^{or}[[\hbar]], d_\hbar\right)\ \  \lon \ \  \Def\left(\caL ie_\infty\{2\}[[\hbar]] \stackrel{\varphi_\hbar}{\rar} \cG ra^\uparrow[[\hbar]]\right)\ \ \stackrel{\rho_\hbar}{\lon}
\ \
 \Def\left(\caL ie_\infty\{2\}[[\hbar]] \stackrel{[\ ,\ ]_\hbar}{\lon} \cE nd_{D}\right)=:CE^\bu\left(D\right)
$$
from the graph complex $\left(\sG\sC_3^{or}[[\hbar]], d_\hbar\right)$ into the Chevalley-Eilenberg dg Lie algebra of $D:=\Def(\LieBi_\infty\stackrel{0}{\to} \cP)[-2]$.
In particular this implies  the following:

\subsection{\bf Theorem}\label{4: Theorem on cohomology of GC-hbar} {\em  $H^0( \GC_3^{or}[[\hbar]], d_\hbar)\simeq
H^0(\GC_3^{or},d_0)\simeq \grt_1$ as Lie algebras. Moreover, $H^i( \GC_3^{or}[[\hbar]], d_\hbar)=0$
for all $i\leq -2$ and $H^{-1}( \GC_3^{or}[[\hbar]], d_\hbar)\cong \K$, with the single class being represented by
\[
\sum_{k=2}^\infty (k-1)\hbar^{k-2} \underbrace{
\Ba{c}\resizebox{7mm}{!}{ \xy
(0,0)*{...},
   \ar@/^1pc/(0,5)*{\bullet};(0,-5)*{\bullet}
   \ar@/^{-1pc}/(0,5)*{\bullet};(0,-5)*{\bullet}
   \ar@/^0.6pc/(0,5)*{\bullet};(0,-5)*{\bullet}
   \ar@/^{-0.6pc}/(0,5)*{\bullet};(0,-5)*{\bullet}
 \endxy}
 \Ea}_{k\ \mathrm{edges}}
\]
.}
\begin{proof}
First note that the element above is exactly $\frac d {d\hbar} \Phi_\hbar$ and the fact that it is closed follows easily by differentiating the Maurer-Cartan equation
\[
 0 = \frac d {d\hbar} [\Phi_\hbar,\Phi_\hbar]_{\mathrm{gra}} = 2 [\Phi_\hbar,\frac d {d\hbar} \Phi_\hbar]_{\mathrm{gra}}
 =
 2d_\hbar \left( \frac d {d\hbar} \Phi_\hbar \right).
\]
It is easy to see that the cocyle $\frac d {d\hbar} \Phi_\hbar$ cannot be exact, by just considering the leading term in $\hbar$.

Let us write
$$
d_\hbar= \sum_{k=1}^\infty \hbar^{k-1} d_k.
$$
Consider a decreasing filtration of $\sG\sC_3^{or}[[\hbar]]$ by the powers in $\hbar$. The
first term of the associated spectral sequence is
$$
\cE_1= \bigoplus_{i\in \Z} \cE_1^i,\ \ \ \ \cE_1^i=\bigoplus_{p\geq 0}
H^{i-2p}(\sG\sC_3^{or}, d_0) \hbar^p
$$
with the differential equal to $\hbar d_1$.
The main result of \cite{Wi2} states that $H^0(\sG\sC_3^{or}, d_0)\simeq \fg\fr\ft_1$, $H^{\leq -2}(\sG\sC_3^{or}, d_0)=0$ and $H^{-1}(\sG\sC_3^{or}, d_0)\cong \K$. The desired results follow by degree reasons.
\end{proof}

\subsubsection{\bf Corollary} {\em The group $GRT_1$ acts on the set of all possible
$\LoB_\infty$-structures up to homotopy on an arbitrary (perhaps infinite-dimensional) differential graded vector space $\fg$}.

\subsubsection{\bf Iterative construction of graph representatives
of elements of $\fg\fr\ft_1$}\label{4: subsubsection on iterated construction}
The above Theorem {\ref{4: Theorem on cohomology of GC-hbar}} says that any degree zero graph $\Ga\in \sG\sC_3^{or}$
satisfying the cocycle condition, $d_0\Gamma_0=0$, can be extended to a formal power
series,
$$
\Ga_\hbar=\Ga_0+\hbar \Gamma_1 + \hbar^2 \Gamma_2+\ldots,
$$
satisfying the  cocycle condition
$
d_\hbar\Ga_\hbar=0.
$
Let us show how this inductive extension works in detail.
The equation $d_\hbar^2=0$ implies, for any $n\geq 0$,
$
\sum_{n=i+j \atop i,j\geq 0} d_id_j=0,
$
which in turn reads,
\Beqrn
d_0^2&=&0\\
d_0d_1+ d_1d_0&=&0\\
d_0d_2+ d_2d_0 + d_1^2&=&0\ \ \ \mbox{etc}.\\
\Eeqrn
Thus the equation $d_0\Ga_0=0$ implies
$$
0=d_1d_0 \Ga_0=- d_0d_1\Ga_0
$$
The oriented graph $d_1\Ga_0\in \sG\sC_3^{or} $ has degree $-1$ and $H^{-1}(
\sG\sC_3^{or}, d_0)\cong \K$.
Since the one cohomology class cannot be hit (its leading term has necessarily only two vertices), there exists a degree $-2$ graph $\Gamma_1$ such that
$
d_1\Ga_0=-d_0\Ga_1
$
so that
$$
d_\hbar (\Gamma_0 + \hbar \Ga_1)=0 \bmod \hbar^2
$$
Assume by induction that we constructed a degree zero polynomial,
$$
\Ga_0+\hbar \Ga_1 +\ldots + \hbar^n \Ga_n \in \sG\sC_3^{or}[[\hbar]]
$$
such that
\Beq\label{induction}
d_\hbar(\Ga_0+\hbar \Ga_1 +\ldots + \hbar^n \Ga_n)=0\bmod \hbar^{n+1}.
\Eeq
Let us show that there exists an oriented graph $\Ga_{n+1}$ of degree $-2n-2$
such that
$$
d_\hbar(\Ga_0+\hbar \Ga_1 +\ldots + \hbar^n \Ga_n + \hbar^{n+1}\Ga_{n+1})=0\bmod
\hbar^{n+2}.
$$
or, equivalently, such that
\Beq\label{n+1 induction step}
d_0\Ga_{n+1} + d_{n+1}\Ga_0 + \sum_{n+1=i+j\atop i,j\geq 1} d_i\Ga_j=0.
\Eeq

Equation (\ref{induction}) implies, for any $j\leq n$,
$$
d_0\Ga_j + d_j\Ga_0 + \sum_{j=p+q\atop p,q\geq 1} d_{p}\Ga_q=0.
$$
We have
\Beqrn
0&=& d_{n+1}d_0\Ga_0
=-d_0d_{n+1}\Ga_0 -  \sum_{n+1=i+j\atop i,j\geq 1} d_id_j\Ga_0\\
&=& -d_0d_{n+1}\Ga_0 +
 \sum_{n+1=i+j\atop i,j\geq 1} d_i d_0\Ga_j +  \sum_{n+1=i+p+q\atop i,p,q\geq 1} d_i
 d_{p}\Ga_q\\
 &=& -d_0d_{n+1}\Ga_0 -  \sum_{n+1=i+j\atop i,j\geq 1} d_0 d_i\Ga_j
 -   \sum_{n+1=i+p+q\atop i,p,q\geq 1} d_i d_{p}\Ga_q +  \sum_{n+1=i+p+q\atop i,p,q\geq 1}
 d_i d_{p}\Ga_q\\
 &=& -d_0\left(d_{n+1}\Ga_0 +  \sum_{n+1=i+j\atop i,j\geq 1}  d_i\Ga_j\right)
\Eeqrn
As $H^{-1-2n}(\sG\sC_3^{or}, d_0)=0$ for all $n\geq 1$, there exists a degree $-2-2n$ graph $\Gamma_{n+1}$
such that the required equation (\ref{n+1 induction step}) is satisfied. This completes an
inductive construction of $\Ga_\hbar$ from $\Ga_0$.

\subsection{Deformations of Frobenius algebra structures}
Note that the complexes $\Def(\invFrob_\infty\to \invFrob)$ and $\Def(\LieBi_\infty\to \hLieBi)$ are isomorphic. We hence have a zigzag of (quasi-)isomorphisms of complexes
\[
 \Der(\invFrob_\infty) \to \Def(\invFrob_\infty\to \invFrob)[1]\cong \Def(\LieBi_\infty\to \hLieBi)[1]\leftarrow \Der(\hLieBi_\infty).
\]
In particular we obtain a map\footnote{In fact, the first arrow is an injection and almost an isomorphism by Theorem {\ref{thm:Fqiso}}.}
\[
 \grt_1 \to H^0(\Der(\hLieBi_\infty))\cong H^0(\Der(\invFrob_\infty)).
\]
Hence we obtain a large class of homotopy non-trivial derivations of the properad $\invFrob_\infty$ and accordingly a large class of potentially homotopy non-trivial universal deformations of any $\invFrob_\infty$ algebra.

\begin{corollary}
There is a map $\grt_1\to H^0(\Der(\invFrob_\infty))$ and hence a map $\grt_1\to H^1(\Def(\invFrob_\infty\to \End_A))$ for any $\invFrob_\infty$ algebra $A$.
\end{corollary}

Next consider the Frobenius properad $\Frob$ and let $\hFrob$ be its genus completion.
Analogously to section {\ref{sec:completed versions}} let $\hFrob_\infty$ be the completion of $\Frob_\infty$ with respect to the total genus and let $\Der(\hFrob_\infty)$ be the continuous derivations.
Note that the complex $\Def(\Frob_\infty\to \hFrob)$ is isomorphic to the complex $\Def(\LoB_\infty\to \hLoB)$.
We hence obtain a zigzag of quasi-isomorphisms
\[
  \Der(\hFrob_\infty) \to \Def(\Frob_\infty\to \hFrob)[1]\cong \Def(\LoB_\infty\to \hLoB)[1]\leftarrow \Der(\hLoB_\infty).
\]
In particular we obtain a map
\[
 \grt_1 \to H^0(\Der(\hLoB_\infty))\cong H^0(\Der(\hFrob_\infty)).
\]
Consider the explicit construction of representatives of $\grt_1$-elements of section {\ref{4: subsubsection on iterated construction}}.
The $\hbar^n$-correction term $\Gamma_n$ to some graph cohomology class $\Gamma$ of genus $g$ has genus $g+n$. It follows that the map $\grt_1\to H^0(\Der(\hFrob_\infty))$ in fact factors through $H^0(\Der(\Frob_\infty))$ and in particular we have a map
\[
\grt_1\to  H^0(\Der(\Frob_\infty))
\]
and hence a map from $\grt_1$ into the deformation complex of any $\Frob_\infty$ algebra.

\begin{corollary}
There is a map $\grt_1\to H^0(\Der(\Frob_\infty))$ and hence a map $\grt_1\to H^1(\Def(\Frob_\infty\to \End_A))$ for any $\Frob_\infty$ algebra $A$.
\end{corollary}

\bip

{\large
\section{\bf Involutive Lie bialgebras as homotopy Batalin-Vilkovisky algebras}
}

Let $\fg$ be a $\LieBi$ algebra. Then it is a well known fact the Chevalley-Eilenberg complex of $\fg$ (as a Lie coalgebra) $CE(\fg)=\odot^\bu\fg[-1]$ carries a Gerstenhaber algebra structure.
Concretely, the commutative algebra structure is the obvious one.
To define the Lie bracket it is sufficient to define it on the generators $\fg[-1]$, where it is given by the Lie bracket on $\fg$.

\sip

Similarly, if $\fg$ is an $\LoB$ algebra, then $CE(\fg)=\odot^{\bu}\fg[-1]$ carries a natural Batalin-Vilkovisky (BV) algebra structure.
The product and Lie bracket are as before. The BV operator $\Delta$ is defined on a word $x_1\cdots x_n$ as
\[
 \Delta (x_1\cdots x_n) = -\sum_{i<j} (-1)^{i+j} [x_i,x_j]x_1\cdots\hat x_i \cdots\hat x_j \cdots x_n.
\]
The involutivity condition is needed for the BV operator to be compatible with the differential.
Now suppose that $\fg$ is a $\LieBi_\infty$ algebra.
We call it \emph{good} if for any fixed $m$ only finitely many of the operations $\mu_{m,n}\in \Hom(\fg^{\otimes m}, \fg^{\otimes n})$ are non-zero.
Then one may define the Chevalley complex $CE(\fg)=\odot^{\bu}\fg[-1]$ of $\fg$ as an $\caL ie_\infty$ coalgebra.
It is known (see, e.~g. Remark 1 of \cite{Wi2}) that $CE(\fg)=\odot^{\bu}\fg[-1]$ carries a natural homotopy Gerstenhaber structure.
In this section we show that similarly, if $\fg$ is a good $\LoB_\infty$ algebra, then the Chevalley-Eilenberg complex $CE(\fg)$ carries a natural homotopy BV algebra structure.

\bip

\subsection{The order of an operator}\label{5: subsec on order of operators} Let $V$ be a graded commutative algebra. For a linear operator $D: V\rar V$ define a collection,
$$
\Ba{rccc}
F_n^D: & \ot^n V &\lon & V\\
 & v_1\ot\ldots\ot v_n & \lon & F_n^D(v_1,\ldots,v_n)
\Ea
$$
of linear maps by induction: $F_1^D=D$,
$$
F_{n+1}^D(v_1,\ldots,v_n, v_{n+1})= F_n^D(v_1,\ldots,v_n\cdot v_{n+1}) -
F_n^D(v_1,\ldots,v_n)\cdot v_{n+1} - (-1)^{|v_n||v_{n+1}} F_n^D(v_1,\ldots,v_{n+1})\cdot v_{n}.
$$
The operator $D$ is said to have {\em order}\, $\leq n$ if $F_{n+1}^D=0$.

\mip

 The operators
$F_n^{D}$ are in fact graded symmetric; moreover, if $D$ is a differential in $V$ (that is, $|D|=1$ and $D^2=0$), then the collection, $\{F_n^D: \odot^n V\rar  V\}_{n\geq 1}$  defines a  $\caL ie_\infty$-structure on the space $V[-1]$ (see \cite{Kr}). Indeed, consider
a graded Lie algebra,
$$
\mathrm{CoDer}(\odot^{\bu \geq 1} V)\cong \prod_{n\geq 1} \Hom_\K(\odot^n V,V),
$$
of coderivations of the graded co-commutative coalgebra $\odot^{\bu\geq 1} V$. As the differential $D:V\rar V$ is a Maurer-Cartan element in this Lie algebra and the multiplication $\mu:\odot^2 V\rar V$
is its degree zero element, we can gauge transform $D$,
$$
D\lon F^D:=e^{-\mu} D e^{\mu}=\sum_{n=0}^\infty \frac{1}{n!}[\ldots[[D,\mu],\mu], \ldots, \mu],
$$
into a less trivial  Maurer-Cartan element which makes $V[-1]$ into a (less trivial) $\caL ie_\infty$-algebra. It is easy to see that
the associated components of the codifferential $F^D$,
$$
 F^D=\left\{F^D_{n+1}=\frac{1}{n!}[\ldots[[D,\mu],\mu], \ldots, \mu]: \odot^{n+1} V\rar V\right\}_{n\geq 0}
$$
coincide precisely with the defined above tensors $F^D_{n+1}$ which measure a failure of $D$ to respect the multiplication operation in $V$.

\subsection{Batalin-Vilkovisky algebras} A {\em Batalin-Vilkovisky algebra}\, is, by definition, a graded commutative algebra $V$ equipped with a degree $-1$ operator
$\Delta:V \rar V$ of order $\leq 2$ such that $\Delta^2=0$. Denote by $\cB\cV$ the operad whose representations
are Batalin-Vilkovisky algebras. This is, therefore,  a graded operad generated by corollas,
$$
\Ba{c}\resizebox{3.6mm}{!}{ \xy
(0,5)*{};
(0,0)*+{_1}*\cir{}
**\dir{-};
(0,-5)*{};
(0,0)*+{_1}*\cir{}
**\dir{-};
\endxy}\Ea\ \ \ \ \mbox{and}\ \ \ \
\begin{xy}
 <0mm,0.66mm>*{};<0mm,3mm>*{}**@{-},
 <0.39mm,-0.39mm>*{};<2.2mm,-2.2mm>*{}**@{-},
 <-0.35mm,-0.35mm>*{};<-2.2mm,-2.2mm>*{}**@{-},
 <0mm,0mm>*{\circ};<0mm,0mm>*{}**@{},
   <0.39mm,-0.39mm>*{};<2.9mm,-4mm>*{^2}**@{},
   <-0.35mm,-0.35mm>*{};<-2.8mm,-4mm>*{^1}**@{},
\end{xy}=
\begin{xy}
 <0mm,0.66mm>*{};<0mm,3mm>*{}**@{-},
 <0.39mm,-0.39mm>*{};<2.2mm,-2.2mm>*{}**@{-},
 <-0.35mm,-0.35mm>*{};<-2.2mm,-2.2mm>*{}**@{-},
 <0mm,0mm>*{\circ};<0mm,0mm>*{}**@{},
   <0.39mm,-0.39mm>*{};<2.9mm,-4mm>*{^1}**@{},
   <-0.35mm,-0.35mm>*{};<-2.8mm,-4mm>*{^2}**@{},
\end{xy}
$$
of homological degrees $-1$ and $0$ respectively, modulo the following relations,
$$
\Ba{c}\resizebox{5mm}{!}{ \xy
(0,0)*+{_1}*\cir{}="b",
(0,6)*+{_1}*\cir{}="c",
%
(0,-4)*{}="-1",
(0,10)*{}="1'",
\ar @{-} "b";"c" <0pt>
\ar @{-} "b";"-1" <0pt>
\ar @{-} "c";"1'" <0pt>
\endxy}
\Ea=0\ \ \ ,\ \ \ \Ba{c}\begin{xy}
 <0mm,0mm>*{\circ};<0mm,0mm>*{}**@{},
 <0mm,0.69mm>*{};<0mm,3.0mm>*{}**@{-},
 <0.39mm,-0.39mm>*{};<2.4mm,-2.4mm>*{}**@{-},
 <-0.35mm,-0.35mm>*{};<-1.9mm,-1.9mm>*{}**@{-},
 <-2.4mm,-2.4mm>*{\circ};<-2.4mm,-2.4mm>*{}**@{},
 <-2.0mm,-2.8mm>*{};<0mm,-4.9mm>*{}**@{-},
 <-2.8mm,-2.9mm>*{};<-4.7mm,-4.9mm>*{}**@{-},
    <0.39mm,-0.39mm>*{};<3.3mm,-4.0mm>*{^3}**@{},
    <-2.0mm,-2.8mm>*{};<0.5mm,-6.7mm>*{^2}**@{},
    <-2.8mm,-2.9mm>*{};<-5.2mm,-6.7mm>*{^1}**@{},
 \end{xy}\Ea
\ = \
 \Ba{c}\begin{xy}
 <0mm,0mm>*{\circ};<0mm,0mm>*{}**@{},
 <0mm,0.69mm>*{};<0mm,3.0mm>*{}**@{-},
 <0.39mm,-0.39mm>*{};<2.4mm,-2.4mm>*{}**@{-},
 <-0.35mm,-0.35mm>*{};<-1.9mm,-1.9mm>*{}**@{-},
 <2.4mm,-2.4mm>*{\circ};<-2.4mm,-2.4mm>*{}**@{},
 <2.0mm,-2.8mm>*{};<0mm,-4.9mm>*{}**@{-},
 <2.8mm,-2.9mm>*{};<4.7mm,-4.9mm>*{}**@{-},
    <0.39mm,-0.39mm>*{};<-3mm,-4.0mm>*{^1}**@{},
    <-2.0mm,-2.8mm>*{};<0mm,-6.7mm>*{^2}**@{},
    <-2.8mm,-2.9mm>*{};<5.2mm,-6.7mm>*{^3}**@{},
 \end{xy}\Ea,\ \ \
 \Ba{c}\resizebox{10mm}{!}{ \xy
(-6,-15.6)*+{_1};
(0,-15.6)*+{_2};
(3,-11.6)*+{_3};
(0,-1)*+{_1}*\cir{}="b",
(0,-6)*{\circ}="c1",
(-3,-10)*{\circ}="c2",
(3,-10)*{}="-3'",
(-6,-14)*{}="-1'",
(0,-14)*{}="-2'",
(0,5)*{}="1'",
\ar @{-} "b";"1'" <0pt>
\ar @{-} "b";"c1" <0pt>
\ar @{-} "c1";"c2" <0pt>
\ar @{-} "c1";"-3'" <0pt>
\ar @{-} "c2";"-1'" <0pt>
\ar @{-} "c2";"-2'" <0pt>
\endxy}\Ea
=
\sum_{n=0}^2\Ba{c}\resizebox{16mm}{!}{
\xy
(-6.9,-15.9)*+{_{_{\zeta^n(1)}}};
(2,-15.9)*+{_{_{\zeta^n(2)}}};
(7,-7.9)*+{_{\zeta^n(3)}};
(0,-6)*+{_1}*\cir{}="c1",
(3,-1)*{\circ}="b",
(-3,-10)*{\circ}="c2",
(6,-6)*{}="-3'",
(-6,-14)*{}="-1'",
(0,-14)*{}="-2'",
(3,4)*{}="1'",
\ar @{-} "b";"1'" <0pt>
\ar @{-} "b";"c1" <0pt>
\ar @{-} "c1";"c2" <0pt>
\ar @{-} "b";"-3'" <0pt>
\ar @{-} "c2";"-1'" <0pt>
\ar @{-} "c2";"-2'" <0pt>
\endxy}\Ea
-
\sum_{n=0}^2
 \Ba{c}\resizebox{16mm}{!}{ \xy
(-6,-20.6)*+{_{_{\zeta^n(1)}}};
(1,-15.6)*+{_{_{\zeta^n(2)}}};
(5,-11.6)*+{_{_{\zeta^n(3)}}};
(-6,-14)*+{_1}*\cir{}="b",
(0,-6)*{\circ}="c1",
(-3,-10)*{\circ}="c2",
(3,-10)*{}="-3'",
(-6,-19)*{}="-1'",
(0,-14)*{}="-2'",
(0,-1)*{}="1'",
\ar @{-} "c1";"1'" <0pt>
\ar @{-} "b";"c2" <0pt>
\ar @{-} "c1";"c2" <0pt>
\ar @{-} "c1";"-3'" <0pt>
\ar @{-} "b";"-1'" <0pt>
\ar @{-} "c2";"-2'" <0pt>
\endxy}\Ea
$$
where $\zeta$ is the cyclic permutation $(123)$.
 A nice non-minimal cofibrant resolution of the operad $\cB\cV$ has been constructed in \cite{GTV}. We denote this resolution by $\cB\cV^K_\infty$ in this paper, $K$ standing for {\em Koszul}. The minimal resolution, $\cB\cV_\infty$, has been constructed in \cite{DV}.

\subsection{An operad ${\cB\cV}_\infty^{com}$}
A ${\BV}_\infty^{com}$-algebra is, by definition \cite{Kr}, a dg graded commutative algebra  $(V,d)$
equipped with a countable collections of linear homogeneous maps, $\{\Delta_a: V\rar V,\
    |\Delta_a|=1-2a\}_{k\geq 1}$, such that each $\Delta_a$ is of order $\leq a+1$ and the equations,
    \Beq\label{5: BV_comm equation for Delta}
    \sum_{a=0}^n \Delta_a \circ \Delta_{n-a}=0,
    \Eeq
hold for any $n\in \N$,
    where $\Delta_0:=-d$.

\mip

Let $\BV_\infty^{com}$ be a dg operad of $\BV_\infty^{com}$-algebras. This operad is a quotient
of the free operad generated by one binary operation in degree zero,
$
\begin{xy}
 <0mm,0.66mm>*{};<0mm,3mm>*{}**@{-},
 <0.39mm,-0.39mm>*{};<2.2mm,-2.2mm>*{}**@{-},
 <-0.35mm,-0.35mm>*{};<-2.2mm,-2.2mm>*{}**@{-},
 <0mm,0mm>*{\circ};<0mm,0mm>*{}**@{},
   <0.39mm,-0.39mm>*{};<2.9mm,-4mm>*{^2}**@{},
   <-0.35mm,-0.35mm>*{};<-2.8mm,-4mm>*{^1}**@{},
\end{xy}=
\begin{xy}
 <0mm,0.66mm>*{};<0mm,3mm>*{}**@{-},
 <0.39mm,-0.39mm>*{};<2.2mm,-2.2mm>*{}**@{-},
 <-0.35mm,-0.35mm>*{};<-2.2mm,-2.2mm>*{}**@{-},
 <0mm,0mm>*{\circ};<0mm,0mm>*{}**@{},
   <0.39mm,-0.39mm>*{};<2.9mm,-4mm>*{^1}**@{},
   <-0.35mm,-0.35mm>*{};<-2.8mm,-4mm>*{^2}**@{},
\end{xy}
$,
and a countable family of unary operations,
$
\left\{ \Ba{c}\resizebox{3.1mm}{!}{  \xy
(0,5)*{};
(0,0)*+{_a}*\cir{}
**\dir{-};
(0,-5)*{};
(0,0)*+{_a}*\cir{}
**\dir{-};
\endxy}\Ea \right\}_{a\geq 1}
$
(of homological degree $1-2a$), modulo the ideal $I$ generated by the associativity relations
for the binary operation $\begin{xy}
 <0mm,0.66mm>*{};<0mm,3mm>*{}**@{-},
 <0.39mm,-0.39mm>*{};<2.2mm,-2.2mm>*{}**@{-},
 <-0.35mm,-0.35mm>*{};<-2.2mm,-2.2mm>*{}**@{-},
 <0mm,0mm>*{\circ};<0mm,0mm>*{}**@{},
\end{xy}$ and the compatibility relations between the latter and unary operations
encoding the requirement that each unary operation $\Ba{c}\resizebox{2.9mm}{!}{\xy
(0,5)*{};
(0,0)*+{_a}*\cir{}
**\dir{-};
(0,-5)*{};
(0,0)*+{_a}*\cir{}
**\dir{-};
\endxy}\Ea$  is of order $\leq a+1$ with
respect to the multiplication operation.
The differential $\delta$ in the operad $\BV_\infty^{com}$ is given by
$$
\delta\begin{xy}
 <0mm,0.66mm>*{};<0mm,3mm>*{}**@{-},
 <0.39mm,-0.39mm>*{};<2.2mm,-2.2mm>*{}**@{-},
 <-0.35mm,-0.35mm>*{};<-2.2mm,-2.2mm>*{}**@{-},
 <0mm,0mm>*{\circ};<0mm,0mm>*{}**@{},
   <0.39mm,-0.39mm>*{};<2.9mm,-4mm>*{^2}**@{},
   <-0.35mm,-0.35mm>*{};<-2.8mm,-4mm>*{^1}**@{},
\end{xy}=0, \ \ \ \
\delta\ \Ba{c}\resizebox{4mm}{!}{ \xy
(0,5)*{};
(0,0)*+{_a}*\cir{}
**\dir{-};
(0,-5)*{};
(0,0)*+{_a}*\cir{}
**\dir{-};
\endxy}\Ea: =\sum_{a=b+c\atop b,c\geq 1}\Ba{c} \resizebox{4mm}{!}{
\xy
(0,6,3)*{};
(0,0)*+{_c}*\cir{}
**\dir{-};
(0,-5)*{};
(0,0)*+{_c}*\cir{}
**\dir{-};
(0,13)*{};
(0,8)*+{_b}*\cir{}
**\dir{-};
\endxy}
\Ea
$$
There is an explicit morphism of dg operads (see  Proposition 23 in \cite{GTV})\footnote{We are grateful to Bruno Vallette for pointing out this result to us.},
$$
 \cB\cV_\infty^K \lon \cB\cV^{com}_\infty,
$$
which implies existence of a morphism of dg operads $\cB\cV_\infty \rar \cB\cV^{com}_\infty$. The existence of such a morphism  follows also from the following Theorem
whose proof is given in Appendix \ref{app:BV_com proof}.

\begin{theorem}\label{5: Theorem on cohom BV_com} { The dg operad $\BV_\infty^{com}$ is formal with the cohomology
operad $H^\bu(\BV_\infty^{com})$ isomorphic to the operad, $\BV$, of
Batalin-Vilkovisky algebras, i.e.\ there is a canonical surjective quasi-isomorphism of
operads,
$$
\pi: \BV_\infty^{com} \lon \BV
$$
which sends to zero all generators $\Ba{c}\resizebox{4mm}{!}{ \xy
(0,5)*{};
(0,0)*+{_a}*\cir{}
**\dir{-};
(0,-5)*{};
(0,0)*+{_a}*\cir{}
**\dir{-};
\endxy}\Ea$ with $a\geq 2$.
}
\end{theorem}

\subsection{From strongly homotopy involutive  Lie bialgebras to $\cB\cV_\infty$-algebras}

We call a $\LoB_\infty$ algebra $\fg$ \emph{good} if for any fixed $m$ and $k$ only finitely many of the operations $\mu_{m,n}^k\in \Hom(\fg^{\otimes m}, \fg^{\otimes n})$ are non-zero.
In this case we define the Chevalley-Eilenberg complex $CE(\fg) = {\odot}^\bu(\fg[-1])$ of $\fg$ as an $\caL ie_\infty$ coalgebra.
More concretely, for a finite dimensional $\fg$ we may understand the $\LoB_\infty$ algebra structure as a formal power series
$\Ga_\hbar=\Ga_\hbar(\psi_i,\eta^i,\hbar)$ as explained in section \ref{3: Section on Def complex}.
Using similar notation, we may the space $CE(\fg)$ as the space of polynomials in the variables $\psi_i$.
Then the differential $\Delta_0$ on $CE(\fg)$ is given by the formula
\[
\Delta_0:=\sum_{i}\frac{\p \Ga_\hbar}{\p \eta^{i}}|_{\hbar=\eta^i=0}\frac{\p}{\p \psi_i} .
\]

\begin{proposition}
Let $\fg$ be a good $\LoB_\infty$ algebra, with the $\LoB_\infty$ algebra structure being defined a power series $\Ga_\hbar=\Ga_\hbar(\psi_i,\eta^i,\hbar)$ as explained in section \ref{3: Section on Def complex}. Then there is a natural $\BV_\infty^{com}$ algebra structure $\rho$ on the complex $CE(\fg)$ given by the formulas:

$$
\rho\left(\begin{xy}
 <0mm,0.66mm>*{};<0mm,3mm>*{}**@{-},
 <0.39mm,-0.39mm>*{};<2.2mm,-2.2mm>*{}**@{-},
 <-0.35mm,-0.35mm>*{};<-2.2mm,-2.2mm>*{}**@{-},
 <0mm,0mm>*{\circ};<0mm,0mm>*{}**@{},
\end{xy}\right):=\mathrm{the\ standard\ multiplication\ in}\
{\odot}^\bu(\fg[-1])[[\hbar]]
$$
and, for any $a\geq 1$,
$$
\rho\left(\Ba{c}\resizebox{4mm}{!}{ \xy
(0,5)*{};
(0,0)*+{_a}*\cir{}
**\dir{-};
(0,-5)*{};
(0,0)*+{_a}*\cir{}
**\dir{-};
\endxy}\Ea\right):= \sum_{p+k=a+1} \frac{1}{p!k!}\sum_{i_1,\ldots,i_{c+1}} \frac{\p^{a+1} \Ga_\hbar}{\p^p\hbar \p \eta^{i_1}\cdots
\p \eta^{i_{k}}}|_{\hbar=\eta^i=0}\frac{\p^{k} }{\p \psi_{i_1}\cdots
\p \psi_{i_{k}}}
$$
\end{proposition}


\begin{proof} It is clear that $\Delta_{a}:=\rho\left(\Ba{c}\resizebox{4mm}{!}{ \xy
(0,5)*{};
(0,0)*+{a}*\cir{}
**\dir{-};
(0,-5)*{};
(0,0)*+{a}*\cir{}
**\dir{-};
\endxy}\Ea\right)$ is an operator of order $\leq a+1$ with respect to the standard multiplication in
the graded commutative algebra
${\odot}^\bu(\fg[-1])$. The verficiation that the operators $\{\Delta_a\}_{a\geq 0}$ satisfy identities \eqref{5: BV_comm equation for Delta}
is best done pictorially. We represent the expression on the right hand side by the picture
$$
\rho\left(\Ba{c}\resizebox{4mm}{!}{ \xy
(0,5)*{};
(0,0)*+{a}*\cir{}
**\dir{-};
(0,-5)*{};
(0,0)*+{a}*\cir{}
**\dir{-};
\endxy}\Ea\right)=  \sum_{a+1=p+k\atop k\geq 1,p\geq 0}
\underbrace{
\Ba{c}\resizebox{7mm}{!}{ \xy
(0.0,-12.2)*{_1},
(0,-5)*{...},
   \ar@/^1pc/(0,0)*+{_{p}}*\frm{o};(0,-10)*{\circ}*\frm{}
   \ar@/^{-1pc}/(0,0)*+{_{p}}*\frm{o};(0,-10)*{\circ}*\frm{}
   \ar@/^0.6pc/(0,0)*+{_{p}}*\frm{o};(0,-10)*{\circ}*\frm{}
   \ar@/^{-0.6pc}/(0,0)*+{_{p}}*\frm{o};(0,-10)*{\circ}*\frm{}
 \endxy}
 \Ea}_{k\ \mathrm{edges}}
$$
Then we compute
\Beqrn
\rho\left(\delta \Ba{c}\resizebox{4mm}{!}{ \xy
(0,5)*{};
(0,0)*+{_a}*\cir{}
**\dir{-};
(0,-5)*{};
(0,0)*+{_a}*\cir{}
**\dir{-};
\endxy}\Ea\right)& =&\sum_{a=b+c\atop b,c\geq 1}\rho\left(
\Ba{c}\resizebox{3mm}{!}{  \xy
(0,6,3)*{};
(0,0)*+{_c}*\cir{}
**\dir{-};
(0,-5)*{};
(0,0)*+{_c}*\cir{}
**\dir{-};
(0,13)*{};
(0,8)*+{_b}*\cir{}
**\dir{-};
\endxy}
\Ea\right)
= \sum_{a=b+c\atop b,c\geq 1} \sum_{b+1=p+k\atop k\geq 1,p\geq 0} \sum_{c+1=q+l\atop k\geq 1,q\geq 0}
\rho\left(
\Ba{c}\resizebox{7mm}{!}{ \xy
(0.0,-12.2)*{_1},
(0,-3.5)*{_k},
(0,-5)*{...},
   \ar@/^1pc/(0,0)*+{_{p}}*\frm{o};(0,-10)*{\circ}*\frm{}
   \ar@/^{-1pc}/(0,0)*+{_{p}}*\frm{o};(0,-10)*{\circ}*\frm{}
   \ar@/^0.6pc/(0,0)*+{_{p}}*\frm{o};(0,-10)*{\circ}*\frm{}
   \ar@/^{-0.6pc}/(0,0)*+{_{p}}*\frm{o};(0,-10)*{\circ}*\frm{}
 \endxy}
 \Ea\right) \circ_1
 \rho\left(
\Ba{c}\resizebox{7mm}{!}{ \xy
(0.0,-12.2)*{_1},
(0,-3.5)*{_l},
(0,-5)*{...},
   \ar@/^1pc/(0,0)*+{_{q}}*\frm{o};(0,-10)*{\circ}*\frm{}
   \ar@/^{-1pc}/(0,0)*+{_{q}}*\frm{o};(0,-10)*{\circ}*\frm{}
   \ar@/^0.6pc/(0,0)*+{_{q}}*\frm{o};(0,-10)*{\circ}*\frm{}
   \ar@/^{-0.6pc}/(0,0)*+{_{q}}*\frm{o};(0,-10)*{\circ}*\frm{}
 \endxy}
 \Ea\right) \\
 &=&
  \sum_{a=b+c\atop b,c\geq 1} \sum_{b+1=p+k\atop k,p\geq 1} \sum_{c+1=q+l\atop k,p\geq 1}
  \sum_{k=k'+k''} \Ba{c}\resizebox{10mm}{!}{ \xy
(0,-5)*{\stackrel{l}{...}},
(0,5)*{\stackrel{k'}{...}},
(8,0)*{\stackrel{k''}{...}},
   \ar@/^{-1pc}/(0,0)*+{_{q}}*\frm{o};(0,-10)*{\circ}
   \ar@/^0.6pc/(0,0)*+{_{q}}*\frm{o};(0,-10)*{\circ}
   \ar@/^{-0.6pc}/(0,0)*+{_{q}}*\frm{o};(0,-10)*{\circ}
   \ar@/^{-1pc}/(0,10)*+{_{p}}*\frm{o};(0,0)*+{_{q}}*\frm{o}
   \ar@/^0.6pc/(0,10)*+{_{p}}*\frm{o};(0,0)*+{_{q}}*\frm{o}
   \ar@/^{-0.6pc}/(0,10)*+{_{p}}*\frm{o};(0,0)*+{_{q}}*\frm{o}
   \ar@/^{2.4pc}/(0,10)*+{_{p}}*\frm{o};(0,-10)*{\circ}*\frm{}
   \ar@/^{1.3pc}/(0,10)*+{_{p}}*\frm{o};(0,-10)*{\circ}*\frm{}
 \endxy}
 \Ea
=  \sum_{a=p+q+k'+k''+l-2\atop {p,q,k',k''\geq 0, k'+k'',l\geq 1,
\atop k'+k''+p\geq 2, l+q\geq 2}}
\Ba{c}\resizebox{10mm}{!}{ \xy
(0,-5)*{\stackrel{l}{...}},
(0,5)*{\stackrel{k'}{...}},
(8,0)*{\stackrel{k''}{...}},
   \ar@/^{-1pc}/(0,0)*+{_{q}}*\frm{o};(0,-10)*{\circ}
   \ar@/^0.6pc/(0,0)*+{_{q}}*\frm{o};(0,-10)*{\circ}
   \ar@/^{-0.6pc}/(0,0)*+{_{q}}*\frm{o};(0,-10)*{\circ}
   \ar@/^{-1pc}/(0,10)*+{_{p}}*\frm{o};(0,0)*+{_{q}}*\frm{o}
   \ar@/^0.6pc/(0,10)*+{_{p}}*\frm{o};(0,0)*+{_{q}}*\frm{o}
   \ar@/^{-0.6pc}/(0,10)*+{_{p}}*\frm{o};(0,0)*+{_{q}}*\frm{o}
   \ar@/^{2.4pc}/(0,10)*+{_{p}}*\frm{o};(0,-10)*{\circ}*\frm{}
   \ar@/^{1.3pc}/(0,10)*+{_{p}}*\frm{o};(0,-10)*{\circ}*\frm{}
 \endxy}
 \Ea\\
\Eeqrn
On the other hand 
\Beqrn
\delta \rho\left(
\Ba{c}\resizebox{4mm}{!}{ \xy
(0,5)*{};
(0,0)*+{a}*\cir{}
**\dir{-};
(0,-5)*{};
(0,0)*+{a}*\cir{}
**\dir{-};
\endxy}\Ea\right) &=&  \sum_{a+1=p+k\atop k\geq 1,p\geq 0} \delta
\Ba{c}\resizebox{7mm}{!}{ \xy
(0.0,-12.2)*{_1},
(0,-3.5)*{_k},
(0,-5)*{...},
   \ar@/^1pc/(0,0)*+{_{p}}*\frm{o};(0,-10)*{\circ}*\frm{}
   \ar@/^{-1pc}/(0,0)*+{_{p}}*\frm{o};(0,-10)*{\circ}*\frm{}
   \ar@/^0.6pc/(0,0)*+{_{p}}*\frm{o};(0,-10)*{\circ}*\frm{}
   \ar@/^{-0.6pc}/(0,0)*+{_{p}}*\frm{o};(0,-10)*{\circ}*\frm{}
 \endxy}
 \Ea
 =
 \sum_{}
 \Ba{c}\resizebox{12mm}{!}{ \xy
(0,-5)*{\stackrel{k'}{...}},
(0,5)*{\stackrel{l}{...}},
(8,0)*{\stackrel{k''}{...}},
   \ar@/^{-1pc}/(0,0)*+{_{p''}}*\frm{o};(0,-10)*{\circ}
   \ar@/^0.6pc/(0,0)*+{_{p''}}*\frm{o};(0,-10)*{\circ}
   \ar@/^{-0.6pc}/(0,0)*+{_{p''}}*\frm{o};(0,-10)*{\circ}
   \ar@/^{-1pc}/(0,10)*+{_{p'}}*\frm{o};(0,0)*+{_{p''}}*\frm{o}
   \ar@/^0.6pc/(0,10)*+{_{p'}}*\frm{o};(0,0)*+{_{p''}}*\frm{o}
   \ar@/^{-0.6pc}/(0,10)*+{_{p'}}*\frm{o};(0,0)*+{_{p''}}*\frm{o}
   \ar@/^{2.4pc}/(0,10)*+{_{p'}}*\frm{o};(0,-10)*{\circ}*\frm{}
   \ar@/^{1.3pc}/(0,10)*+{_{p'}}*\frm{o};(0,-10)*{\circ}*\frm{}
 \endxy}
 \Ea
 +
  \sum_{}
 \Ba{c}\resizebox{12mm}{!}{ \xy
(0,5)*{\stackrel{k'}{...}},
(8,0)*{\stackrel{k''}{...}},
   \ar@/^{-0.0pc}/(0,0)*+{_{0}}*\frm{o};(0,-10)*{\circ}
   \ar@/^{-1pc}/(0,10)*+{_{p}}*\frm{o};(0,0)*+{_{0}}*\frm{o}
   \ar@/^0.6pc/(0,10)*+{_{p}}*\frm{o};(0,0)*+{_{0}}*\frm{o}
   \ar@/^{-0.6pc}/(0,10)*+{_{p}}*\frm{o};(0,0)*+{_{0}}*\frm{o}
   \ar@/^{2.4pc}/(0,10)*+{_{p}}*\frm{o};(0,-10)*{\circ}*\frm{}
   \ar@/^{1.3pc}/(0,10)*+{_{p}}*\frm{o};(0,-10)*{\circ}*\frm{}
 \endxy}
 \Ea
 +
 \Ba{c}\resizebox{7mm}{!}{ \xy
(0,-5)*{\stackrel{k}{...}},
   \ar@/^{-1pc}/(0,0)*+{_{p}}*\frm{o};(0,-10)*{\circ}
   \ar@/^0.6pc/(0,0)*+{_{p}}*\frm{o};(0,-10)*{\circ}
   \ar@/^{-0.6pc}/(0,0)*+{_{p}}*\frm{o};(0,-10)*{\circ}
   \ar@/^{-0pc}/(0,10)*+{_{0}}*\frm{o};(0,0)*+{_{p}}*\frm{o}
   %
 \endxy}
 \Ea
  +
 \Ba{c}\resizebox{12mm}{!}{ \xy
(0,-5)*{\stackrel{k}{...}},
   \ar@/^{-1pc}/(0,0)*+{_{p}}*\frm{o};(0,-10)*{\circ}
   \ar@/^0.6pc/(0,0)*+{_{p}}*\frm{o};(0,-10)*{\circ}
   \ar@/^{-0.6pc}/(0,0)*+{_{p}}*\frm{o};(0,-10)*{\circ}
   \ar@/^{2pc}/(0,10)*+{_{0}}*\frm{o};(0,-10)*{\circ}
   %
 \endxy}
 \Ea
\Eeqrn

Notice that the three terms on the right just kill those terms that were excluded before due to the restrictions $k'+k''+p\geq 2, l+q\geq 2$.
\end{proof}

\bip

\bip
\appendix


\renewcommand{\thesection}{{\Alph{section}}}
\renewcommand{\thesubsection}{{\bf\Alph{section}.\arabic{subsection}}}
\renewcommand{\thesubsubsection}{\bf\Alph{section}.\arabic{subsection}.\arabic{subsubsection}}

\section{ Proof of Proposition \ref{2: toy problem}}\label{app:koszulnessproof}
In this section we show that the quadratic algebra $\cA_n$ of section {\ref{sec:extracomplexes}} is Koszul. In fact, we will show the equivalent statement that the Koszul dual algebra $B_n=\cA_n^!$ is Koszul.
Concretely, $B_n$ is generated by $x_1,\dots, x_n$ with relations $x_ix_j=0$ for $|i-j|\neq 1$ and $x_ix_{i+1}=-x_{i+1}x_{i}$.

We denote by $C_n \otimes_{\kappa} B_n$ the Koszul complex of $B_n$, i.~e., the complex $(C_n\otimes B_n, d_{\kappa})$, where $C_n$ is the coalgebra with quadratic corelations $R = \text{span}( \{x_ix_j | |i-j|\ne 1\} \cup
\{x_ix_{i+1} + x_{i+1}x_i\})$ and $\kappa : C_n \to B_n$ is the degree $-1$ map that is zero everywhere except on $V$, where it identifies $V\subset C_n$ with $V \subset B_n$.

Notice that $B_n$ and $C_n$ are weight graded and the weight $k$ component of $B_n$, $B_n^{(k)}$ is zero if $k\geq 3$.

The result will be shown by constructing a contracting homotopy $h$. Let us denote $x_{i_1}\dots x_{i_l} \otimes [x_jx_{j+1}]$ by $x$ and let us  define $h:V^{\otimes l} \otimes B_n^{(2)} \to V^{\otimes l+1} \otimes B_n^{(1)}$, for $n\geq 1$ by \\

$\begin{cases}
 h(x) = x_{i_1}\dots x_{i_l} x_j\otimes x_{j+1} & \text{ if } |j-i_l|\ne 1 \\
 h(x) = x_{i_1}\dots x_{i_{l-1}} (x_{i_l} x_j+ x_j x_{i_l})\otimes x_{j+1} & \text{ if }|j-i_l| = 1 \text{ and } |j-i_{l-1}| \ne 1 \\
 h(x) = -x_{i_1}\dots x_{i_l} x_{j+1}\otimes x_{j} & \text{ if }|j-i_l| = 1 \text{ and } |j-i_{l-1}| = 1
\end{cases}$

\begin{lemma}
$h$ maps $RV^{\otimes l-2}\otimes B_n^{(2)}$ to $RV^{\otimes l-1}\otimes B_n^{(1)} \cap V^{\otimes l-1}R\otimes B_n^{(1)}$ and it maps $V^{\otimes a}R V^{\otimes b} \otimes B_n^{(2)} $ to $V^{\otimes a}R V^{\otimes b+1} \otimes B_n^{(1)}$.
As a consequence $h$ restricts to a function $C_n^{(l)} \otimes B_n^{(2)} \to C_n^{(l+1)}\otimes B_n^{(1)}$ and  moreover $d_\kappa h = id_{C_n^{(l)} \otimes B_n^{(2)}}$.
\end{lemma}

\begin{proof}
For the first part of the lemma notice that $RV^{\otimes l-2}$ is generated by the elements $x_{i_1}x_{i_2}\dots x_{i_l}$ such that $|i_1-i_2|\ne 1$ and the elements $(x_{i+1}x_i + x_ix_{i+1})x_{i_3}\dots x_{i_l}$ ($l\geq 2$).

If we suppose $l$ to be at least 3, it is clear that $h$ does not do anything to the initial two terms, keeping them in $R$ (thus their image will be in $RV^{\otimes l-1}\otimes B_n^{(1)}$) and by the construction of $h$ it is also clear that the image of a generator is in $V^{\otimes l-1}R\otimes B_n^{(1)}$.

In fact, the same ``not touching the $R$ part" argument also shows that $h$ maps $V^{\otimes a}R V^{\otimes b}$ to $V^{\otimes a}R V^{\otimes b+1}$ if $b$ is at least $1$.

It remains to show this for the case where $b=0$. In this case, $V^{ \ot l-2}R$ has two classes of generators: $x_{i_1}\dots x_{i_{l-1}}x_{i_l} \otimes [x_jx_{j+1}]$ where $|i_l-i_{l-1}|\ne 1$ and  $x_{i_1}\dots x_{i_{l-2}}(x_ix_{i+1}+x_{i+1}x_i) \otimes [x_jx_{j+1}]$.

For the first class of generators, if $|j-i_n|\ne 1$ the result is clear. If   $|j-i_n|= 1$ and $|j-i_{l-1}| \ne 1$ we use the second formula from the definition of $h$ and since $x_{i_{l-1}}x_{i_l} \in R$ and $x_{i_{l-1}}x_{j}\in R$ the result is also clear. If $|j-i_n| = 1 \text{ and } |j-i_{l-1}| = 1$ we use the third formula and the elements in $R$ are not touched.

For the second class of generators, this verification must be split into some cases.

Let us consider $x=x_{i_1}\dots x_{i_{l-2}}(x_ix_{i+1}+x_{i+1}x_i) \otimes [x_jx_{j+1}]$. To simplify the verification of the calculations, notice that for these generators the third formula of the definition of $h$ will never be used. If $j<i-1$ or $j>i+2$ then $h$ only moves moves $x_j$ before the tensor product and thus $h(x) \in V^{n-2}RV\otimes B_n^{(1)}$.

If $j=i-1$,

$$h(x) = x_{i_1}\dots x_{i_{l-2}} x_i x_{i+1}x_{i-1} \otimes x_i + x_{i_1}\dots x_{i_{l-2}} x_{i+1}(x_ix_{i-1}+x_{i-1}x_i)\otimes x_i \in V^{\otimes l-2}RV\otimes B_n^{(1)}$$

If $j=i$,
$$h(x) = x_{i_1}\dots x_{i_{l-2}} x_{i}(x_{i+1}x_{i}+x_{i}x_{i+1})\otimes x_{i+1} + x_{i_1}\dots x_{i_{l-2}} x_{i+1} x_{i}x_{i} \otimes x_{i+1}$$
$$= x_{i_1}\dots x_{i_{l-2}} (x_{i}x_{i+1}+x_{i+1}x_{i}) x_{i}\otimes x_{i+1} + x_{i_1}\dots x_{i_{l-2}} x_{i}x_{i}x_{i+1}\otimes x_{i+1} \in V^{\otimes l-2}RV\otimes B_n^{(1)}$$

If $j=i+1$,
$$h(x) = x_{i_1}\dots x_{i_{l-2}} x_i x_{i+1}x_{i+1} \otimes x_{i+2} + x_{i_1}\dots x_{i_{l-2}} x_{i+1}(x_ix_{i+1}+x_{i+1}x_i)\otimes x_{i+2} $$
$$=    x_{i_1}\dots x_{i_{l-2}} (x_i x_{i+1}+x_{i+1}x_i)x_{i+1} \otimes x_{i+2}    +   x_{i_1}\dots x_{i_{l-2}} x_{i+1}x_{i+1}x_i\otimes x_{i+2}             \in V^{\otimes l-2}RV\otimes B_n^{(1)}$$

If $j=i+2$,
$$h(x) = x_{i_1}\dots x_{i_{l-2}} x_{i}(x_{i+1}x_{i+2}+x_{i+2}x_{i+1})\otimes x_{i+3} + x_{i_1}\dots x_{i_{l-2}} x_{i+1} x_{i}x_{i+2} \otimes x_{i+3} \in V^{\otimes l-2}RV\otimes B_n^{(1)}$$

\end{proof}

We define $h$ on $V^{\otimes 0} \otimes B_n^{(2)} \to V^{\otimes 1} \otimes B_n^{(1)}$ by $h(1\otimes [x_jx_{j+1}]) = x_j\otimes x_{j+1}$.

To define $h:C_n^{(l)} \otimes B_n^{(1)} \to C_n^{(n+1)} \otimes B_n^{(0)}= C^{(l+1)}$, as before we define it on $V^{\otimes l} \otimes B_n^{(1)}$ and we verify that it restricts properly.

Let us denote $x_{i_1}\dots x_{i_{l-1}}x_{i_l} \otimes x_j$ by $x$ and let us  define $h:V^{\otimes l} \otimes B_n^{(1)} \to V^{\otimes l+1}$, by $h(x)=$

$$\begin{cases}
 x_{i_1}\dots x_{i_l} x_j & \text{if } |j-i_l|\ne 1 \\
 0 & \text{if } j=i_l+1 \text{ and } |i_l-i_{l-1}| \ne 1 \\
  -x_{i_1}\dots x_{i_{l-2}}x_{i_l}x_{i_{l-1}} x_{j} & \text{if } j=i_l+1 \text{, } |i_l-i_{l-1}| = 1 \text{, } |i_l-i_{l-2}|\ne 1 \\
  x_{i_1}\dots x_{i_{l-2}}x_{i_{l-1}}(x_{i_{l}} x_{j}+ x_{j}x_{i_{l}}) & \text{if } j=i_l+1 \text{, } |i_l-i_{l-1}| = 1 \text{, } |i_l-i_{l-2}|= 1 \\
 x_{i_1}\dots x_{i_{l-1}} (x_{i_{l}} x_{j}+x_jx_{i_{l}}) & \text{if } j=i_l-1 \text{ and } |j-i_{l-1}| \ne 1\\
 x_{i_1}\dots x_{i_{l-2}}(x_{i_{l-1}}x_j+x_jx_{i_{l-1}})x_{i_{l}} +  x_{i_1}\dots x_{i_l} x_j           & \text{if }j=i_l-1 \text{, } |j-i_{l-1}| = 1 \text{, } |j-i_{l-2}| \ne 1\\
 0 & \text{if }j=i_l-1 \text{, } |j-i_{l-1}| = 1 \text{, } |j-i_{l-2}| = 1
\end{cases}$$
Interpret this definition for $l\leq 2$ in the following way: Whenever $i_{l-2}$ or $i_{l-1}$ are not defined, take the case in the definition where the absolute value of the difference is different from $1$.

Notice that the image of $h$ sits inside $V^{\otimes l-1}R$.

\begin{lemma}
$h$ maps $V^{\otimes a}R V^{\otimes b} \otimes B_n^{(1)}$ to $V^{\otimes a}R V^{\otimes b+1}$.
Therefore $h$ restricts to a function $C_n^{(l)} \otimes B_n^{(1)} \to C_n^{(l+1)}$.
Moreover, $dh+hd=id$.
\end{lemma}

\begin{proof}

As before, the $R$ part in $V^{\otimes a}R V^{\otimes b}$ is not touched by $h$ if $b$ is at least $2$ hence $V^{\otimes a}R V^{\otimes b}\otimes B_n^{(1)}$ is sent to $V^{\otimes a}R V^{\otimes b+1}$.

If $b=1$ the same argument holds, except for the 6th case of the definition of $h$, when $j=i_l-1 \text{, } |j-i_{l-1}| = 1 \text{ and } |j-i_{l-2}| \ne 1$. The only way to use this case of the definition of $h$ is when we apply $h$ to the element
$$ x=x_{i_1}\dots x_{i_{l-3}}(x_{i}x_{i+1}+x_{i+1}x_{i})x_{j+1} \otimes x_j$$
and either $|j-i|=1$ (which implies $|j-(i+1)|\ne 1$) or $|j-(i+1)|=1$ (which implies $|j-i|\ne 1$).\\

If $|j-i|=1$,

\begin{align*}
h(x) &= x_{i_1}\dots x_{i_{l-3}}\left[ x_{i}x_{i+1}(x_{j+1}x_{j} + x_{j}x_{j+1}) + x_{i+1}(x_ix_j+x_jx_i)x_{j+1} + x_{i+1}x_{i}x_{j+1}x_{j} \right] \\
&= x_{i_1}\dots x_{i_{l-3}}\left[(x_{i}x_{i+1} + x_{i+1}x_{i})x_{j}x_{j+1} + (x_{i}x_{i+1}+x_{i+1}x_{i})x_{j+1}x_{j}) + x_{i+1}x_{j}x_{i}x_{j+1}\right] \in V^{\otimes a}R V^{\otimes 2}
\end{align*}\\

If $|j-(i+1)|=1$,

\begin{align*}
h(x) &= x_{i_1}\dots x_{i_{l-3}}\left[ x_{i}(x_{i+1}x_{j}+x_{j}x_{i+1})x_{j+1} + x_{i}x_{i+1}x_{j+1}x_{j} + x_{i+1}x_{i}(x_{j+1}x_{j}+x_{j}x_{j+1}) \right] \\
&= x_{i_1}\dots x_{i_{l-3}}\left[(x_{i}x_{i+1} + x_{i+1}x_{i})x_{j}x_{j+1} + (x_{i}x_{i+1}+x_{i+1}x_{i})x_{j+1}x_{j}) + x_{i}x_{j}x_{i+1}x_{j+1}\right] \in V^{\otimes a}R V^{\otimes 2}
\end{align*}\\

Let us show that $V^{\otimes l-2}R$ is sent to $V^{\otimes l-2}RV$ by looking at each generator and meanwhile we will check that $dh+hd=id$. $V^{ l-2}R$ has two types of generators: $x_{i_1}\dots x_{i_{l-1}}x_{i_l} \otimes x_j$ where $|i_l-i_{l-1}|\ne 1$ and  $x_{i_1}\dots x_{i_{l-2}}(x_ix_{i+1}+x_{i+1}x_i) \otimes x_j$.

For the first type of generators, let us denote $x_{i_1}\dots x_{i_{l-1}}x_{i_l} \otimes x_j$ by $x$. If $|j-i_l|\ne 1$, $h(x) = x_{i_1}\dots x_{i_l} x_j$ clearly belongs to $V^{ \ot l-2}RV$ and since $dx=0$ we have that $(dh+hd)(x)=x$.\\

If $j=i_l + 1$, $h(x)=0$  and $(dh+hd)(x)=x$.\\

If $j=i_l - 1$ and $|j-i_{l-1}|\ne 1$, $$h(x)= x_{i_1}\dots x_{i_{l-1}} (x_{i_{l}} x_{j}+x_jx_{i_{l}}) \in V^{ \ot l-2}RV.$$
$hd(x) = h(-x_{i_1}\dots x_{i_{l-1}} \otimes x_jx_{i_l}) = -x_{i_1}\dots x_{i_{l-1}} x_j\otimes x_{i_l}$ and thus

$(hd+dh)(x)=x$.\\

If $j=i_l - 1$, $|j-i_{l-1}|= 1$ and $|j-i_{l-2}|\ne 1$
$$h(x)= x_{i_1}\dots x_{i_{l-2}}(x_{i_{l-1}}x_j+x_jx_{i_{l-1}})x_{i_{l}} +  x_{i_1}\dots x_{i_{l-1}}x_{i_l} x_j\in V^{\ot l-2}RV.$$
$hd(x) = h(-x_{i_1}\dots x_{i_{l-1}}\otimes x_jx_{i_l}) = -x_{i_1}\dots x_{i_{l-2}} (x_{i_{l-1}} x_j+ x_j x_{i_{l-1}})\otimes x_{i_l}$
$\Rightarrow (hd+dh)(x)=x$.\\

If $j=i_l - 1$, $|j-i_{l-1}|= 1$ and $|j-i_{l-2}|= 1$, $h(x)=0 $
$$hd(x) = h(-x_{i_1}\dots x_{i_{l-1}}\otimes x_jx_{i_n}) = -(-x)=x$$
$\Rightarrow (hd+dh)(x)=x$.\\

It only remains to check generators of the second type, i.~e., elements $x:=x_{i_1}\dots x_{i_{l-2}}(x_ix_{i+1}+x_{i+1}x_i) \otimes x_j$.

If $j<i-1$ or $j>i+2$ it is clear that $h(x) \in V^{\ot l-2}RV$ and $(dh+hd)(x)=dh(x)=x$.\\

If $j=i-1$,
$$h(x) = x_{i_1}\dots x_{i_{l-2}}x_ix_{i+1}x_{i-1} + x_{i_1}\dots x_{i_{l-2}}x_{i+1}(x_ix_{i-1}+x_{i-1}x_i)\in V^{\ot l-2}RV$$
$hd(x)= h(x_{i_1}\dots x_{i_{l-2}}x_{i+1} \otimes x_{i}x_{i-1}) = -x_{i_1}\dots x_{i_{l-2}}x_{i+1}x_{i-1} \otimes x_{i+1}$

$\Rightarrow (dh+hd)(x)=x$.\\

If $j= i$,
$$h(x) = x_{i_1}\dots x_{i_{l-2}}x_i(x_{i+1}x_{i}+x_{i}x_{i+1})+ x_{i_1}\dots x_{i_{l-2}}x_{i+1}x_ix_{i} \in V^{\ot l-2}RV$$
$hd(x)= h(x_{i_1}\dots x_{i_{l-2}}x_{i} \otimes x_{i+1}x_{i})= -x_{i_1}\dots x_{i_{l-2}}x_{i}x_{i}  \otimes x_{i+1}$

$\Rightarrow (dh+hd)(x)=x$.\\

If $j= i+1$ and $|i-i_{l-2}|\ne 1$,
$$h(x)= x_{i_1}\dots x_{i_{l-2}}x_ix_{i+1}x_{i+1}-x_{i_1}\dots x_{i_{l-2}}x_ix_{i+1}x_{i+1}=0$$
$hd(x)=h(x_{i_1}\dots x_{i_{l-2}}x_{i+1}\otimes x_{i}x_{i+1})= x_{i_1}\dots x_{i_{l-2}}(x_{i+1}x_i+x_{i}x_{i+1})\otimes x_{i+1} = x$\\

If $j= i+1$ and $|i-i_{l-2}|= 1$,
$$h(x) = x_{i_1}\dots x_{i_{l-2}}x_{i}x_{i+1}x_{i+1} + x_{i_1}\dots x_{i_{l-2}}x_{i+1}(x_{i} x_{i+1}+ x_{i+1}x_{i})\in V^{\ot l-2}RV$$
$hd(x)= h(x_{i_1}\dots x_{i_{l-2}}x_{i+1}\otimes x_{i}x_{i+1})= -x_{i_1}\dots x_{i_{l-2}}x_{i+1}x_{i+1}\otimes x_{i}$.

$\Rightarrow (dh+hd)(x)=x$.\\

If $j= i+2$ and $|i_{l-2}-(i+1)|\ne 1$
$$h(x) = -x_{i_1}\dots x_{i_{l-2}}x_{i+1}x_{i}x_{i+2} + x_{i_1}\dots x_{i_{l-2}}x_{i+1}x_{i}x_{i+2} = 0$$
$hd(x) = h(x_{i_1}\dots x_{i_{l-2}}x_{i}\otimes x_{i+1}x_{i+2}) = x_{i_1}\dots x_{i_{l-2}}(x_ix_{i+1}+x_{i+1}x_i) \otimes x_{i+2}$

$\Rightarrow (dh+hd)(x)=x$.\\

If $j= i+2$ and $|i_{l-2}-(i+1)|= 1$
$$h(x) = x_{i_1}\dots x_{i_{l-2}}x_{i}(x_{i+1}x_{i+2}+x_{i+2}x_{i+1}) + x_{i_1}\dots x_{i_{l-2}}x_{i+1}x_{i}x_{i+2} \in V^{\ot l-2}RV$$
$hd(x)=h(x_{i_1}\dots x_{i_{l-2}}x_{i}\otimes x_{i+1}x_{i+2}) = - x_{i_1}\dots x_{i_{l-2}}x_{i}x_{i+2}\otimes x_{i+1}$

$\Rightarrow (dh+hd)(x)=x$.
\end{proof}

\bip

\section{Computations of the cohomology of deformation complexes}

In this section we compute the cohomology of several of the deformation complexes and show theorems \ref{thm:Fqiso} and \ref{thm:Fhbarqiso}.

\subsection{The proof of Theorem \ref{thm:Fqiso}} \label{app:defproof1}
Let us recall the definition of the graph complex $\hGCor_3$ from \cite[section 3.3]{Wi2}.
The elements of $\hGCor_3$ are $\K$-linear series in directed acyclic graphs with outgoing legs such that all vertices are at least bivalent, and such that there are no bivalent vertices with one incoming and one outgoing edge.\footnote{The last condition is again not present on \cite{Wi2}, but it does not change the cohomology.} We set to zero graphs containing vertices without outgoing edges. Here is an example graph:
$
 \Ba{c}\resizebox{6mm}{!} 
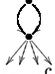
{\xy
 (0,0)*{\bu}="o",
(-5,-6)*{}="d1",
(-2,-6)*{}="d2",
(2,-6)*{}="d3",
(5,-6)*{}="d4",
\ar @{->} "o";"d1" <0pt>
\ar @{->} "o";"d2" <0pt>
\ar @{->} "o";"d3" <0pt>
\ar @{->} "o";"d4" <0pt>
   \ar@/^0.6pc/(0,8)*{\bullet};(0,0)*{\bullet}
   \ar@/^{-0.6pc}/(0,8)*{\bullet};(0,0)*{\bullet}
 \endxy}
 \Ea
%
$.
The degrees are computed just as for graphs occurring in $\GCor_3$, with the external legs considered to be of degree 0. For the description of the differential we refer the reader to \cite[section 3.3]{Wi2}.

There is a map $\Psi: \GCor_3\to \hGCor_3$ sending a graph $\Gamma$ to the linear combination
\begin{equation}\label{equ:hairymap}
\Gamma \mapsto
\sum_{j=1}^\infty
\underbrace{
  \Ba{c}\resizebox{9mm}{!}  {\xy
(0,-5.2)*+{...},
(0,0)*+{\Ga}="o",
(-3,-7)*{}="5",
(3,-7)*{}="6",
(5,-7)*{}="7",
(-5,-7)*{}="8",
\ar @{->} "o";"5" <0pt>
\ar @{->} "o";"6" <0pt>
\ar @{->} "o";"7" <0pt>
\ar @{->} "o";"8" <0pt>
\endxy}\Ea
%
 }_{j\times}
\end{equation}
where the picture on the right means that one should sum over all ways of connecting $j$ outgoing edges to the graph $\Gamma$. Graphs for which there remain vertices with no outgoing edge are identified with $0$.

The following proposition has been shown in loc. cit.
\begin{proposition}[Proposition 3 of \cite{Wi2}]\label{prop:GChGC}
 The map $\Psi: \GCor_3 \to \hGCor_3$ is a quasi-isomorphism up to the class in $H(\hGCor_n)$ represented by the graph cocycle
\begin{equation}
\label{equ:singleclass}
 \sum_{j\geq 2}
 (j-1)
 \underbrace{
  \Ba{c}\resizebox{9mm}{!}  {\xy
(0,-5.2)*+{...},
(0,0)*{\bu}="o",
(-3,-7)*{}="5",
(3,-7)*{}="6",
(5,-7)*{}="7",
(-5,-7)*{}="8",
\ar @{->} "o";"5" <0pt>
\ar @{->} "o";"6" <0pt>
\ar @{->} "o";"7" <0pt>
\ar @{->} "o";"8" <0pt>
\endxy}\Ea
 %
 %
 }_{j\times}.
\end{equation}
\end{proposition}

There is a map $G:\hGCor_3 \to \Der(\LieBi_\infty)$ sending a graph $\Gamma$ in $\GCor_3$ to the series
\Beq 
G( \Gamma )=
 \sum
    \overbrace{
 \underbrace{ \Ba{c}\resizebox{9mm}{!}  {\xy
(0,4.9)*+{...},
(0,-4.9)*+{...},
(0,0)*+{\Ga}="o",
(-5,6)*{}="1",
(-3,6)*{}="2",
(3,6)*{}="3",
(5,6)*{}="4",
(-3,-6)*{}="5",
(3,-6)*{}="6",
(5,-6)*{}="7",
(-5,-6)*{}="8",
\ar @{<-} "o";"1" <0pt>
\ar @{<-} "o";"2" <0pt>
\ar @{<-} "o";"3" <0pt>
\ar @{<-} "o";"4" <0pt>
\ar @{->} "o";"5" <0pt>
\ar @{->} "o";"6" <0pt>
\ar @{->} "o";"7" <0pt>
\ar @{->} "o";"8" <0pt>
\endxy}\Ea
 }_{n\times}
 }^{m\times}
\Eeq

The map $F : \GC_3^{or}\to \Der(\LieBi_\infty)$ from Theorem \ref{thm:Fqiso} factors through the map $G$ above, i.~e., it can be written as the composition
\[
 \GC_3^{or}\stackrel{\Psi}{\to} \hGCor_3 \stackrel{G}{\to} \Der(\LieBi_\infty).
\]
In view of Proposition \ref{prop:GChGC} Theorem \ref{thm:Fqiso} hence follows immediately from the following result.

\begin{proposition}\label{prop:Gqiso}
 The map $G:\hGCor_3 \to \Der(\LieBi_\infty)$ is a quasi-isomorphism.
\end{proposition}
\begin{proof}
 For a graph in $\Der(\LieBi_\infty)$ we will call its \emph{skeleton} the graph obtained in the following way:
 \begin{enumerate}
  \item Remove all input legs and recursively remove all valence 1 vertices created.
  \item Remove valence 2 vertices with one incoming and one outgoing edge and connect the two edges.
 \end{enumerate}
An example of a graph and its skeleton the following
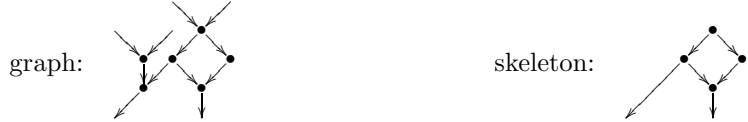
\begin{align*}
 \text{graph: }&
 \Ba{c}
\resizebox{17mm}{!}
{ \xy
(0,20)*{\bu}="u",
(-5,15)*{\bu}="L",
 (5,15)*{\bu}="R",
(0,10)*{\bu}="d",
(-5,25)*{}="u1",
(5,25)*{}="u2",
(0,5)*{}="d1",
(-10,10)*{\bu}="b",
(-10,15)*{\bu}="a",
(-15,20)*{}="a1",
(-5,20)*{}="a2",
(-15,5)*{}="b1",
\ar @{->} "u1";"u" <0pt>
\ar @{->} "u2";"u" <0pt>
\ar @{->} "u";"L" <0pt>
\ar @{->} "u";"R" <0pt>
\ar @{->} "L";"d" <0pt>
\ar @{->} "R";"d" <0pt>
\ar @{->} "d";"d1" <0pt>
\ar @{->} "L";"b" <0pt>
\ar @{->} "a1";"a" <0pt>
\ar @{->} "a2";"a" <0pt>
\ar @{->} "b";"b1" <0pt>
\ar @{->} "a";"b" <0pt>
\endxy}
\Ea
%
 &
  \text{skeleton: }&
   \Ba{c}
\resizebox{17mm}{!}
{ \xy
(0,20)*{\bu}="u",
(-5,15)*{\bu}="L",
 (5,15)*{\bu}="R",
(0,10)*{\bu}="d",
(-5,25)*{}="u1",
(5,25)*{}="u2",
(0,5)*{}="d1",
(-15,20)*{}="a1",
(-5,20)*{}="a2",
(-15,5)*{}="b1",
\ar @{->} "u";"L" <0pt>
\ar @{->} "u";"R" <0pt>
\ar @{->} "L";"d" <0pt>
\ar @{->} "R";"d" <0pt>
\ar @{->} "d";"d1" <0pt>
\ar @{->} "L";"b1" <0pt>
\endxy}
\Ea
\end{align*}
We put a filtration on $\Der(\LieBi_\infty)$ by the total number of vertices in the skeleton. Let $\gr\Der(\LieBi_\infty)$ be the associated graded.
Note that for elements in the image of some graph $\Gamma\in \hGCor_3$ under $G$ the skeleton is just the graph $\Gamma$, and hence there is a map of complexes $\hGCor_3 \to \gr\Der(\LieBi_\infty)$, where we consider the left hand side with zero differential.
We claim that the induced map $\hGCor_3 \to H(\gr\Der(\LieBi_\infty))$ is an isomorphism. From this claim the Proposition follows immediately by a standard spectral sequence argument.

The differential on $\gr\Der(\LieBi_\infty)$ does not change the skeleton. Hence the complex $\gr\Der(\LieBi_\infty)$ splits into a direct product of complexes, say $\tilde C_\gamma$, one for each skeleton $\gamma$
\[
 \gr\Der(\LieBi_\infty) = \prod_\gamma \tilde C_\gamma .
\]
Furthermore, each skeleton represents an automorphism class of graphs, and we may write
\[
 \tilde C_\gamma = C_\gamma^{\Aut_\gamma}
\]
where $C_{\tilde \gamma}$ is an appropriately defined complex for one representative $\tilde \gamma$ of the isomorphism class $\gamma$ and $\Aut_\gamma$ is the automorphism group associated to the skeleton. In other words, the $\tilde \gamma$ now has distinguishable vertices and edges.
More conretely, the complex $C_{\tilde \gamma}$ is the complex of $\K$-linear series of graphs obtained from $\tilde \gamma$ by
\begin{enumerate}
 \item Adding some bivalent vertices with one input and one output on edges. We call these vertices ``edge vertices''.
 \item Attaching input forests at the vertices, such that all vertices are at least trivalent and have at least one input and one output.
 We call the forest attached to a vertex the forest of that vertex.
\end{enumerate}
An example is the following:
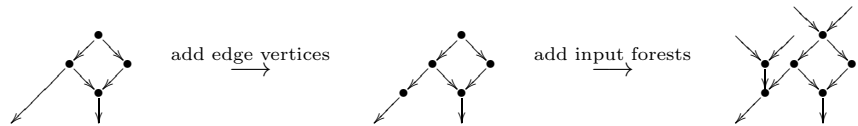
\[
  \Ba{c}
\resizebox{17mm}{!}
{ \xy
(0,20)*{\bu}="u",
(-5,15)*{\bu}="L",
 (5,15)*{\bu}="R",
(0,10)*{\bu}="d",
(-5,25)*{}="u1",
(5,25)*{}="u2",
(0,5)*{}="d1",
(-15,20)*{}="a1",
(-5,20)*{}="a2",
(-15,5)*{}="b1",
\ar @{->} "u";"L" <0pt>
\ar @{->} "u";"R" <0pt>
\ar @{->} "L";"d" <0pt>
\ar @{->} "R";"d" <0pt>
\ar @{->} "d";"d1" <0pt>
\ar @{->} "L";"b1" <0pt>
\endxy}
\Ea
\ \ \stackrel{\text{add edge vertices}}{\longrightarrow} \ \
  \Ba{c}
\resizebox{17mm}{!}
{ \xy
(0,20)*{\bu}="u",
(-5,15)*{\bu}="L",
 (5,15)*{\bu}="R",
(0,10)*{\bu}="d",
(-5,25)*{}="u1",
(5,25)*{}="u2",
(0,5)*{}="d1",
(-10,10)*{\bu}="b",
(-15,20)*{}="a1",
(-5,20)*{}="a2",
(-15,5)*{}="b1",
\ar @{->} "u";"L" <0pt>
\ar @{->} "u";"R" <0pt>
\ar @{->} "L";"d" <0pt>
\ar @{->} "R";"d" <0pt>
\ar @{->} "d";"d1" <0pt>
\ar @{->} "L";"b" <0pt>
\ar @{->} "b";"b1" <0pt>
\endxy}
\Ea
%
 %
 \ \ \stackrel{\text{add input forests}}{\longrightarrow}\ \
%
%
 \Ba{c}
\resizebox{17mm}{!}
{ \xy
(0,20)*{\bu}="u",
(-5,15)*{\bu}="L",
 (5,15)*{\bu}="R",
(0,10)*{\bu}="d",
(-5,25)*{}="u1",
(5,25)*{}="u2",
(0,5)*{}="d1",
(-10,10)*{\bu}="b",
(-10,15)*{\bu}="a",
(-15,20)*{}="a1",
(-5,20)*{}="a2",
(-15,5)*{}="b1",
\ar @{->} "u1";"u" <0pt>
\ar @{->} "u2";"u" <0pt>
\ar @{->} "u";"L" <0pt>
\ar @{->} "u";"R" <0pt>
\ar @{->} "L";"d" <0pt>
\ar @{->} "R";"d" <0pt>
\ar @{->} "d";"d1" <0pt>
\ar @{->} "L";"b" <0pt>
\ar @{->} "a1";"a" <0pt>
\ar @{->} "a2";"a" <0pt>
\ar @{->} "b";"b1" <0pt>
\ar @{->} "a";"b" <0pt>
\endxy}
\Ea
\]
We next put another filtration on $C_{\tilde \gamma}$ by the number of edge vertices added in the first step above and consider the associated graded $\gr C_{\tilde \gamma}$.
Note that the differential $\gr C_{\tilde \gamma}$ acts on each of the forests attached to the vertices separately and hence the complex splits into a (completed) tensor product of complexes, one for each such vertex. Let us call the complex made from the possible forests at the vertex $v$ the forest complex at that vertex.
By the same argument showing that the cohomology of a free $\Lie_\infty$ algebra generated by a single generator is two dimensional, we find that the
forest complex at $v$ has either one or two dimensional cohomology. If vertex $v$ has no incoming edge that it is one dimensional, the class being represented by the forests
\[
\sum_{j\geq 1}
\overbrace{
 \Ba{c}\resizebox{9mm}{!}  {\xy
(0,4.5)*+{...},
(0,-2)*{_v},
(0,0)*{\bu}="o",
(-5,5)*{}="1",
(-3,5)*{}="2",
(3,5)*{}="3",
(5,5)*{}="4",
\ar @{<-} "o";"1" <0pt>
\ar @{<-} "o";"2" <0pt>
\ar @{<-} "o";"3" <0pt>
\ar @{<-} "o";"4" <0pt>
\endxy}\Ea
 }^{j \times} \,.
\]
If the vertex $v$ already has an incoming edge, then there is one additional class obtained by not adding any input forest.

Hence we find that $H(\gr C_{\tilde \gamma})$ is spanned by graphs obtained from $\tilde \gamma$ as follows:
\begin{enumerate}
 \item Add some bivalent vertices with one input and one output on edges.
 \item For each vertex that is either not at least trivalent or does not have an incoming edge, sum over all ways of attaching incoming legs at that vertex.
 \item For at least trivalent vertices with an incoming edge, there is a choice of either not adding anything at that vertex, or summing over all ways of attaching incoming legs at that vertex. Let us call the vertices for which the first choice is made bald vertices and the others hairy.
\end{enumerate}

Let us look at the next page in the spectral sequence associated to our filtration on $C_{\tilde \gamma}$.
The differential creates one edge vertex by either splitting an existing edge vertex or by splitting a skeleton vertex. Again, the complex splits into a product of complexes, one for each edge of $\tilde \gamma$. For each such edge we have to consider 3 cases separately:
\begin{enumerate}
 \item Both endpoints in $\tilde \gamma$ are hairy.
 \item Both endpoints in $\tilde \gamma$ are bald.
 \item One endpoints is hairy and one is bald.
\end{enumerate}

We leave it to the reader to check that:
\begin{enumerate}
 \item In the first case the cohomology is one-dimensional, represented by a single edge without edge vertices.
 \item In the third case the cohomology vanishes.
\end{enumerate}

Since there is necessarily at least one hairy vertex in the graph, the second assertion implies that if there is a bald vertex as well, the resulting complex is acyclic.
Hence all vertices must be hairy. By the first assertion the cohomology is one-dimensional for each skeleton.
One easily checks that the representative is exactly the image of the skeleton considered as element in $\hGCor_3$. Hence the proposition follows.
\end{proof}

\begin{remark}\label{rem:alternative Fqiso proof}
 There is also an alternative way to compute the cohomology of the deformation complex $\Der(\LieBi_\infty)$.
 Namely, by Koszulness of $\LieBi$ this complex is quasi-isomorphic to $\Def(\LieBi_\infty\to \LieBi)[1]$.
 It is well known that the prop governing Lie bialgebras $\LieBiP$ may be written as
 \[
\LieBiP(n,m) \cong \bigoplus_N \caL ieP(n,N) \otimes_{\bS_N} co\caL ieP(N,m)
 \]
using the props governing Lie algebras and Lie coalgebras. Interpreting elements of the above prop as linear combinations directed acyclic graphs, the sub-properad $\LieBi$ may be obtained as that formed by the connected such graphs.
 It is hence an easy exercise to check that $\Def(\LieBi_\infty\to \LieBi)[1]$ is identical to the complex $\Def(\hoe_2 \to e_2)_{\rm conn}[1]$ from \cite{Wi1}, up to unimportant completion issues.
 The cohomology of the latter complex has been computed in loc. cit. to be
 \[
  H(\GC_2) \oplus \bigoplus_{j=1,5,\dots} \K[2-j] \oplus \K \, .
 \]
Using the main result of \cite{Wi2} this agrees with the cohomology as computed by Theorem \ref{thm:Fqiso}. Conversely, the above proof of Theorem \ref{thm:Fqiso} together with this remark yields an alternative proof of the main result of \cite{Wi2}.
\end{remark}

\bip

\subsection{The proof of Theorem \ref{thm:Fhbarqiso}}\label{app:defproof2}
Let us next consider Theorem \ref{thm:Fhbarqiso}, whose proof will be a close analog of that of Theorem \ref{thm:Fqiso} in the previous subsection.
There is a natural differential graded Lie algebra structure on $\hGCor_3$ such that the map $\Psi: \GCor_3\to \hGCor_3$ from the previous section is a map of Lie algebras. The map $\Psi$ extends $\hbar$-linearly to a map of graded Lie algebras $\Psi_\hbar: \GCor_3[[\hbar]]\to \hGCor_3[[\hbar]]$.
The Maurer-Cartan element $\Phi_\hbar\in \GCor_3[[\hbar]]$ from Proposition \ref{prop:phihbar} is sent to a Maurer-Cartan element $\hat \Phi_\hbar:=\Psi_\hbar(\Phi_\hbar) \in \hGCor_3[[\hbar]]$. We endow $\hGCor_3[[\hbar]]$ with the differential
\[
 d_\hbar \Gamma = [\hat \Phi_\hbar, \, ].
\]
In particular it follows that we have a map of differential graded Lie algebras
\[
 \Psi_\hbar: (\GCor_3[[\hbar]], d_\hbar)\to (\hGCor_3[[\hbar]], d_\hbar).
\]

The map $F_\hbar$ from Theorem \ref{thm:Fhbarqiso} factors through $\hGCor_3[[\hbar]]$:
\[
 \GC_3^{or}[[\hbar]]\stackrel{\Psi_\hbar}{\longrightarrow} \hGCor_3[[\hbar]] \stackrel{G_\hbar}{\longrightarrow} \Der(\LoB_\infty).
\]
The second map $G_\hbar: \hGCor_3[[\hbar]]\to \Der(\LoB_\infty)$ sends $\hbar^N\Gamma$, for $\Gamma\in \hGCor_3$ to
\[
  \sum_{j\geq 1}\sum_{k=0}^N \hbar^{N-k}
 \xy
(0,5.5)*+{...},
%
(0,0)*+{_{\Ga_k}}*\cir{}="o",
(-5,7)*{}="1",
(-3,7)*{}="2",
(3,7)*{}="3",
(5,7)*{}="4",
(-3,-5)*{}="5",
(3,-5)*{}="6",
(5,-5)*{}="7",
(-5,-5)*{}="8",
\ar @{<-} "o";"1" <0pt>
\ar @{<-} "o";"2" <0pt>
\ar @{<-} "o";"3" <0pt>
\ar @{<-} "o";"4" <0pt>
\endxy
\]
where we  again sum over all ways of attaching the incoming legs, setting to zero graphs with vertices without incoming edges. Furthermore, $\Gamma_k$ is the linear combination of graphs obtained by summing over all ways of assigning weights to the vertices of $\Gamma$ such that the total weight is $k$. 
We have the following two results, from which Theorem \ref{thm:Fhbarqiso} immediately follows.

\begin{proposition}
 The map $\Psi_\hbar: (\GC_3^{or}[[\hbar]],d_\hbar) \to (\hGCor_3[[\hbar]],d_\hbar)$ is a quasi-isomorphism up to the classes $T\K[[\hbar]]\subset \Der(\LieBi_\infty)$ where
 \[
  T=
  \sum_{n,p}
 (n+2p-2)
 \underbrace{
 \xy
(0,-4.5)*+{...},
(0,0)*+{_p}*\cir{}="o",
(-5,5)*{}="1",
(-3,5)*{}="2",
(3,5)*{}="3",
(5,5)*{}="4",
(-3,-5)*{}="5",
(3,-5)*{}="6",
(5,-5)*{}="7",
(-5,-5)*{}="8",
%
\ar @{->} "o";"5" <0pt>
\ar @{->} "o";"6" <0pt>
\ar @{->} "o";"7" <0pt>
\ar @{->} "o";"8" <0pt>
\endxy
%
 }_{n\times}.
 \]
\end{proposition}
\begin{proof}[Proof sketch]
 Take filtrations on $\GCor_3[[\hbar]]$ and $\hGCor_3[[\hbar]]$ by the power of $\hbar$.
 The differential on the associated graded spaces is te $\hbar$-linear extension of the differentials on $\GCor_3$ and $\hGCor_3$. Hence by Proposition \ref{prop:GChGC} the map $\Psi_\hbar$ is a quasi-isomorphism on the level of the associated graded spaces, up to the classes above. The result follows by a standard spectral sequence argument, noting that the above element $T$ is indeed $d_\hbar$-closed.
\end{proof}

\begin{proposition}\label{prop:Ghbarqiso}
 The map $G_\hbar: \hGCor_3[[\hbar]]\to \Der(\LoB_\infty)$ is a quasi-isomorphism.
\end{proposition}
\begin{proof}[Proof sketch]
 Take filtrations on $\hGCor_3[[\hbar]]$ and
 \[
  \Der(\LoB_\infty)\cong \prod_{n,m\geq 1} (\LoB_\infty(n,m)\otimes \sgn_n\otimes \sgn_m )^{\bS_n\times \bS_m} [[\hbar]][-n-m+1]
 \]
by genus and by powers of $\hbar$.
Then we claim that the induced map on the associated graded complexes $\gr G_\hbar: \gr\hGCor_3[[\hbar]]\to \gr\Der(\LoB_\infty)$ is a quasi-isomorphism, thus showing the proposition by a standard spectral sequence argument.

To show the claim, we proceed analogously to the proof of Proposition \ref{prop:Gqiso}. Let us go through the proof again and highlight only the differences.
The skeleton of a graph is defined as before, except that one also forgets the weights of all vertices.
The complex $\gr \Der(\LoB_\infty)$ splits into a product of subcomplexes that we again call $\tilde C_\gamma$, one for each skeleton $\gamma$. Again
\[
 \tilde C_\gamma =  C_{\tilde \gamma}^{\Aut_\gamma}
\]
for some representative $\tilde \gamma$ of the isomorphism class $\gamma$. Hence it again suffices to compute the cohomology of $C_{\tilde \gamma}$.
Graphs contributing are again obtained by adding edge vertices and input forests, except that now all vertices are also assigned an arbitrary weight.
Again we take a filtration on the number of edge vertices, which leaves us with the task of computing
the cohomology of a complex associated to one forest attached to a vertex $v$. We find that representatives of cohomology classes are either:
\begin{itemize}
 \item Vertex $v$ with any weight and no attached forest. Let us call such a $v$ again bald.
 \item Vertex $v$ with weight $0$ and input legs attached in all possible ways, let us call such a $v$ again hairy.
\end{itemize}

The differential on the second page of the spectral sequence again adds one edge vertex, which however can have a non-zero weight now, and if it has a non-zero weight it may be bald.
We may introduce another filtration by the number of non-hairy edge vertices. The differential on the associated graded creates one hairy edge vertex.
The resulting complex is a tensor product of complexes, one for each edge. The complexes associated to each edge again can have three different types: (i) both endpoints in the skeleton are hairy, (ii) both are bald or (iii) one is hairy, one is bald. Again one checks that in case (iii) the complex is acyclic and in case (i) one-dimensional, the cohomology class represented by a single edge.
Hence, since at least one vertex must be hairy, all vertices must be. Hence we recover at this stage the image of $\hGCor_3[[\hbar]]$ and are done.
\end{proof}

%

\bip

\section{ Computation of the cohomology of the operad ${\cB\cV}_\infty^{com}$}\label{app:BV_com proof}

\bip
 \subsection{\bf An equivalent definition of the operad $\cB\cV$}\label{5: subsect on L-diamond and BV}
Let $\caL^\diamond$ be an operad generated by two degree $-1$ corollas,
 $
\Ba{c}\resizebox{3.6mm}{!}{  \xy
(0,4)*{};
(0,0)*+{_1}*\cir{}
**\dir{-};
(0,-4)*{};
(0,0)*+{_1}*\cir{}
**\dir{-};
\endxy}\Ea$ and
$\Ba{c}\resizebox{7mm}{!}{
\xy
(-4,-4)*{};
(0,0)*+{_0}*\cir{}
**\dir{-};
(4,-4)*{};
(0,0)*+{_0}*\cir{}
**\dir{-};
(4,-6)*{_2};
(-4,-6)*{_1};
(0,5)*{};
(0,0)*+{_0}*\cir{}
**\dir{-};
\endxy}\Ea=
\Ba{c}\resizebox{7mm}{!}{ \xy
(-4,-4)*{};
(0,0)*+{_0}*\cir{}
**\dir{-};
(4,-4)*{};
(0,0)*+{_0}*\cir{}
**\dir{-};
(4,-6)*{_1};
(-4,-6)*{_2};
(0,5)*{};
(0,0)*+{_0}*\cir{}
**\dir{-};
\endxy}\Ea,
$
subject to the following relations,
$$
\Ba{c}
\resizebox{3.6mm}{!}{ \xy
(0,0)*+{_1}*\cir{}="b",
(0,6)*+{_1}*\cir{}="c",
%
(0,-4)*{}="-1",
(0,10)*{}="1'",
\ar @{-} "b";"c" <0pt>
\ar @{-} "b";"-1" <0pt>
\ar @{-} "c";"1'" <0pt>
\endxy}
\Ea\hspace{-1mm} = 0\ \  ,\ \
%
%
\Ba{c}
\resizebox{6mm}{!}{ \xy
(0,0)*+{_0}*\cir{}="b",
(0,7)*+{_1}*\cir{}="c",
%
(-4,-5)*{}="-1",
(-2,-5)*{}="-2",
(4,-5)*{}="-3",
(0,12)*{}="1'",
%
\ar @{-} "b";"c" <0pt>
\ar @{-} "b";"-1" <0pt>
\ar @{-} "b";"-3" <0pt>
\ar @{-} "c";"1'" <0pt>
\endxy}
\Ea
+
\Ba{c}
\resizebox{8mm}{!}{ \xy
(0,7)*+{_0}*\cir{}="b",
(-4,0)*+{_1}*\cir{}="c",
(-4,-5)*{}="-1",
(4,-5)*{}="-2",
(4,1)*{}="-3",
(0,12)*{}="1'",
\ar @{-} "b";"c" <0pt>
\ar @{-} "c";"-1" <0pt>
\ar @{-} "b";"-3" <0pt>
\ar @{-} "b";"1'" <0pt>
\endxy}
\Ea
+
\Ba{c}
\resizebox{8mm}{!}{ \xy
(0,7)*+{_0}*\cir{}="b",
(4,0)*+{_1}*\cir{}="c",
(4,-5)*{}="-1",
(-4,-5)*{}="-2",
(-4,1)*{}="-3",
(0,12)*{}="1'",
\ar @{-} "b";"c" <0pt>
\ar @{-} "c";"-1" <0pt>
\ar @{-} "b";"-3" <0pt>
\ar @{-} "b";"1'" <0pt>
\endxy}
\Ea
=0, \ \ \
\Ba{c}
\resizebox{10.5mm}{!}{ \xy
(-8,-7.5)*{_{_1}};
(-0,-7.5)*{_{_2}};
(4.5,-0.6)*{_{_3}};
(0,7)*+{_0}*\frm{o}="b",
(-4,0)*+{_0}*\frm{o}="c",
%
(-8,-6)*{}="-1",
(0,-6)*{}="-2",
(4,1)*{}="-3",
(0,12)*{}="1'",
%
\ar @{-} "b";"c" <0pt>
\ar @{-} "c";"-1" <0pt>
\ar @{-} "c";"-2" <0pt>
\ar @{-} "b";"-3" <0pt>
\ar @{-} "b";"1'" <0pt>
\endxy}
\Ea
+
\Ba{c}
\resizebox{10.5mm}{!}{ \xy
(-8,-7.5)*{_{_2}};
(-0,-7.5)*{_{_3}};
(4.5,-0.6)*{_{_1}};
(0,7)*+{_0}*\frm{o}="b",
(-4,0)*+{_0}*\frm{o}="c",
%
(-8,-6)*{}="-1",
(0,-6)*{}="-2",
(4,1)*{}="-3",
(0,12)*{}="1'",
%
\ar @{-} "b";"c" <0pt>
\ar @{-} "c";"-1" <0pt>
\ar @{-} "c";"-2" <0pt>
\ar @{-} "b";"-3" <0pt>
\ar @{-} "b";"1'" <0pt>
\endxy}
\Ea
+
\Ba{c}\resizebox{10.5mm}{!}{ \xy
(-8,-7.5)*{_{_3}};
(-0,-7.5)*{_{_1}};
(4.5,-0.6)*{_{_2}};
(0,7)*+{_0}*\frm{o}="b",
(-4,0)*+{_0}*\frm{o}="c",
%
(-8,-6)*{}="-1",
(0,-6)*{}="-2",
(4,1)*{}="-3",
(0,12)*{}="1'",
%
\ar @{-} "b";"c" <0pt>
\ar @{-} "c";"-1" <0pt>
\ar @{-} "c";"-2" <0pt>
\ar @{-} "b";"-3" <0pt>
\ar @{-} "b";"1'" <0pt>
\endxy}
\Ea
=0
$$

Let $\cC om$ be the operad of commutative algebras with the generator controlling the graded commutative multiplication denoted by
$\begin{xy}
 <0mm,0.66mm>*{};<0mm,3mm>*{}**@{-},
 <0.39mm,-0.39mm>*{};<2.2mm,-2.2mm>*{}**@{-},
 <-0.35mm,-0.35mm>*{};<-2.2mm,-2.2mm>*{}**@{-},
 <0mm,0mm>*{\circ};<0mm,0mm>*{}**@{},
\end{xy}$. Define an operad, $\cB\cV$, of Batalin-Vilkovisky algebras as the free operad
generated by operads $\caL^\diamond$ and $\cC om$ modulo the following relations,
\Beq\label{5: BV operad relations 2}
\Ba{c}\resizebox{7mm}{!}{ \xy
(-4,-4)*{};
(0,0)*+{_0}*\cir{}
**\dir{-};
(4,-4)*{};
(0,0)*+{_0}*\cir{}
**\dir{-};
(4,-6)*{_2};
(-4,-6)*{_1};
(0,5)*{};
(0,0)*+{_0}*\cir{}
**\dir{-};
\endxy}\Ea
=
\Ba{c}
\resizebox{6mm}{!}{ \xy
(0,0)*{\circ}="b",
(0,6)*+{_1}*\cir{}="c",
%
(-4,-5)*{}="-1",
(4,-5)*{}="-2",
(0,12)*{}="1'",
%
\ar @{-} "b";"c" <0pt>
\ar @{-} "b";"-1" <0pt>
\ar @{-} "b";"-2" <0pt>
\ar @{-} "c";"1'" <0pt>
\endxy}
\Ea
-
\Ba{c}
\resizebox{8mm}{!}{ \xy
(0,6)*{\circ}="b",
(-4,0)*+{_1}*\cir{}="c",
%
(-4,-5)*{}="-1",
(4,-5)*{}="-2",
(4,1)*{}="-3",
(0,12)*{}="1'",
%
\ar @{-} "b";"c" <0pt>
\ar @{-} "c";"-1" <0pt>
\ar @{-} "b";"-3" <0pt>
\ar @{-} "b";"1'" <0pt>
\endxy}
\Ea
-
\Ba{c}
\resizebox{8mm}{!}{ \xy
(0,6)*{\circ}="b",
(4,0)*+{_1}*\cir{}="c",
%
(4,-5)*{}="-1",
(-4,-5)*{}="-2",
(-4,1)*{}="-3",
(0,12)*{}="1'",
%
\ar @{-} "b";"c" <0pt>
\ar @{-} "c";"-1" <0pt>
\ar @{-} "b";"-3" <0pt>
\ar @{-} "b";"1'" <0pt>
\endxy}
\Ea
\ \ \ ,\ \ \
\Ba{c}
\resizebox{10mm}{!}{ \xy
(-8,-7.5)*{_{_1}};
(-0,-7.5)*{_{_2}};
(4.5,-0.6)*{_{_3}};
(0,7)*+{_0}*\frm{o}="b",
(-4,0)*{\circ}="c",
%
(-8,-6)*{}="-1",
(0,-6)*{}="-2",
(4,1)*{}="-3",
(0,12)*{}="1'",
%
\ar @{-} "b";"c" <0pt>
\ar @{-} "c";"-1" <0pt>
\ar @{-} "c";"-2" <0pt>
\ar @{-} "b";"-3" <0pt>
\ar @{-} "b";"1'" <0pt>
\endxy}
\Ea
-
\Ba{c}
\resizebox{10mm}{!}{ \xy
(-8,-7.5)*{_{_2}};
(-0,-7.5)*{_{_3}};
(4.5,-0.6)*{_{_1}};
(0,7)*{\circ}="b",
(-4,0)*+{_0}*\frm{o}="c",
%
(-8,-6)*{}="-1",
(0,-6)*{}="-2",
(4,1)*{}="-3",
(0,12)*{}="1'",
%
\ar @{-} "b";"c" <0pt>
\ar @{-} "c";"-1" <0pt>
\ar @{-} "c";"-2" <0pt>
\ar @{-} "b";"-3" <0pt>
\ar @{-} "b";"1'" <0pt>
\endxy}
\Ea
-
\Ba{c}
\resizebox{10mm}{!}{ \xy
(-8,-7.5)*{_{_1}};
(-0,-7.5)*{_{_3}};
(4.5,-0.6)*{_{_2}};
(0,7)*{\circ}="b",
(-4,0)*+{_0}*\frm{o}="c",
%
(-8,-6)*{}="-1",
(0,-6)*{}="-2",
(4,1)*{}="-3",
(0,12)*{}="1'",
%
\ar @{-} "b";"c" <0pt>
\ar @{-} "c";"-1" <0pt>
\ar @{-} "c";"-2" <0pt>
\ar @{-} "b";"-3" <0pt>
\ar @{-} "b";"1'" <0pt>
\endxy}
\Ea
=0.
\Eeq
In fact, the second relation in (\ref{5: BV operad relations 2}) follows from the previous ones.
We keep it the list in order to define, following \cite{GTV}, an operad $q\cB\cV$ as an operad freely generated by $\caL^\diamond$ and $\cC om$ modulo a version of relations (\ref{5: BV operad relations 2}) in which the first relation is replaced by the following one,
$$
\Ba{c}
\resizebox{6mm}{!}{ \xy
(0,0)*{\circ}="b",
(0,6)*+{_1}*\cir{}="c",
%
(-4,-5)*{}="-1",
(4,-5)*{}="-2",
(0,12)*{}="1'",
%
\ar @{-} "b";"c" <0pt>
\ar @{-} "b";"-1" <0pt>
\ar @{-} "b";"-2" <0pt>
\ar @{-} "c";"1'" <0pt>
\endxy}
\Ea
-
\Ba{c}
\resizebox{8mm}{!}{ \xy
(0,6)*{\circ}="b",
(-4,0)*+{_1}*\cir{}="c",
%
(-4,-5)*{}="-1",
(4,-5)*{}="-2",
(4,1)*{}="-3",
(0,12)*{}="1'",
%
\ar @{-} "b";"c" <0pt>
\ar @{-} "c";"-1" <0pt>
\ar @{-} "b";"-3" <0pt>
\ar @{-} "b";"1'" <0pt>
\endxy}
\Ea
-
\Ba{c}
\resizebox{8mm}{!}{ \xy
(0,6)*{\circ}="b",
(4,0)*+{_1}*\cir{}="c",
%
(4,-5)*{}="-1",
(-4,-5)*{}="-2",
(-4,1)*{}="-3",
(0,12)*{}="1'",
%
\ar @{-} "b";"c" <0pt>
\ar @{-} "c";"-1" <0pt>
\ar @{-} "b";"-3" <0pt>
\ar @{-} "b";"1'" <0pt>
\endxy}
\Ea=0.
$$
Being a quotient of a free operad, the operad $\cB\cV$ inherits an increasing filtration by the number of vertices in the trees. It is clear that there is a morphism,
$$
g: q\cB\cV\lon gr(\cB\cV),
$$
from $q\cB\cV$ into the associated graded operad.

\subsection{\bf Proposition \cite{GTV}}\label{5: Propos on qBV and BV} {\em The morphism $g:q\cB\cV\lon gr(\cB\cV)$ is an isomorphism. }

\subsection{\bf Remark}\label{5: remark on qBV} The relations in the operad $q\cB\cV$ are homogeneous. It is easy to see that, as an $\bS$-module, $q\cB\cV$ is isomorphic to  $\cC om \circ \caL^\diamond$,
the vector space spanned by graphs from $\cC om$ whose legs are decorated with elements from $\caL^\diamond$.

\subsection{An auxiliary dg operad}

\mip

For any natural number $a\geq 1$ define by induction (over the number, $k=1,2,\ldots, a+1$, of input legs) a collection of $a+1$ elements,
\Beq\label{5: square vertices}
\Ba{c}
{\resizebox{4.0mm}{!}{ \xy
(0,5)*{};
(0,0)*+{_a}*\frm{-}
**\dir{-};
(0,-5)*{};
(0,0)*+{_a}*\frm{-}
**\dir{-};
\endxy}}\Ea :=\Ba{c}
{\resizebox{4.0mm}{!}{ \xy
(0,5)*{};
(0,0)*+{_a}*\cir{}
**\dir{-};
(0,-5)*{};
(0,0)*+{_a}*\cir{}
**\dir{-};
\endxy}}\Ea\ \ , \ \ \ldots \ \ ,\
\Ba{c}\resizebox{14mm}{!}{ \xy
(-7.5,-8.6)*{_{_1}};
(-4.1,-8.6)*{_{_2}};
(4.5,-8.6)*{_{_{k\hspace{-0.3mm}-\hspace{-0.3mm}1}}};
(9.0,-8.5)*{_{_{k}}};
(0.0,-6)*{...};
(0,5)*{};
(0,0)*+\hbox{$_{{a}}$}*\frm{-}
**\dir{-};
(-4,-7)*{};
(0,0)*+\hbox{$_{{a}}$}*\frm{-}
**\dir{-};
(-7,-7)*{};
(0,0)*+\hbox{$_{{a}}$}*\frm{-}
**\dir{-};
(8,-7)*{};
(0,0)*+\hbox{$_{{a}}$}*\frm{-}
**\dir{-};
(4,-7)*{};
(0,0)*+\hbox{$_{{a}}$}*\frm{-}
**\dir{-};
\endxy}\Ea:=
\Ba{c}
\resizebox{18mm}{!}{ \xy
(-7.5,-8.6)*{_{_1}};
(-4.1,-8.6)*{_{_2}};
(4.3,-8.6)*{_{_{k\hspace{-0.3mm}-\hspace{-0.3mm}2}}};
(6,-14.6)*{_{_{k\hspace{-0.3mm}-\hspace{-0.3mm}1}}};
(13.9,-14.6)*{_{_{k}}};
(0.0,-6)*{...};
(0,5)*{};
(0,0)*+\hbox{$_{{a}}$}*\frm{-}
**\dir{-};
(-4,-7)*{};
(0,0)*+\hbox{$_{{a}}$}*\frm{-}
**\dir{-};
(-7,-7)*{};
(0,0)*+\hbox{$_{{a}}$}*\frm{-}
**\dir{-};
(9,-7)*{\circ};
(0,0)*+\hbox{$_{{a}}$}*\frm{-}
**\dir{-};
(4,-7)*{};
(0,0)*+\hbox{$_{{a}}$}*\frm{-}
**\dir{-};
(9,-7)*{\circ};
 <9.3mm,-7.3mm>*{};<13mm,-13mm>*{}**@{-},
 <8.7mm,-7.3mm>*{};<6mm,-13mm>*{}**@{-},
\endxy}
\Ea
-
\Ba{c}
\resizebox{18mm}{!}{ \xy
(-7.5,-8.6)*{_{_1}};
(-4.1,-8.6)*{_{_2}};
(4.3,-8.6)*{_{_{k\hspace{-0.3mm}-\hspace{-0.3mm}2}}};
(11,-8.6)*{_{_{k\hspace{-0.3mm}-\hspace{-0.3mm}1}}};
(11,-1)*{_{_{k}}};
(0.0,-6)*{...};
(5,7)*{\circ};
(0,0)*+\hbox{$_{{a}}$}*\frm{-}
**\dir{-};
(-4,-7)*{};
(0,0)*+\hbox{$_{{a}}$}*\frm{-}
**\dir{-};
(-7,-7)*{};
(0,0)*+\hbox{$_{{a}}$}*\frm{-}
**\dir{-};
(9,-7)*{};
(0,0)*+\hbox{$_{{a}}$}*\frm{-}
**\dir{-};
(4,-7)*{};
(0,0)*+\hbox{$_{{a}}$}*\frm{-}
**\dir{-};
(9,-7)*{};
 <5.0mm,7.5mm>*{};<5.0mm,13.9mm>*{}**@{-},
 <5.4mm,6.6mm>*{};<10mm,1mm>*{}**@{-},
\endxy}
\Ea
-
\Ba{c}
\resizebox{18mm}{!}{ \xy
(-7.5,-8.6)*{_{_1}};
(-4.1,-8.6)*{_{_2}};
(4.3,-8.6)*{_{_{k\hspace{-0.3mm}-\hspace{-0.3mm}2}}};
(11,-1)*{_{_{k\hspace{-0.3mm}-\hspace{-0.3mm}1}}};
(10,-8.6)*{_{_{k}}};
(0.0,-6)*{...};
(5,7)*{\circ};
(0,0)*+\hbox{$_{{a}}$}*\frm{-}
**\dir{-};
(-4,-7)*{};
(0,0)*+\hbox{$_{{a}}$}*\frm{-}
**\dir{-};
(-7,-7)*{};
(0,0)*+\hbox{$_{{a}}$}*\frm{-}
**\dir{-};
(9,-7)*{};
(0,0)*+\hbox{$_{{a}}$}*\frm{-}
**\dir{-};
(4,-7)*{};
(0,0)*+\hbox{$_{{a}}$}*\frm{-}
**\dir{-};
(9,-7)*{};
 <5.0mm,7.5mm>*{};<5.0mm,13.9mm>*{}**@{-},
 <5.4mm,6.6mm>*{};<10mm,1mm>*{}**@{-},
\endxy}
\Ea
\Eeq
of the operad $\cB\cV_\infty^{com}$. If $\rho: \cB\cV_\infty^{com}\rar \cE nd_V$ is a representation,
then, in the notation of \S{\ref{5: subsec on order of operators}},
$$
\rho\left(\Ba{c}\resizebox{14mm}{!}{ \xy
(-7.5,-8.6)*{_{_1}};
(-4.1,-8.6)*{_{_2}};
(4.5,-8.6)*{_{_{k\hspace{-0.3mm}-\hspace{-0.3mm}1}}};
(9.0,-8.5)*{_{_{k}}};
(0.0,-6)*{...};
(0,5)*{};
(0,0)*+\hbox{$_{{a}}$}*\frm{-}
**\dir{-};
(-4,-7)*{};
(0,0)*+\hbox{$_{{a}}$}*\frm{-}
**\dir{-};
(-7,-7)*{};
(0,0)*+\hbox{$_{{a}}$}*\frm{-}
**\dir{-};
(8,-7)*{};
(0,0)*+\hbox{$_{{a}}$}*\frm{-}
**\dir{-};
(4,-7)*{};
(0,0)*+\hbox{$_{{a}}$}*\frm{-}
**\dir{-};
\endxy}\Ea\right)= F_k^{\Delta_a}.
$$
Note that $F_k^{\Delta_a}$ identically vanishes for $k\geq a+2$ as the operator $\Delta_a$ is, by its definition, of order $\leq a+1$; this is the reason why we defined the above elements of $\cB\cV_\infty^{com}$ only in the range
$1\leq k\leq a+1$: for all other $k$ these elements vanish identically due to the relations between the generators of $\cB\cV_\infty^{com}$.

\mip

Consider next a free operad, $\caL ^\diamond_{\infty}$, generated, for all integers $p\geq 0$, $k\geq 1$ with
 $p+k\geq 2$, by the symmetric corollas
$$
\Ba{c}\resizebox{14mm}{!}{ \xy
(-7.5,-8.6)*{_{_1}};
(-4.1,-8.6)*{_{_2}};
(9.0,-8.5)*{_{_{k}}};
(0.0,-6)*{...};
(0,5)*{};
(0,0)*+\hbox{$_{{p}}$}*\frm{o}
**\dir{-};
(-4,-7)*{};
(0,0)*+\hbox{$_{{p}}$}*\frm{o}
**\dir{-};
(-7,-7)*{};
(0,0)*+\hbox{$_{{p}}$}*\frm{o}
**\dir{-};
(8,-7)*{};
(0,0)*+\hbox{$_{{p}}$}*\frm{o}
**\dir{-};
(4,-7)*{};
(0,0)*+\hbox{$_{{p}}$}*\frm{o}
**\dir{-};
\endxy}\Ea=
\Ba{c}\resizebox{16.2mm}{!}{ \xy
(-8.5,-8.6)*{_{_{\sigma(1)}}};
(-3.1,-8.6)*{_{_{\sigma(2)}}};
(9.0,-8.5)*{_{_{\sigma(k)}}};
(0.0,-6)*{...};
(0,5)*{};
(0,0)*+\hbox{$_{{p}}$}*\frm{o}
**\dir{-};
(-4,-7)*{};
(0,0)*+\hbox{$_{{p}}$}*\frm{o}
**\dir{-};
(-7,-7)*{};
(0,0)*+\hbox{$_{{p}}$}*\frm{o}
**\dir{-};
(8,-7)*{};
(0,0)*+\hbox{$_{{p}}$}*\frm{o}
**\dir{-};
(4,-7)*{};
(0,0)*+\hbox{$_{{p}}$}*\frm{o}
**\dir{-};
\endxy}\Ea\ \ \ \forall\ \sigma\in \bS_n,
$$
of homological degree $3-2k-2p$, and equipped with the following differential
$$
d\Ba{c}
\resizebox{14mm}{!}{ \xy
(-7.5,-8.6)*{_{_1}};
(-4.1,-8.6)*{_{_2}};
(9.0,-8.5)*{_{_{k}}};
(0.0,-6)*{...};
(0,5)*{};
(0,0)*+\hbox{$_{{p}}$}*\frm{o}
**\dir{-};
(-4,-7)*{};
(0,0)*+\hbox{$_{{p}}$}*\frm{o}
**\dir{-};
(-7,-7)*{};
(0,0)*+\hbox{$_{{p}}$}*\frm{o}
**\dir{-};
(8,-7)*{};
(0,0)*+\hbox{$_{{p}}$}*\frm{o}
**\dir{-};
(4,-7)*{};
(0,0)*+\hbox{$_{{p}}$}*\frm{o}
**\dir{-};
\endxy}\Ea
=
\sum_{p=q+r\atop [k]=I_1\sqcup I_2}
\Ba{c}
%
%
\resizebox{19mm}{!}{ \xy
(0,0)*+{q}*\cir{}="b",
(10,10)*+{r}*\cir{}="c",
%
(-4,-6)*{}="-1",
(-2,-6)*{}="-2",
(4,-6)*{}="-3",
(1,-5)*{...},
(0,-8)*{\underbrace{\ \ \ \ \ \ \ \ }},
(0,-11)*{_{I_1}},
(10,16)*{}="2'",
(11,4)*{}="-1'",
(16,4)*{}="-2'",
(18,4)*{}="-3'",
(13.5,4)*{...},
(15,2)*{\underbrace{\ \ \ \ \ \ \ }},
(15,-1)*{_{I_2}},
%
\ar @{-} "b";"c" <0pt>
\ar @{-} "b";"-1" <0pt>
\ar @{-} "b";"-2" <0pt>
\ar @{-} "b";"-3" <0pt>
%
\ar @{-} "c";"2'" <0pt>
\ar @{-} "c";"-1'" <0pt>
\ar @{-} "c";"-2'" <0pt>
\ar @{-} "c";"-3'" <0pt>
\endxy}
\Ea
$$
Representations, $\rho: \caL^\diamond_\infty\rar \cE nd_V$, of this operad in a dg vector space $(V,d)$
is the same thing as continuous representations of  the  operad $\caL ie_\infty\{1\}[[\hbar]]$
in the topological vector space $V[[\hbar]]$ equipped with the differential
$$
d + \sum_{p\geq 1}\hbar^p \Delta_p, \ \ \ \Delta_p:=\rho \left(\Ba{c}
\resizebox{4mm}{!}{ \xy
(0,5)*{};
(0,0)*+{_a}*\cir{}
**\dir{-};
(0,-5)*{};
(0,0)*+{_a}*\cir{}
**\dir{-};
\endxy}\Ea\right).
$$
where the formal parameter $\hbar$ is assumed to have homological degree $2$.

\begin{proposition}\label{5: Propos on cohom of L diamnd infty}  { The cohomology of the dg operad $\caL^\diamond_\infty$ is the operad $\caL^\diamond$ defined in \S {\ref{5: subsect on L-diamond and BV}}, i.e.\
$\caL^\diamond_\infty$ is a minimal resolution of $\caL^\diamond$.
}
\end{proposition}
\begin{proof}
The dg operad $\caL_\infty^\diamond\{1\}$ is a direct summand of the graded properad $gr\LoB_\infty$
associated with the genus filtration of the properad $\LoB_\infty$ .
Hence the required result follows from the proof of  Proposition~{\ref{2: proposition on nu quasi-iso}}.
\end{proof}

\mip
We are interested in the operad $\caL_\infty^\diamond$ because of the following property.

\begin{lemma}\label{5: monomorphism chi} { There is a monomorphism of dg operads,
$$
\chi: \caL_\infty^\diamond \lon \cB\cV_\infty^{com}
$$
given on generators as follows,
$$
\chi\left(\Ba{c}
\resizebox{12mm}{!}{ \xy
(-7.5,-7.6)*{_{_1}};
(-4.1,-7.6)*{_{_2}};
(4.5,-7.6)*{_{_{k\hspace{-0.3mm}-\hspace{-0.3mm}1}}};
(9.0,-7.5)*{_{_{k}}};
(0.0,-5)*{...};
(0,5)*{};
(0,0)*+\hbox{$_{{p}}$}*\frm{o}
**\dir{-};
(-4,-6)*{};
(0,0)*+\hbox{$_{{p}}$}*\frm{o}
**\dir{-};
(-7,-6)*{};
(0,0)*+\hbox{$_{{p}}$}*\frm{o}
**\dir{-};
(8,-6)*{};
(0,0)*+\hbox{$_{{p}}$}*\frm{o}
**\dir{-};
(4,-6)*{};
(0,0)*+\hbox{$_{{p}}$}*\frm{o}
**\dir{-};
\endxy}\Ea\right)
=
\Ba{c}\resizebox{12mm}{!}{
\xy
(-7.5,-7.6)*{_{_1}};
(-4.1,-7.6)*{_{_2}};
(4.5,-7.6)*{_{_{k\hspace{-0.3mm}-\hspace{-0.3mm}1}}};
(9.0,-7.5)*{_{_{k}}};
(0.0,-5)*{...};
(0,5)*{};
(0,0)*+\hbox{$_{{p+k-1}}$}*\frm{-}
**\dir{-};
(-4,-6)*{};
(0,0)*+\hbox{$_{{p+k-1}}$}*\frm{-}
**\dir{-};
(-7,-6)*{};
(0,0)*+\hbox{$_{{p+k-1}}$}*\frm{-}
**\dir{-};
(8,-6)*{};
(0,0)*+\hbox{$_{{p+k-1}}$}*\frm{-}
**\dir{-};
(4,-6)*{};
(0,0)*+\hbox{$_{{p+k-1}}$}*\frm{-}
**\dir{-};
\endxy}\Ea
$$
}
\end{lemma}
\begin{proof} For notation reason, we prove the proposition in terms of representations: for any representation
$\rho: \cB\cV_\infty^{com}\rar \cE nd_V$ we construct an associated representation $\rho':  \caL_\infty^\diamond\rar \cE nd_V$ such that $\rho'=\rho\circ \chi$.

\sip

Let $\{\Delta_a: V\rar V[1-2a] , \mu:\odot^2 V\rar V\}_{a\geq 1}$ be a $\cB\cV_\infty^{com}$-structure
in a dg vector space $(V,d)$. Then
$$
\Delta:= d + \sum_{a\geq 1}\hbar^a \Delta_{a}
$$
is a degree 1 differential in the graded vector space $V[[\hbar]]$, $\hbar$ being a formal parameter of homological degree $2$.  As explained in \S {\ref{5: subsec on order of operators}}, this differential makes
the graded commutative algebra $(V[[\hbar]], \mu)$ into a $\caL ie_\infty\{1\}[[\hbar]]$ algebra over the ring
$\K[[\hbar]]$, with higher Lie brackets given by,
$$
F^\Delta_k:= \frac{1}{(k-1)!}\underbrace{[\ldots[[\Delta,\mu],\mu], \ldots, \mu]}_{k-1\ \mathrm{brackets}}=
\frac{\hbar^{k-1}}{(k-1)!}\sum_{p=0}^\infty\hbar^p
\underbrace{[\ldots[[\Delta_{p+k-1},\mu],\mu], \ldots, \mu]}_{k-1\ \mathrm{brackets}}
$$
It is well-known that $\caL ie_\infty$-algebra structures are homogeneous in the sense that if $\{\mu_n\}_{n\geq 1}$ is a $\caL ie_\infty$-algebra structure in some vector space, then, for any $\la\in \R$, the collection
 $\{\la^{n-1}\mu_n\}_{n\geq 1}$ is again a  $\caL ie_\infty$-algebra structure in the same space. Therefore, the rescaled collection of operations,
 $$
 \hat{F}^\Delta_k:=\frac{1}{(k-1)!}\sum_{p=0}^\infty\hbar^p
\underbrace{[\ldots[[\Delta_{p+k-1},\mu],\mu], \ldots, \mu]}_{k-1\ \mathrm{brackets}}
 $$
 also defines a continuous representation of $\caL ie_\infty\{1\}[[\hbar]]$ in the dg space $(V[[\hbar]],\Delta)$, and hence a representation of $\caL_\infty^\diamond$ in the space $(V,d)$ given on the generators as follows,
 $$
 \rho'\left(\Ba{c}
\resizebox{12mm}{!}{ \xy
(-7.5,-7.6)*{_{_1}};
(-4.1,-7.6)*{_{_2}};
(4.5,-7.6)*{_{_{k\hspace{-0.3mm}-\hspace{-0.3mm}1}}};
(9.0,-7.5)*{_{_{k}}};
(0.0,-5)*{...};
(0,5)*{};
(0,0)*+\hbox{$_{{p}}$}*\frm{o}
**\dir{-};
(-4,-6)*{};
(0,0)*+\hbox{$_{{p}}$}*\frm{o}
**\dir{-};
(-7,-6)*{};
(0,0)*+\hbox{$_{{p}}$}*\frm{o}
**\dir{-};
(8,-6)*{};
(0,0)*+\hbox{$_{{p}}$}*\frm{o}
**\dir{-};
(4,-6)*{};
(0,0)*+\hbox{$_{{p}}$}*\frm{o}
**\dir{-};
\endxy}\Ea\right)
= \frac{1}{(k-1)!}\underbrace{[\ldots[[\Delta_{p+k-1},\mu],\mu], \ldots, \mu]}_{k-1\ \mathrm{brackets}}
 $$
As
$
\rho\left(
\Ba{c}\resizebox{12mm}{!}{ \xy
(-7.5,-7.6)*{_{_1}};
(-4.1,-7.6)*{_{_2}};
(4.5,-7.6)*{_{_{k\hspace{-0.3mm}-\hspace{-0.3mm}1}}};
(9.0,-7.5)*{_{_{k}}};
(0.0,-5)*{...};
(0,5)*{};
(0,0)*+\hbox{$_{{p+k-1}}$}*\frm{-}
**\dir{-};
(-4,-6)*{};
(0,0)*+\hbox{$_{{p+k-1}}$}*\frm{-}
**\dir{-};
(-7,-6)*{};
(0,0)*+\hbox{$_{{p+k-1}}$}*\frm{-}
**\dir{-};
(8,-6)*{};
(0,0)*+\hbox{$_{{p+k-1}}$}*\frm{-}
**\dir{-};
(4,-6)*{};
(0,0)*+\hbox{$_{{p+k-1}}$}*\frm{-}
**\dir{-};
\endxy}\Ea\right)$ equals $\frac{1}{(k-1)!} \underbrace{[\ldots[[\Delta_{p+k-1},\mu],\mu], \ldots, \mu]}_{k-1\ \mathrm{brackets}}
$
as well, the proof is completed.
\end{proof}


Finally we can give the proof of Theorem \ref{5: Theorem on cohom BV_com}.

\begin{proof}[Proof of Theorem \ref{5: Theorem on cohom BV_com}] The map $\pi$ induces obviously a morphism of operads,
$$
[\pi]: H^\bu(\BV_\infty^{com})\lon \cB\cV.
$$
Therefore to prove the theorem it is enough to show that $[\pi]$ induces an isomorphism of $\bS$-modules.
\sip

Denote  the following (equivalence class of a) graph in $\cB\cV_\infty^{com}$ by
$$
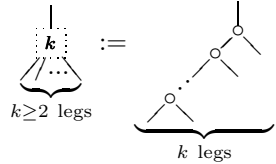
\underbrace{\xy
(1,-4)*{...};
(0,5)*{};
(0,0)*+\hbox{$_k$}*\frm{.}
**\dir{-};
(-2,-5)*{};
(0,0)*+\hbox{$_k$}*\frm{.}
**\dir{-};
(-4,-5)*{};
(0,0)*+\hbox{$_k$}*\frm{.}
**\dir{-};
(4,-5)*{};
(0,0)*+\hbox{$_k$}*\frm{.}
**\dir{-};
\endxy}_{k\geq 2\ \mathrm{legs}}:=
\underbrace{
\Ba{c}
\xy
(0,-1)*+{}="0",
(0,-6)*{\circ}="-1",
(-3,-9)*{\circ}="-2",
(3,-9)*{}="-2'",
(-6,-12)*{}="-3",
(-0,-12)*{}="-3'",
(-6.7,-12.7)*{\cdot};
(-7.7,-13.7)*{\cdot};
(-9,-15)*{\circ}="-4",
(-12,-18)*{}="-4'",
(-6,-18)*{}="-4''",
%
\ar @{-} "-1";"0" <0pt>
\ar @{-} "-1";"-2" <0pt>
\ar @{-} "-1";"-2'" <0pt>
\ar @{-} "-2";"-3" <0pt>
\ar @{-} "-2";"-3'" <0pt>
\ar @{-} "-4";"-4'" <0pt>
\ar @{-} "-4";"-4''" <0pt>
\endxy
\Ea}_{k\ \mathrm{legs}}
$$
and call it a {\em dashed square vertex}. Consider an operad, $\f$, freely  generated by
 the operad $\cC om$ and  a countable family of unary operations,
$
\left\{ \Ba{c}\resizebox{3.1mm}{!}{  \xy
(0,5)*{};
(0,0)*+{_a}*\cir{}
**\dir{-};
(0,-5)*{};
(0,0)*+{_a}*\cir{}
**\dir{-};
\endxy}\Ea \right\}_{a\geq 1}
$
of homological degree $1-2a$ (so that $\cB\cV_\infty^{com}$ is a quotient of $\f$ by the ideal
 encoding the requirement that each unary operation $\Ba{c}\resizebox{2.9mm}{!}{\xy
(0,5)*{};
(0,0)*+{_a}*\cir{}
**\dir{-};
(0,-5)*{};
(0,0)*+{_a}*\cir{}
**\dir{-};
\endxy}\Ea$  is of order $\leq a+1$ with
respect to the multiplication operation).
Let $t^{(1)}$ be  any tree built from the following ``corollas",
$$
\Ba{c}
\resizebox{9mm}{!}{
\xy
(-4,-8)*{_{_1}};
(-2,-8)*{_{_2}};
(4.9,-8)*{_{_{k}}};
(0,0)*++{_a}*\frm<8pt,8pt>{ee}="b",
(-4,-6)*{}="-1",
(-2,-6)*{}="-2",
(4,-6)*{}="-3",
(1,-5)*{...},
(0,7)*{}="1'",
\ar @{-} "b";"1'" <0pt>
\ar @{-} "b";"-1" <0pt>
\ar @{-} "b";"-2" <0pt>
\ar @{-} "b";"-3" <0pt>
\endxy}
\Ea
:=
\Ba{c}
\resizebox{6mm}{!}{ \xy
(0,0)*+{k}*\frm{.}="b",
(0,7)*+{a}*\cir{}="c",
%
(-4,-6)*{}="-1",
(-2,-6)*{}="-2",
(4,-6)*{}="-3",
(1,-5)*{...},
(0,13)*{}="1'",
%
\ar @{-} "b";"c" <0pt>
\ar @{-} "b";"-1" <0pt>
\ar @{-} "b";"-2" <0pt>
\ar @{-} "b";"-3" <0pt>
\ar @{-} "c";"1'" <0pt>
\endxy}\Ea\ \mbox{with}\ \ k\geq 1, a\geq 1
$$
(where we  assume implicitly that for $k=1$ the l.h.s.\ corolla equals  $\Ba{c}
{\resizebox{3.0mm}{!}{ \xy
(0,5)*{};
(0,0)*+{_a}*\cir{}
**\dir{-};
(0,-5)*{};
(0,0)*+{_a}*\cir{}
**\dir{-};
\endxy}}\Ea$) and let $t^{(2)}$ be any graph obtained by attaching to one or more (or none) input leg of a dashed
square vertex a tree of the type $t^{(1)}$, e.g.
$$
t^{(2)}=
\Ba{c}
\resizebox{16mm}{!}{ \xy
(0,0)*+{_{t_1^{(1)}}}="b1",
(7,0)*+{_{t_2^{(1)}}}="b2",
(14,2)*+{}="b3",
(22,0)*+{_{t_k^{(1)}}}="b4",
(10,10)*+{k}*\frm{.}="c",
(10,18)*+{}="u",
%
\ar @{-} "u";"c" <0pt>
\ar @{-} "b1";"c" <0pt>
\ar @{-} "b2";"c" <0pt>
\ar @{-} "b3";"c" <0pt>
\ar @{-} "b4";"c" <0pt>
\endxy}
\Ea
$$
The family $\{t^{(1)}, t^{(2)}\}$ forms a basis of $\f$ as an $\bS$-module.

\sip

Define for any $a,k\geq 1$ a linear combination,
\Beq\label{C: from T to T}
\Ba{c}
\resizebox{13mm}{!}{ \xy
(-7.5,-8.6)*{_{_1}};
(-4.1,-8.6)*{_{_2}};
(4.5,-8.6)*{_{_{k\hspace{-0.3mm}-\hspace{-0.3mm}1}}};
(9.0,-8.5)*{_{_{k}}};
(0.0,-6)*{...};
(0,5)*{};
(0,0)*+\hbox{$_{{a}}$}*\frm{-}
**\dir{-};
(-4,-7)*{};
(0,0)*+\hbox{$_{{a}}$}*\frm{-}
**\dir{-};
(-7,-7)*{};
(0,0)*+\hbox{$_{{a}}$}*\frm{-}
**\dir{-};
(8,-7)*{};
(0,0)*+\hbox{$_{{a}}$}*\frm{-}
**\dir{-};
(4,-7)*{};
(0,0)*+\hbox{$_{{a}}$}*\frm{-}
**\dir{-};
\endxy}\Ea
=
\Ba{c}
\resizebox{9mm}{!}{
\xy
(-4,-8)*{_{_1}};
(-2,-8)*{_{_2}};
(4.9,-8)*{_{_{k}}};
(0,0)*++{_a}*\frm<8pt,8pt>{ee}="b",
(-4,-6)*{}="-1",
(-2,-6)*{}="-2",
(4,-6)*{}="-3",
(1,-5)*{...},
(0,7)*{}="1'",
\ar @{-} "b";"1'" <0pt>
\ar @{-} "b";"-1" <0pt>
\ar @{-} "b";"-2" <0pt>
\ar @{-} "b";"-3" <0pt>
\endxy}
\Ea
-\sum_{\sigma\in \bS_k}\frac{1}{(k-1)!}\hspace{-7mm}
\Ba{c}
\resizebox{19mm}{!}{ \xy
(-8.5,-8.6)*{_{_{\sigma(1)}}};
(-3.1,-8.6)*{_{_{\sigma(2)}}};
(11,-8.6)*{_{_{\sigma(k\hspace{-0.3mm}-\hspace{-0.3mm}1)}}};
(11,-1)*{_{_{\sigma(k)}}};
(0.0,-6)*{...};
(5,7)*{\circ};
(0,0)*++\hbox{$_{{a}}$}*\frm<8pt,8pt>{ee}
**\dir{-};
(-4,-7)*{};
(0,0)*++\hbox{$_{{}}$}*\frm<8pt,8pt>{}
**\dir{-};
(-7,-7)*{};
(0,0)*++\hbox{$_{{}}$}*\frm<8pt,8pt>{}
**\dir{-};
(9,-7)*{};
(0,0)*++\hbox{$_{{}}$}*\frm<8pt,8pt>{}
**\dir{-};
(4,-7)*{};
(0,0)*++\hbox{$_{{}}$}*\frm<8pt,8pt>{}
**\dir{-};
(9,-7)*{};
 <5.0mm,7.5mm>*{};<5.0mm,13.9mm>*{}**@{-},
 <5.4mm,6.6mm>*{};<10mm,1mm>*{}**@{-},
\endxy}
\Ea
+
\sum_{\sigma\in \bS_k}\frac{1}{2!(k-2)!}\hspace{-3mm}
\Ba{c}
\resizebox{19mm}{!}{ \xy
(-7.5,-8.6)*{_{_1}};
(-4.1,-8.6)*{_{_2}};
(7,-1)*{_{_{\sigma(k\hspace{-0.3mm}-\hspace{-0.3mm}1)}}};
(10,-8.6)*{_{_{\sigma(k\hspace{-0.3mm}-\hspace{-0.3mm}2)}}};
(14,-1)*{_{_{\sigma(k)}}};
(0.0,-6)*{...};
(6,7)*{\circ};
(0,0)*++\hbox{$_{{a}}$}*\frm<8pt,8pt>{ee}
**\dir{-};
(-4,-7)*{};
(0,0)*++\hbox{$_{{}}$}*\frm<8pt,8pt>{}
**\dir{-};
(-7,-7)*{};
(0,0)*++\hbox{$_{{}}$}*\frm<8pt,8pt>{}
**\dir{-};
(9,-7)*{};
(0,0)*++\hbox{$_{{}}$}*\frm<8pt,8pt>{}
**\dir{-};
(4,-7)*{};
(0,0)*++\hbox{$_{{}}$}*\frm<8pt,8pt>{}
**\dir{-};
(9,-7)*{};
 <6.0mm,7.5mm>*{};<6.0mm,13.9mm>*{}**@{-},
 <6.0mm,6.6mm>*{};<6.0mm,1mm>*{}**@{-},
 <6.4mm,6.6mm>*{};<11mm,1mm>*{}**@{-},
\endxy}
\Ea
+ \ldots,
\Eeq

Consider next (i) a set  $\{T^{(1)}\}$  of all possible trees generated
by these ``square" corollas, e.g.
%
%
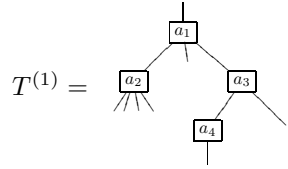
$$
T^{(1)}=\Ba{c}
\resizebox{25mm}{!}{ \xy
(0,0)*+{a_2}*\frm{-}="b",
(10,10)*+{a_1}*\frm{-}="c",
(22,0)*+{a_3}*\frm{-}="r",
(15,-10)*+{a_4}*\frm{-}="r1",
(32,-10)*+{}="r2",
(15,-18)*+{}="d",
(-4,-6)*{}="-1",
(-2,-6)*{}="-2",
(4,-6)*{}="-3",
(1,-6)*{}="-4",
(10,16)*{}="2'",
(11,4)*{}="-1'",
(16,4)*{}="-2'",
(18,4)*{}="-3'",
%
\ar @{-} "b";"c" <0pt>
\ar @{-} "r";"c" <0pt>
\ar @{-} "r";"r1" <0pt>
\ar @{-} "r";"r2" <0pt>
\ar @{-} "r1";"d" <0pt>
\ar @{-} "b";"-1" <0pt>
\ar @{-} "b";"-2" <0pt>
\ar @{-} "b";"-3" <0pt>
\ar @{-} "b";"-4" <0pt>
\ar @{-} "c";"2'" <0pt>
\ar @{-} "c";"-1'" <0pt>
\endxy}\Ea
$$
and also (ii) a set  $\{T^{(2)}\}$ of all possible trees  obtained by attaching to (some) legs of a dashed square vertex trees from the set $\{T^{(1)}\}$, e.g.
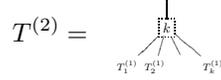
$$
T^{(2)}=
\Ba{c}
\resizebox{16mm}{!}{ \xy
(0,0)*+{_{T_1^{(1)}}}="b1",
(7,0)*+{_{T_2^{(1)}}}="b2",
(14,2)*+{}="b3",
(22,0)*+{_{T_k^{(1)}}}="b4",
(10,10)*+{k}*\frm{.}="c",
(10,18)*+{}="u",
%
\ar @{-} "u";"c" <0pt>
\ar @{-} "b1";"c" <0pt>
\ar @{-} "b2";"c" <0pt>
\ar @{-} "b3";"c" <0pt>
\ar @{-} "b4";"c" <0pt>
\endxy}
\Ea
$$
Formulae (\ref{C: from T to T}) define a natural linear map
$$
\phi: \mathrm{span}\langle T^{(1)}, T^{(2)}\rangle \lon  \mathrm{span}\langle t^{(1)}, t^{(2)} \rangle =
\f.
$$
The expressions  (\ref{C: from T to T}) can be (inductively) inverted,
\Beq\label{C: from t to T}
\Ba{c}
\resizebox{9mm}{!}{
\xy
(-4,-8)*{_{_1}};
(-2,-8)*{_{_2}};
(4.9,-8)*{_{_{k}}};
(0,0)*++{_a}*\frm<8pt,8pt>{ee}="b",
(-4,-6)*{}="-1",
(-2,-6)*{}="-2",
(4,-6)*{}="-3",
(1,-5)*{...},
(0,7)*{}="1'",
\ar @{-} "b";"1'" <0pt>
\ar @{-} "b";"-1" <0pt>
\ar @{-} "b";"-2" <0pt>
\ar @{-} "b";"-3" <0pt>
\endxy}
\Ea
=
\Ba{c}
\resizebox{13mm}{!}{ \xy
(-7.5,-8.6)*{_{_1}};
(-4.1,-8.6)*{_{_2}};
(4.5,-8.6)*{_{_{k\hspace{-0.3mm}-\hspace{-0.3mm}1}}};
(9.0,-8.5)*{_{_{k}}};
(0.0,-6)*{...};
(0,5)*{};
(0,0)*+\hbox{$_{{a}}$}*\frm{-}
**\dir{-};
(-4,-7)*{};
(0,0)*+\hbox{$_{{a}}$}*\frm{-}
**\dir{-};
(-7,-7)*{};
(0,0)*+\hbox{$_{{a}}$}*\frm{-}
**\dir{-};
(8,-7)*{};
(0,0)*+\hbox{$_{{a}}$}*\frm{-}
**\dir{-};
(4,-7)*{};
(0,0)*+\hbox{$_{{a}}$}*\frm{-}
**\dir{-};
\endxy}\Ea
+\sum_{\sigma\in \bS_k}\frac{1}{(k-1)!}\hspace{-7mm}
\Ba{c}
\resizebox{19mm}{!}{ \xy
(-8.5,-8.6)*{_{_{\sigma(1)}}};
(-3.1,-8.6)*{_{_{\sigma(2)}}};
(11,-8.6)*{_{_{\sigma(k\hspace{-0.3mm}-\hspace{-0.3mm}1)}}};
(11,-1)*{_{_{\sigma(k)}}};
(0.0,-6)*{...};
(5,7)*{\circ};
(0,0)*+\hbox{$_{{a}}$}*\frm{-}
**\dir{-};
(-4,-7)*{};
(0,0)*+\hbox{$_{{a}}$}*\frm{-}
**\dir{-};
(-7,-7)*{};
(0,0)*+\hbox{$_{{a}}$}*\frm{-}
**\dir{-};
(9,-7)*{};
(0,0)*+\hbox{$_{{a}}$}*\frm{-}
**\dir{-};
(4,-7)*{};
(0,0)*+\hbox{$_{{a}}$}*\frm{-}
**\dir{-};
(9,-7)*{};
 <5.0mm,7.5mm>*{};<5.0mm,13.9mm>*{}**@{-},
 <5.4mm,6.6mm>*{};<10mm,1mm>*{}**@{-},
\endxy}
\Ea
-
\sum_{\sigma\in \bS_k}\frac{1}{2!(k-2)!}\hspace{-3mm}
\Ba{c}
\resizebox{19mm}{!}{ \xy
(-7.5,-8.6)*{_{_1}};
(-4.1,-8.6)*{_{_2}};
(7,-1)*{_{_{\sigma(k\hspace{-0.3mm}-\hspace{-0.3mm}1)}}};
(10,-8.6)*{_{_{\sigma(k\hspace{-0.3mm}-\hspace{-0.3mm}2)}}};
(14,-1)*{_{_{\sigma(k)}}};
(0.0,-6)*{...};
(6,7)*{\circ};
(0,0)*+\hbox{$_{{a}}$}*\frm{-}
**\dir{-};
(-4,-7)*{};
(0,0)*+\hbox{$_{{a}}$}*\frm{-}
**\dir{-};
(-7,-7)*{};
(0,0)*+\hbox{$_{{a}}$}*\frm{-}
**\dir{-};
(9,-7)*{};
(0,0)*+\hbox{$_{{a}}$}*\frm{-}
**\dir{-};
(4,-7)*{};
(0,0)*+\hbox{$_{{a}}$}*\frm{-}
**\dir{-};
(9,-7)*{};
 <6.0mm,7.5mm>*{};<6.0mm,13.9mm>*{}**@{-},
 <6.0mm,6.6mm>*{};<6.0mm,1mm>*{}**@{-},
 <6.4mm,6.6mm>*{};<11mm,1mm>*{}**@{-},
\endxy}
\Ea
+ \ldots,
\Eeq
and hence give us, again by induction, a linear map
$$
\psi: \mathrm{span}\langle t^{(1)}, t^{(2)}\rangle \lon  \mathrm{span}\langle T^{(1)}, T^{(2)} \rangle
$$
as follows. On $1$-vertex trees from the family  $\{ t^{(1)}, t^{(2)}\}$ the map $\psi$ is given
by
$$
\psi\left(\Ba{c}\xy
(1,-4)*{...};
(0,5)*{};
(0,0)*+\hbox{$_k$}*\frm{.}
**\dir{-};
(-2,-5)*{};
(0,0)*+\hbox{$_k$}*\frm{.}
**\dir{-};
(-4,-5)*{};
(0,0)*+\hbox{$_k$}*\frm{.}
**\dir{-};
(4,-5)*{};
(0,0)*+\hbox{$_k$}*\frm{.}
**\dir{-};
\endxy\Ea\right)= \Ba{c}\xy
(1,-4)*{...};
(0,5)*{};
(0,0)*+\hbox{$_k$}*\frm{.}
**\dir{-};
(-2,-5)*{};
(0,0)*+\hbox{$_k$}*\frm{.}
**\dir{-};
(-4,-5)*{};
(0,0)*+\hbox{$_k$}*\frm{.}
**\dir{-};
(4,-5)*{};
(0,0)*+\hbox{$_k$}*\frm{.}
**\dir{-};
\endxy\Ea, \ \ \ \ \ \
\psi\left(
\Ba{c}
\resizebox{8mm}{!}{
\xy
(-4,-8)*{_{_1}};
(-2,-8)*{_{_2}};
(4.9,-8)*{_{_{k}}};
(0,0)*++{_a}*\frm<8pt,8pt>{ee}="b",
(-4,-6)*{}="-1",
(-2,-6)*{}="-2",
(4,-6)*{}="-3",
(1,-5)*{...},
(0,7)*{}="1'",
\ar @{-} "b";"1'" <0pt>
\ar @{-} "b";"-1" <0pt>
\ar @{-} "b";"-2" <0pt>
\ar @{-} "b";"-3" <0pt>
\endxy}
\Ea
\right)=\mathrm{the\ r.h.s.\ of}\ (\ref{C: from t to T}).
$$
Assume that the map $\psi$ is constructed on $n$-vertex trees from the family  $\{ t^{(1)}, t^{(2)}\}$. Let $t$ be a tree with $n+1$-vertices. The complement to the root vertex of $t$
is a disjoint union of trees, $\{t'\}$, with at most $n-1$ vertices. To get $\psi(t)$ apply first  $\psi$ to the subtrees $t'$ to get a linear combination of trees, $\sum t''$, where each  $t''$ is  obtained by attaching to (some) input legs of a one-vertex tree $v$ from $\{ t^{(1)}, t^{(2)}\}$ an element of the set $\{ T^{(1)}, T^{(2)}\}$; finally, apply $\psi$ to the root vertex $v$ of each summand $ t''$.
By construction,  $\psi\circ \phi=\Id$ and $\phi\circ \psi=\Id$  so that  the map $\phi$ is an isomorphism,
$$
 \f\cong \mathrm{span}\langle T^{(1)}, T^{(2)} \rangle.
$$
The ideal $I$ in $\f$ defining the operad $\cB\cV_{\infty}^{com}$ takes now a very simple form --- this is an $\bS$-submodule of $\f$ spanned by trees
from the family $\{T^{(1)}, T^{(2)}\}$ which contain at least one ``bad" square vertex $\Ba{c}
\resizebox{13mm}{!}
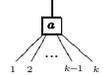
{ \xy
(-7.5,-8.6)*{_{_1}};
(-4.1,-8.6)*{_{_2}};
(4.5,-8.6)*{_{_{k\hspace{-0.3mm}-\hspace{-0.3mm}1}}};
(9.0,-8.5)*{_{_{k}}};
(0.0,-6)*{...};
(0,5)*{};
(0,0)*+\hbox{$_{{a}}$}*\frm{-}
**\dir{-};
(-4,-7)*{};
(0,0)*+\hbox{$_{{a}}$}*\frm{-}
**\dir{-};
(-7,-7)*{};
(0,0)*+\hbox{$_{{a}}$}*\frm{-}
**\dir{-};
(8,-7)*{};
(0,0)*+\hbox{$_{{a}}$}*\frm{-}
**\dir{-};
(4,-7)*{};
(0,0)*+\hbox{$_{{a}}$}*\frm{-}
**\dir{-};
\endxy}\Ea$ with $a<k-1$. If  $ \{T_{+}^{(1)}, T_{+}^{(2)}\}\subset  \{T^{(1)}, T^{(2)}\}$ is a subset of trees containing no bad vertices, then we can write an isomorphism of $\bS$-modules,
$$
\cB\cV_\infty^{com}\cong \mathrm{span}\langle T_{+}^{(1)}, T_{+}^{(2)}\rangle
$$
Therefore, as a complex $\cB\cV_\infty^{com}$ splits into a direct sum of subcomplexes,
$$
 \cB\cV_\infty^{com}\cong \mbox{span} \left\langle T_+^{(1)}  \right\rangle\ \oplus \ \mbox{span} \left\langle T_+^{(2)}  \right\rangle.
$$
The subcomplex $\mbox{span} \left\langle T_+^{(1)} \right\rangle$ is the image of the dg operad $\caL^\diamond_\infty$ under the monomorphism $\chi$ (see Lemma {\ref{5: monomorphism chi}}) so that we get eventually an isomorphism of complexes,
$$
\cB\cV_\infty^{com}\cong \cC om\circ \caL^\diamond_\infty,
$$
and the above splitting corresponds to the augmentation splitting of the operad $\cC om$,
$$
\cC om= \mbox{span}\left\langle 1\right\rangle \oplus \overline{\cC om}.
$$
Therefore, by Proposition {\ref{5: Propos on cohom of L diamnd infty}} and Remark {\ref{5: remark on qBV}},
$$
H^\bu(\cB\cV_\infty^{com})\cong \cC om \circ \caL^\diamond\cong q\cB\cV.
$$
By Proposition {\ref{5: Propos on qBV and BV}}, we get isomorphisms of $\bS$-modules,
$$
H^\bu(\cB\cV_\infty^{com})\cong Gr(\cB\cV)\cong \cB\cV
$$
which completes the proof of the Theorem.
\end{proof}

\bip

\def\cprime{$'$}

\end{document}